\documentclass[preprint,11pt]{elsarticle}
\usepackage[left=20mm,right=20mm,top=20mm,bottom=20mm]{geometry}
\usepackage{amssymb}
\usepackage{amsthm}
\usepackage{amsmath}
\usepackage{amsfonts}
\usepackage{hyperref} 
\usepackage{cleveref}
\usepackage{graphicx} 
\usepackage{mathtools}
\usepackage{enumitem} 
\usepackage{listings}         
\usepackage{bm}
\usepackage{color}
\usepackage{xcolor}
\usepackage{diagbox}   
\usepackage{booktabs}
\usepackage{subcaption} 
\usepackage{multirow}
\usepackage{caption}
\usepackage{algorithm}
\usepackage{algorithmic}

\newtheorem{theorem}{Theorem}[section]
\newtheorem{lemma}{Lemma}[section]

\begin{document}

\title{Dynamical low-rank approximation for the semiclassical Schr\"{o}dinger equation with uncertainties}

\author[1]{Liu Liu\corref{cor1}}
\ead{lliu@math.cuhk.edu.hk}

\author[2]{Limin Xu}
\ead{xulimin@westlake.edu.cn}

\author[3]{Zhenyi Zhu}
\ead{zyzhu@math.cuhk.edu.hk}

\cortext[cor1]{Corresponding author}

\address[1]{Department of Mathematics, The Chinese University of Hong Kong, Hong Kong, 999077, China}

\address[2]{Institute for Theoretical Sciences, Westlake University, Hangzhou, 310030, China}
            
\address[3]{Department of Mathematics, The Chinese University of Hong Kong, Hong Kong, 999077, China}

\begin{abstract}
In this paper, we propose a dynamical low-rank (DLR) approximation framework for solving the semiclassical Schrödinger equation with uncertainties. The primary numerical challenges arise from the dual nature of the oscillations: the spatial oscillations inherent in the semiclassical limit and the high-frequency oscillations in the random space induced by uncertainties. We extend two robust integrators—the projector-splitting integrator and the unconventional integrator—to the semiclassical regime to evolve the solution on a low-rank manifold. Through extensive numerical experiments, we demonstrate that the DLR method is significantly more computationally efficient than the standard stochastic Galerkin method, as it captures the essential quantum dynamics using a much smaller number of basis functions. Our findings reveal that despite the complex oscillatory patterns of the wave function, its evolution remains concentrated in a low-rank subspace for the cases investigated. Specifically, we observe that the DLR method achieves high fidelity with a remarkably small numerical rank, which remains robust even as the semiclassical parameter $\varepsilon$ decreases. Within our problem settings, the results further suggest that the rank growth is primarily driven by the randomness and regularity of the potential. These results provide practical insights into the low-rank structure of uncertain quantum systems and offer an efficient approach for high-dimensional uncertainty quantification in the semiclassical regime.
\end{abstract}


\maketitle

\section{Introduction}

In this article, we consider the semiclassical Schrödinger equation with uncertainties, formulated as follows:
\begin{equation}
\label{Schro-equation}
\begin{cases}
\mathrm{i} \varepsilon \partial_t \psi^{\varepsilon}(t, \bm{x}, \bm{\xi}) = -\frac{\varepsilon^{2}}{2} \Delta_{\bm{x}} \psi^{\varepsilon}(t, \bm{x}, \bm{\xi}) + V(\bm{x}, \bm{\xi}) \psi^{\varepsilon}(t, \bm{x}, \bm{\xi}), \\
\psi^{\varepsilon}(0, \bm{x}, \bm{\xi}) = \psi^{\varepsilon}_{0}(\bm{x}, \bm{\xi}), 
\end{cases}
\end{equation}
where $\bm{x} \in \mathbb{R}^m$ and $\bm{\xi} \in \mathbb{R}^n$ denote the spatial and random variables, respectively. Here, $\varepsilon$ is the semiclassical parameter ($0 < \varepsilon \ll 1$), $\psi^{\varepsilon}$ is the wave function, $V$ is the potential, and $\pi(\bm{\xi})$ is the probability density function of the random variable $\bm{\xi}$. 

In the deterministic setting (where $V$ and $\psi_0$ are independent of $\bm{\xi}$), the semiclassical Schrödinger equation describes the quantum dynamics of nuclei on potential energy surfaces under the Born-Oppenheimer approximation 
\cite{lasserComputingQuantumDynamics2020}. 
It serves as a cornerstone in fields such as quantum chemistry and solid-state physics. It is well known that numerically solving the deterministic equation already presents two formidable challenges: (i) the high dimensionality of the spatial variable $\bm{x}$, which leads to the curse of dimensionality , and (ii) the highly oscillatory nature of the wave function with wavelength $\mathcal{O}(\varepsilon)$, which imposes stringent requirements on numerical discretization \cite{jinMathematicalComputationalMethods2011}.

For the direct simulation of the wave function, the time-splitting spectral method (TSSP) is a widely used approach \cite{baoTimeSplittingSpectralApproximations2002}. However, to correctly capture the oscillations, TSSP requires a restrictive meshing strategy where both the time step and spatial grid size satisfy $\Delta t, \Delta x = \mathcal{O}(\varepsilon)$. While $\Delta t = \mathcal{O}(1)$ can be sufficient if only physical observables are of interest, resolving the full wave function remains computationally intensive for small $\varepsilon$ \cite{golseConvergenceTimeSplitting2019}. To mitigate these costs, various asymptotic methods valid as $\varepsilon \to 0$ have been developed, including level set methods \cite{jinComputationSemiclassicalLimit} and moment closure methods based on WKB analysis or Wigner transforms. Additionally, wave-packet-based approaches, such as the Gaussian beam method 
\cite{huangGaussianBeamMethods2012,jinBlochDecompositionbasedGaussian2010,jinGaussianBeamMethods2014}
and the frozen Gaussian wave packet method 
\cite{chaiFrozenGaussianApproximation2019,luFrozenGaussianApproximation,luFrozenGaussianApproximation2011,xieFrozenGaussianSampling2021a},
reduce the full quantum dynamics to the evolution of Gaussian wave packets. We refer to \cite{jinMathematicalComputationalMethods2011,lasserComputingQuantumDynamics2020} for a comprehensive review of these deterministic methods.

When the physical medium is influenced by external environments—such as randomly distributed impurities 
\cite{avishai1992electron,fisher1991directed}
amorphous materials, or random alloys 
\cite{johnson1986density,leschke2005survey}
—the potential field must be treated as a random variable \cite{anderson1958absence,roati2008anderson}. Furthermore, practical limitations in experimental techniques inevitably introduce uncertainties into the initial conditions \cite{jinGaussianWavePacket2020}. While several theoretical studies have addressed the semiclassical Schrödinger equation with uncertainties \cite{gerard1997homogenization,hagedorn1998raising,lee1982exact,xiu2002wiener,zhou2014numerical}, numerical methods specifically designed for \eqref{Schro-equation} remain relatively scarce \cite{gunzburger2014stochastic}.

To resolve the general numerical challenges in Uncertainty Quantification (UQ), such as the curse of dimensionality, Stochastic Galerkin (SG) and multi-fidelity methods within the stochastic collocation framework \cite{Babuska2010siamrev, Giles2008or, Mishra2016siamjuq} have proven to be highly effective. Compared with the classical Monte Carlo method, the SG approach enjoys spectral accuracy in the random space provided the solution is sufficiently smooth. Simultaneously, numerous non-intrusive multi-fidelity approaches have been developed to approximate high-fidelity solutions efficiently by combining models of varying costs and accuracies \cite{Survey2018, Zhu2014siamjuq}. These frameworks have seen successful applications in kinetic equations \cite{ DLPZ25, Dimarco2019jcp, Liu2020jcp, YL26} and, more recently, the Schrödinger equation \cite{lin2025class}.

Beyond the inherent difficulties of the deterministic setting, the semiclassical Schrödinger equation with uncertainties introduces two additional significant challenges. The first is the high dimensionality of the random variable $\bm{\xi}$, which leads to the curse of dimensionality—a common bottleneck in most uncertainty quantification (UQ) problems \cite{wuBlochDecompositionbasedStochastic2016}. The second, and more specific to the equation under study, is the high oscillation of the wave function with respect to the random variable \cite{jinGaussianWavePacket2020}. These oscillations necessitate specialized numerical treatments to avoid prohibitive computational costs as $\varepsilon \ll 1$. To address these challenges, several effective approaches have been proposed specifically for the Schrödinger equation: Wu and Huang \cite{wuBlochDecompositionbasedStochastic2016} combined the stochastic Galerkin method with a Bloch-decomposition-based time-splitting spectral method for periodic potentials; Liu et al. \cite{jinGaussianWavePacket2020} integrated the stochastic collocation method with the Gaussian wave packet transform to resolve random-space oscillations; and Lin and Liu \cite{lin2025class} employed multi-fidelity approaches for efficient numerical computations.

In addition to these developments, several effective numerical strategies have been proposed to handle Schr\"odinger-type equations with random or multiscale potentials. For instance, the localized orthogonal decomposition method was developed to treat multiscale potential structures \cite{WZ2022}, while quasi-Monte Carlo time-splitting pseudospectral methods have been employed to improve sampling efficiency \cite{WZZ2024}. Other notable contributions include the implementation of efficient finite element methods for semiclassical problems \cite{JLZC25, PZ2025} and the design of specialized numerical integrators for continuous disordered systems \cite{Zhao2021}. These works have  enriched research results on the topic of uncertainty quantification for quantum dynamics.

\textbf{Main contributions.} In this article, we address the challenge of resolving high oscillations in the random space for the semiclassical Schrödinger equation with uncertainties by extending the dynamical low-rank (DLR) approximation framework. Originally developed for matrices \cite{kochDynamicalLowRankApproximation2007} and time-dependent PDEs \cite{einkemmerLowrankProjectorsplittingIntegrator2020}, we adapt this framework to the uncertainty quantification (UQ) setting to exploit the latent structure of the stochastic wave function. Our primary contribution is a systematic numerical investigation into the performance and efficiency of the DLR approach in the semiclassical regime. Through extensive numerical experiments, we demonstrate that—at least for the cases investigated—the essential dynamics of the oscillatory wave function remain concentrated on a low-rank manifold even as $\varepsilon \to 0$. Most notably, the DLR method is shown to achieve comparable accuracy to the standard stochastic Galerkin method while utilizing a significantly smaller number of basis functions. This drastic reduction in the required degrees of freedom leads to substantial savings in both computational and storage costs, providing a robust and efficient alternative for high-dimensional UQ problems in quantum dynamics. These findings offer practical guidance for the efficient simulation of quantum systems where traditional methods become computationally prohibitive.

The rest of this paper is organized as follows. In \cref{Sec: DLR_Schrodinger}, we derive the DLR evolution equations specifically for the semiclassical Schrödinger equation, adapt the splitting and unconventional integrators to this setting, and analyze their computational and storage complexities. Detailed numerical experiments, including efficiency benchmarks against the stochastic Galerkin method and sensitivity analyses of the rank dynamics, are presented in \cref{Sec: Numerical}. Finally, \cref{Sec: conclusion} concludes the paper with a summary of our findings and potential directions for future research. For the reader's convenience, a concise review of the mathematical formulation and foundational integrators of the DLR approximation for matrices is provided in \ref{Sec: DLR_matrix}.

    \section{Dynamical low-rank approximation for the uncertain Schr\"odinger equation}
\label{Sec: DLR_Schrodinger}

In this section, we apply the dynamical low-rank (DLR) framework to the semiclassical Schrödinger equation subject to uncertainties. We begin by establishing the necessary mathematical notation and deriving the explicit DLR evolution equations tailored for the wave function $\psi(t, x, \xi)$. Subsequently, we extend both the projector-splitting and unconventional integrators to this stochastic setting. A comprehensive description of the resulting numerical algorithms is provided, followed by a detailed discussion on their computational complexity and implementation efficiency.

\subsection{Notations and Formula}

For clarity, we define the inner products and the associated norms for the wave functions as follows:
\begin{align*}
    \langle \psi, \phi \rangle_{x,\xi} &:= \int_{\mathbb{R}^m} \int_{\mathbb{R}^n} \psi(t,x,\xi) \bar{\phi}(t,x,\xi) \pi(\xi) \mathrm{d}x\mathrm{d}\xi, \quad \|\psi\| := \sqrt{\langle \psi, \psi \rangle_{x,\xi}}, \\
    \langle \psi, \phi \rangle_{x} &:= \int_{\mathbb{R}^m} \psi(t,x,\xi) \bar{\phi}(t,x,\xi) \mathrm{d}x, \\
    \langle \psi, \phi \rangle_{\xi} &:= \int_{\mathbb{R}^n} \psi(t,x,\xi) \bar{\phi}(t,x,\xi) \pi(\xi) \mathrm{d}\xi,
\end{align*}
where $\bar{\phi}$ denotes the complex conjugate of $\phi$. Furthermore, the expectation with respect to the random variable $\xi$ is denoted by:
\begin{align*}
    \mathbb{E}[f(t,x,\xi)] := \int_{\mathbb{R}^n} f(t,x,\xi) \pi(\xi) \mathrm{d}\xi = \langle f, 1 \rangle_{\xi}.
\end{align*}

We seek a low-rank approximation of the wave function $\psi(t,x,\xi)$ in the following form:
\begin{align*}
    \psi(t,x,\xi) \approx \psi_A(t,x,\xi) = \sum_{i=1}^{r} \sum_{j=1}^{r} X_i(t,x) S_{ij}(t) W_j(t,\xi),
\end{align*}
where $\{X_i(t,x)\}_{i=1}^r$ and $\{W_j(t,\xi)\}_{j=1}^r$ are sets of orthonormal basis functions satisfying:
\begin{align}\label{eq: orthonormal basis}
    \langle X_i, X_j \rangle_{x} = \delta_{ij}, \quad \langle W_i, W_j \rangle_{\xi} = \delta_{ij}, \quad \forall i,j = 1, \dots, r,
\end{align}
and $S(t) = (S_{ij}(t)) \in \mathbb{C}^{r \times r}$ is a nonsingular core matrix. The low-rank manifold $\mathcal{M}_r$ is then defined as:
\begin{equation}
    \mathcal{M}_r := \left\{ \sum_{i,j=1}^{r} X_i S_{ij} W_j : \{X_i, W_j\} \text{ satisfy \eqref{eq: orthonormal basis}}, S \text{ is nonsingular} \right\}.
\end{equation}

To derive the evolution equations, we characterize the tangent space $\mathcal{T}_{\psi_A}\mathcal{M}_r$ at a point $\psi_A \in \mathcal{M}_r$:
\begin{lemma}\label{lemma: tangent space of M_r at psi_A} 
    Any tangent vector $\eta \in \mathcal{T}_{\psi_A}\mathcal{M}_r$ can be uniquely decomposed into the following form:
    \begin{align*}
        \eta = \sum_{i,j=1}^{r} \left( \eta_{X,i} S_{ij} W_j + X_i \eta_{S,ij} W_j + X_i S_{ij} \eta_{W,j} \right),
    \end{align*}  
    where $\eta_{S,ij} \in \mathbb{C}$ represents the variation of the core matrix, and $\eta_{X,i}(x)$, $\eta_{W,j}(\xi)$ are the tangent components associated with the basis functions. These components must satisfy the following orthogonality relations:
    \begin{align}\label{eq: derivative of orthogonality}
        \langle \eta_{X,i}, X_j \rangle_{x} + \langle X_i, \eta_{X,j} \rangle_{x} = 0, \quad \langle \eta_{W,i}, W_j \rangle_{\xi} + \langle W_i, \eta_{W,j} \rangle_{\xi} = 0.
    \end{align}
\end{lemma}

\begin{proof}
    Consider a smooth curve $\psi_A(\tau,x,\xi)=\sum_{i,j=1}^{r}X_i(\tau,x)S_{ij}(\tau)W_j(\tau,\xi)$ within the manifold $\mathcal{M}_r$, such that $\psi_A(0,x,\xi)=\psi_A(x,\xi)$. Differentiating $\psi_A(\tau,x,\xi)$ with respect to $\tau$ at $\tau=0$ yields a general tangent vector $\eta \in \mathcal{T}_{\psi_A}\mathcal{M}_r$:
    \begin{align*}
        \eta = \left.\frac{\mathrm{d}}{\mathrm{d}\tau}\right|_{\tau=0} \psi_A(\tau,x,\xi)
        &= \sum_{i,j=1}^{r} \left( \dot{X}_i(0) S_{ij}(0) W_j(0) + X_i(0) \dot{S}_{ij}(0) W_j(0) + X_i(0) S_{ij}(0) \dot{W}_j(0) \right) \\
        &= \sum_{i,j=1}^{r} \left( \eta_{X,i} S_{ij} W_j + X_i \eta_{S,ij} W_j + X_i S_{ij} \eta_{W,j} \right),
    \end{align*}
    where we identify the tangent components as $\eta_{X,i} = \dot{X}_i(0)$, $\eta_{S,ij} = \dot{S}_{ij}(0)$, and $\eta_{W,j} = \dot{W}_j(0)$. 
    Since the basis functions $\{X_i(\tau,x)\}$ and $\{W_j(\tau,\xi)\}$ remain orthonormal for all $\tau$ along the curve, they satisfy the conditions:
    \begin{align*}
        \langle X_i(\tau), X_j(\tau) \rangle_{x} = \delta_{ij}, \quad \langle W_i(\tau), W_j(\tau) \rangle_{\xi} = \delta_{ij}.
    \end{align*} 
    Differentiating these identities with respect to $\tau$ at $\tau=0$ leads to:
    \begin{align*}
        \langle \dot{X}_i(0), X_j(0) \rangle_{x} + \langle X_i(0), \dot{X}_j(0) \rangle_{x} = 0, \quad \langle \dot{W}_i(0), W_j(0) \rangle_{\xi} + \langle W_i(0), \dot{W}_j(0) \rangle_{\xi} = 0,
    \end{align*}
    which, in terms of the tangent components, is equivalent to:
    \begin{align*}
        \langle \eta_{X,i}, X_j \rangle_{x} + \langle X_i, \eta_{X,j} \rangle_{x} = 0, \quad 
        \langle \eta_{W,i}, W_j \rangle_{\xi} + \langle W_i, \eta_{W,j} \rangle_{\xi} = 0.
    \end{align*}
    This completes the proof.
\end{proof}

The aforementioned lemma indicates that the representation of a tangent vector $\eta$ in terms of $\eta_X, \eta_W$, and $\eta_S$ is not unique. To ensure a one-to-one correspondence between the tangent vector and its components, we impose the following gauge conditions:
\begin{align}\label{eq: orthogonal gauge}
    \langle \eta_{X,i}, X_j \rangle_{x} = 0, \quad \langle \eta_{W,i}, W_j \rangle_{\xi} = 0, \quad \forall i, j = 1, \dots, r.
\end{align}
Under these conditions, the components are uniquely determined by $\eta$, as stated in the following lemma.

\begin{lemma}
    If the tangent components satisfy the gauge conditions in \eqref{eq: orthogonal gauge}, then $\eta_S, \eta_X$, and $\eta_W$ are uniquely determined by $\eta$ through the following relations:
    \begin{align}
        \eta_{S,ij} &= \langle \eta, X_i W_j \rangle_{x,\xi}, \label{eq: unique S} \\
        \sum_{j=1}^{r} S_{ij} \eta_{W,j} &= \langle \eta, X_i \rangle_{x} - \sum_{j=1}^{r} \eta_{S,ij} W_j, \label{eq: unique W} \\
        \sum_{i=1}^{r} \eta_{X,i} S_{ij} &= \langle \eta, W_j \rangle_{\xi} - \sum_{i=1}^{r} X_i \eta_{S,ij}. \label{eq: unique X}
    \end{align}
\end{lemma}
\begin{proof}
    The proof follows directly by taking the inner products of $\eta$ with the basis functions $X_i$ and $W_j$, and applying the gauge conditions \eqref{eq: orthogonal gauge} along with the orthonormality of the basis.
\end{proof}

We now derive the dynamical low-rank evolution equations for the semiclassical Schrödinger equation with uncertainties, rewritten as:
\begin{align*}
    \partial_t \psi = F(\psi) := \frac{1}{\mathrm{i}\varepsilon} \left( -\frac{\varepsilon^2}{2} \Delta_x \psi + V(x,\xi) \psi \right).
\end{align*}
The DLR approximation seeks an approximate trajectory $\psi_A(t) \in \mathcal{M}_r$ by enforcing the Galerkin projection onto the tangent space $\mathcal{T}_{\psi_A}\mathcal{M}_r$:
\begin{align}\label{eq: Galerkin projection for semiclassical Schrodinger equation}
    \langle \partial_t \psi_A - F(\psi_A), \eta \rangle_{x,\xi} = 0, \quad \forall \eta \in \mathcal{T}_{\psi_A}\mathcal{M}_r.
\end{align}
Substituting the explicit form of $\partial_t \psi_A$ and the test function $\eta$ into \eqref{eq: Galerkin projection for semiclassical Schrodinger equation} yields the following governing equations for the low-rank factors:
\begin{theorem}
    The Galerkin projection \eqref{eq: Galerkin projection for semiclassical Schrodinger equation} is equivalent to the following system of evolution equations for the low-rank factors, under the orthogonal gauge $\langle \dot{X}_i, X_j \rangle_x = 0$ and $\langle \dot{W}_i, W_j \rangle_\xi = 0$:
    \begin{align}\label{eq: total_velocity}
        \dot{\psi}_A = \sum_{i,j=1}^{r} \left( \dot{X}_i S_{ij} W_j + X_i \dot{S}_{ij} W_j + X_i S_{ij} \dot{W}_j \right),
    \end{align}
    where
    \begin{align}
        \dot{S}_{ij} &= \langle F(\psi_A), X_i W_j \rangle_{x,\xi}, \label{eq: dot_S} \\
        \sum_{k=1}^{r} S_{ik} \dot{W}_k &= \langle F(\psi_A), X_i \rangle_{x} - \sum_{k=1}^{r} \dot{S}_{ik} W_k, \label{eq: dot_W} \\
        \sum_{k=1}^{r} \dot{X}_k S_{kj} &= \langle F(\psi_A), W_j \rangle_{\xi} - \sum_{k=1}^{r} X_k \dot{S}_{kj}. \label{eq: dot_X}
    \end{align}
\end{theorem}

\begin{proof}
    By Lemma \ref{lemma: tangent space of M_r at psi_A}, any tangent vector $\eta \in \mathcal{T}_{\psi_A}\mathcal{M}_r$ can be decomposed into components $\eta_X, \eta_W$, and $\eta_S$. We substitute these into the Galerkin condition \eqref{eq: Galerkin projection for semiclassical Schrodinger equation} and test against different component directions.

    \textit{(1) Derivation of $\dot{S}_{ij}$:} 
    By setting $\eta_{X,i} = 0$ and $\eta_{W,j} = 0$ for all $i,j$, the test function becomes $\eta = \sum_{i,j=1}^{r} X_i \eta_{S,ij} W_j$. Substituting this into \eqref{eq: Galerkin projection for semiclassical Schrodinger equation} and using the orthonormality of the bases, we have:
    \begin{align*}
        \langle \dot{\psi}_A, \eta \rangle_{x,\xi} = \sum_{i,j=1}^r \dot{S}_{ij} \bar{\eta}_{S,ij}, \quad \text{and} \quad \langle F(\psi_A), \eta \rangle_{x,\xi} = \sum_{i,j=1}^r \langle F(\psi_A), X_i W_j \rangle_{x,\xi} \bar{\eta}_{S,ij}.
    \end{align*}
    Since $\eta_S$ is arbitrary, we obtain $\dot{S}_{ij} = \langle F(\psi_A), X_i W_j \rangle_{x,\xi}$, which is \eqref{eq: dot_S}.

    \textit{(2) Derivation of $\dot{X}_i$:} 
    Next, we set $\eta_{S} = 0$ and $\eta_{W} = 0$. To isolate $\dot{X}_l$, we choose $\eta_{X,i} = 0$ for $i \neq l$. The orthogonal gauge $\langle \eta_{X,l}, X_k \rangle_x = 0$ implies that $\eta_{X,l} = (I - P_X) \phi$, where $P_X = \sum_{k=1}^r X_k X_k^H$ is the projection onto the space spanned by the current basis, and $\phi(x)$ is an arbitrary function. The test function is $\eta = \sum_{j=1}^r \eta_{X,l} S_{lj} W_j$.
    The Galerkin condition yields:
    \begin{align*}
        \langle \dot{\psi}_A, \eta \rangle_{x,\xi} &= \left\langle \sum_{i,j=1}^r \dot{X}_i S_{ij} W_j, \eta_{X,l} S_{lj} W_j \right\rangle_{x,\xi} = \left\langle \sum_{i=1}^r \dot{X}_i \sum_{j=1}^r S_{ij} \bar{S}_{lj}, (I - P_X) \phi \right\rangle_x, \\
        \langle F(\psi_A), \eta \rangle_{x,\xi} &= \sum_{j=1}^r \bar{S}_{lj} \langle F(\psi_A), (I - P_X) \phi W_j \rangle_{x,\xi} = \sum_{j=1}^r \bar{S}_{lj} \langle (I - P_X) \langle F(\psi_A), W_j \rangle_\xi, \phi \rangle_x.
    \end{align*}
    Utilizing the fact that $(I-P_X)$ is self-adjoint and $\dot{X}_i$ already satisfies the gauge condition (i.e., $(I-P_X)\dot{X}_i = \dot{X}_i$), the arbitrariness of $\phi$ leads to:
    \begin{align*}
        \sum_{i=1}^r \dot{X}_i \sum_{j=1}^r S_{ij} \bar{S}_{lj} = \sum_{j=1}^r \bar{S}_{lj} \left( \langle F(\psi_A), W_j \rangle_\xi - \sum_{k=1}^r X_k \langle F(\psi_A), X_k W_j \rangle_{x,\xi} \right).
    \end{align*}
    Assuming $S$ is nonsingular and substituting $\dot{S}_{kj}$ from \eqref{eq: dot_S}, we obtain:
    \begin{align*}
        \sum_{i=1}^r \dot{X}_i S_{ij} = \langle F(\psi_A), W_j \rangle_\xi - \sum_{k=1}^r X_k \dot{S}_{kj},
    \end{align*}
    which is \eqref{eq: dot_X}.

    \textit{(3) Derivation of $\dot{W}_j$:}
    The derivation for the evolution of the random basis $\dot{W}_j$ follows a procedure analogous to that of the spatial basis in step (2) by choosing $\eta_S = 0$, $\eta_X = 0$, and is therefore omitted.
\end{proof}

\subsection{Numerical Integrators for the DLR Approximation}
In this section, we apply the projector-splitting and unconventional integrators to the semiclassical Schrödinger equation with uncertainties. 

Substituting the evolution equations \eqref{eq: dot_S}--\eqref{eq: dot_X} into the total formula \eqref{eq: total_velocity}, we observe that the DLR approximation can be recast as a projection of the governing operator:
\begin{align}\label{eq: projected_F}
    \partial_t \psi_A = \mathcal{P}(\psi_A) F(t, \psi_A),
\end{align}
where $\mathcal{P}(\psi_A)$ is the orthogonal projection operator onto the tangent space $\mathcal{T}_{\psi_A} \mathcal{M}_r$. For a given function $g(x, \xi)$, this operator can be explicitly expressed as:
\begin{align}\label{eq: projector_expression}
    \mathcal{P}(\psi_A) g = \sum_{j=1}^{r} \langle g, W_j \rangle_{\xi} W_j - \sum_{i,j=1}^{r} \langle g, X_i W_j \rangle_{x,\xi} X_i W_j + \sum_{i=1}^{r} \langle g, X_i \rangle_{x} X_i.
\end{align}
Let $\mathcal{V}_X = \operatorname{span}\{X_1, \dots, X_r\}$ and $\mathcal{V}_W = \operatorname{span}\{W_1, \dots, W_r\}$ denote the subspaces spanned by the spatial and random bases, respectively. The projector $\mathcal{P}(\psi_A)$ can then be written in the following operator form:
\begin{align}\label{eq: projector_operator_form}
    \mathcal{P}(\psi_A) g = P_{\mathcal{V}_W} g - P_{\mathcal{V}_W} P_{\mathcal{V}_X} g + P_{\mathcal{V}_X} g,
\end{align}
where $P_{\mathcal{V}_W}$ and $P_{\mathcal{V}_X}$ are the orthogonal projectors onto $\mathcal{V}_W$ (with respect to $\langle \cdot, \cdot \rangle_{\xi}$) and $\mathcal{V}_X$ (with respect to $\langle \cdot, \cdot \rangle_{x}$), respectively.

By applying the Lie-Trotter splitting method to \eqref{eq: projected_F} according to the three terms in \eqref{eq: projector_operator_form}, the projector-splitting scheme decomposes each time step $t \in [t_n, t_{n+1}]$ into three sub-problems:
\begin{align*}
    \text{(I) } \partial_t \psi^{(1)} &= P_{\mathcal{V}_W} F(t, \psi^{(1)}), \quad \psi^{(1)}(t_n) = \psi_A(t_n), \\
    \text{(II) } \partial_t \psi^{(2)} &= -P_{\mathcal{V}_W} P_{\mathcal{V}_X} F(t, \psi^{(2)}), \quad \psi^{(2)}(t_n) = \psi^{(1)}(t_{n+1}), \\
    \text{(III) } \partial_t \psi^{(3)} &= P_{\mathcal{V}_X} F(t, \psi^{(3)}), \quad \psi^{(3)}(t_n) = \psi^{(2)}(t_{n+1}),
\end{align*}
where the solution at the next step is set as $\psi_A(t_{n+1}) = \psi^{(3)}(t_{n+1})$. The following lemma characterizes the evolution of the low-rank factors within each substep:
\begin{lemma}
    The solutions $\psi^{(k)}(t) = \sum_{i,j=1}^r X_i^{(k)}(t) S_{ij}^{(k)}(t) W_j^{(k)}(t)$ for $k=1,2,3$ are determined as follows:
    \begin{enumerate}
        \item \textbf{K-step}: $W_j^{(1)}$ remains constant ($W_j^{(1)}(t) = W_j(t_n)$). The updated spatial factors $K_j(t, x) = \sum_{i=1}^r X_i^{(1)}(t) S_{ij}^{(1)}(t)$ evolve according to:
        \begin{align*}
            \partial_t K_j(t, x) = \langle F(t, \sum_{j=1}^r K_j(t, x) W_j(t_n, \xi)), W_j(t_n, \xi) \rangle_{\xi}, \quad j=1,\dots,r.
        \end{align*}
        \item \textbf{S-step}: Both $X_i^{(2)}$ and $W_j^{(2)}$ remain constant ($X_i^{(2)} = X_i^{(1)}(t_{n+1}), W_j^{(2)} = W_j(t_n)$). The core matrix $S_{ij}^{(2)}$ evolves according to:
        \begin{align*}
            \partial_t S_{ij}^{(2)} = -\langle F(t, \sum_{k,l=1}^r X_k^{(2)} S_{kl}^{(2)} W_l^{(2)}), X_i^{(2)} W_j^{(2)} \rangle_{x,\xi}, \quad i,j=1,\dots,r.
        \end{align*}
        \item \textbf{L-step}: $X_i^{(3)}$ remains constant ($X_i^{(3)} = X_i^{(1)}(t_{n+1})$). The updated random factors $L_i(t, \xi) = \sum_{j=1}^r S_{ij}^{(3)}(t) W_j^{(3)}(t)$ evolve according to:
        \begin{align*}
            \partial_t L_i(t, \xi) = \langle F(t, \sum_{i=1}^r X_i^{(3)}(x) L_i(t, \xi)), X_i^{(3)}(x) \rangle_{x}, \quad j=1,\dots,r.
        \end{align*}
    \end{enumerate}
\end{lemma}
\begin{proof}
    \textit{Step 1: The K-step.} 
    Let $\psi^{(1)} = \sum_{j=1}^r K_j^{(1)} W_j^{(1)}$, where the spatial factors are defined as $K_j^{(1)} = \sum_{i=1}^r X_i^{(1)} S_{ij}^{(1)}$. The evolution equation for this step is given by:
    \begin{align*}
        \partial_t \psi^{(1)} = P_{\mathcal{V}_W} F(\psi^{(1)}) = \sum_{j=1}^r \langle F(\psi^{(1)}), W_j^{(1)} \rangle_{\xi} W_j^{(1)}.
    \end{align*} 
    Substituting the expression for $\psi^{(1)}$ into the above yields:
    \begin{align}\label{eq: proof_step1_expand}
        \sum_{j=1}^r \left( \partial_t K_j^{(1)} W_j^{(1)} + K_j^{(1)} \partial_t W_j^{(1)} \right) = \sum_{j=1}^r \langle F(\psi^{(1)}), W_j^{(1)} \rangle_{\xi} W_j^{(1)}.
    \end{align}
    Taking the inner product with $W_k^{(1)}$ on both sides, and invoking the orthonormal property $\langle W_j^{(1)}, W_k^{(1)} \rangle_{\xi} = \delta_{jk}$ along with the gauge condition $\langle \partial_t W_j^{(1)}, W_k^{(1)} \rangle_{\xi} = 0$, we obtain:
    \begin{align*}
        \partial_t K_k^{(1)} = \langle F(\psi^{(1)}), W_k^{(1)} \rangle_{\xi}.
    \end{align*}
    Substituting this back into \eqref{eq: proof_step1_expand} implies $\sum_{j=1}^r K_j^{(1)} \partial_t W_j^{(1)} = 0$. By writing $K_j^{(1)} = \sum_i X_i^{(1)} S_{ij}^{(1)}$ and taking the inner product with $X_i^{(1)}$, the orthonormality of $X_i^{(1)}$ and the nonsingularity of $S^{(1)}$ lead to $\partial_t W_j^{(1)} = 0$.

    \textit{Step 2: The S-step.} 
    For $\psi^{(2)} = \sum_{i,j=1}^r X_i^{(2)} S_{ij}^{(2)} W_j^{(2)}$, the evolution is governed by:
    \begin{align*}
        \partial_t \psi^{(2)} = -P_{\mathcal{V}_W} P_{\mathcal{V}_X} F(\psi^{(2)}) = -\sum_{i,j=1}^r \langle F(\psi^{(2)}), X_i^{(2)} W_j^{(2)} \rangle_{x,\xi} X_i^{(2)} W_j^{(2)}.
    \end{align*}
    Expanding the time derivative $\partial_t \psi^{(2)}$ gives:
    \begin{equation}\label{eq: proof_step2_expand}
    \begin{aligned}
        &\sum_{i,j=1}^r \left( \partial_t X_i^{(2)} S_{ij}^{(2)} W_j^{(2)} + X_i^{(2)} \partial_t S_{ij}^{(2)} W_j^{(2)} + X_i^{(2)} S_{ij}^{(2)} \partial_t W_j^{(2)} \right) \\
        =& -\sum_{i,j=1}^r \langle F(\psi^{(2)}), X_i^{(2)} W_j^{(2)} \rangle_{x,\xi} X_i^{(2)} W_j^{(2)}.
    \end{aligned}
    \end{equation}
    Taking the inner product with $X_k^{(2)} W_l^{(2)}$ and applying the gauge conditions $\langle \partial_t X_i, X_k \rangle_x = 0$ and $\langle \partial_t W_j, W_l \rangle_\xi = 0$, we isolate the core matrix evolution:
    \begin{align*}
        \partial_t S_{ij}^{(2)} = -\langle F(\psi^{(2)}), X_i^{(2)} W_j^{(2)} \rangle_{x,\xi}.
    \end{align*}
    Substituting this result into \eqref{eq: proof_step2_expand} yields $\sum_{i,j} (\partial_t X_i^{(2)} S_{ij}^{(2)} W_j^{(2)} + X_i^{(2)} S_{ij}^{(2)} \partial_t W_j^{(2)}) = 0$. Using the orthonormality and nonsingularity arguments as in Step 1, we conclude $\partial_t X_i^{(2)} = 0$ and $\partial_t W_j^{(2)} = 0$.

    \textit{Step 3: The L-step.} 
    Let $\psi^{(3)} = \sum_{i=1}^r X_i^{(3)} L_i^{(3)}$ with $L_i^{(3)} = \sum_{j=1}^r S_{ij}^{(3)} W_j^{(3)}$. The governing equation is:
    \begin{align*}
        \partial_t \psi^{(3)} = P_{\mathcal{V}_X} F(\psi^{(3)}) = \sum_{i=1}^r \langle F(\psi^{(3)}), X_i^{(3)} \rangle_{x} X_i^{(3)}.
    \end{align*}
    Repeating the projection procedure by taking the inner product with $X_k^{(3)}$, we find:
    \begin{align*}
        \partial_t L_i^{(3)} = \langle F(\psi^{(3)}), X_i^{(3)} \rangle_{x}.
    \end{align*}
    Similarly, substituting this back into the evolution of $\psi^{(3)}$ and utilizing the orthonormality of $W_j$ and nonsingularity of $S^{(3)}$, we obtain $\partial_t X_i^{(3)} = 0$, which concludes the proof.
\end{proof}

\subsection{Numerical Implementation}

We now discuss the numerical solution of the three sub-problems derived in the previous lemma. By substituting the specific form of the operator $F(\psi_A)$ and utilizing the fact that certain factors remain stationary in each substep, the equations can be simplified significantly.

\paragraph{K-step} Since $\partial_t W_j^{(1)}=0$, the evolution of the spatial factors $K_j(t,x)$ is governed by a matrix-valued Schrödinger equation:
\begin{align*}
    \partial_t K_j(t,x) = \frac{1}{\mathrm{i}\varepsilon} \sum_{l=1}^r \left( -\frac{\varepsilon^2}{2} \delta_{jl} \Delta_x + \mathcal{C}_{jl}^{(1)} \right) K_l(t,x), 
\end{align*}
where the coefficient matrix $\mathcal{C}_{jl}^{(1)} = \langle V(x,\xi) W_l^{(1)},W_j^{(1)}\rangle_{\xi}$ is precomputed. This system can be efficiently integrated using the time-splitting spectral method \cite{baoTimeSplittingSpectralApproximations2002}.

\paragraph{S-step} In this step, the bases $\{X_i^{(2)}\}$ and $\{W_j^{(2)}\}$ are fixed. The evolution of the core matrix $S(t)$ reduces to a low-dimensional system of ODEs:
\begin{align*}
    \dot{S}_{ij}(t) = -\frac{1}{\mathrm{i}\varepsilon} \sum_{l,k=1}^r \left[ -\frac{\varepsilon^2}{2} \mathcal{C}_{il}^{(21)} \delta_{kj} + \mathcal{D}_{iljk}^{(22)} \right] S_{lk}(t),
\end{align*}
where $\mathcal{C}_{il}^{(21)} = \langle \Delta_x X_l^{(2)}, X_i^{(2)} \rangle_x$ and $\mathcal{D}_{iljk}^{(22)} = \langle V X_l^{(2)} W_k^{(2)}, X_i^{(2)} W_j^{(2)} \rangle_{x,\xi}$. This system can be solved using standard ODE integrators, such as fourth-order Runge-Kutta or exponential integrators.

\paragraph{L-step} With $X_i^{(3)}$ constant, the random factors $L_i(t, \xi)$ satisfy:
\begin{align*}
    \partial_t L_i(t,\xi) = \frac{1}{\mathrm{i}\varepsilon} \sum_{l=1}^r \left( -\frac{\varepsilon^2}{2} \mathcal{C}_{il}^{(21)} + \mathcal{V}_{il}(t, \xi) \right) L_l(t,\xi),
\end{align*}
where $\mathcal{V}_{il}(t, \xi) = \langle V(x,\xi) X_l^{(3)}, X_i^{(3)} \rangle_x$. For each realization of $\xi$, this constitutes a small-scale ODE.

The preceding lemma and the subsequent discussion establish the foundation for the projector-splitting integrator tailored for the semiclassical Schrödinger equation with uncertainties, as detailed in \cref{alg: PS_integrator_Schrodinger}. This integrator is particularly robust as it maintains the low-rank structure while providing a first-order accurate approximation in time.

Similarly, we extend the unconventional integrator to this stochastic semiclassical setting, presented in \cref{alg: UI_integrator_Schrodinger}. Compared to the projector-splitting method, the unconventional integrator offers the advantage of parallelizing the updates of the spatial and random bases, which significantly enhances computational efficiency in large-scale simulations.

\begin{algorithm}[htbp]
    \caption{The projector splitting integrator for the semiclassical Schrödinger equation with uncertainties}
    \label{alg: PS_integrator_Schrodinger}
    \begin{algorithmic}[1]
        \STATE \textbf{Input}: Initial factors $X^0, S^0, W^0$, time step $\Delta t$, total steps $N$, spatial grid size $h$.
        \FOR {$n=0, 1, \dots, N-1$}
            \STATE 1. \textbf{K-step}: Solve the following equation from $t_n$ to $t_{n+1}$ with $K_j(t_n) = \sum_{i=1}^r X_i^n S_{ij}^n$:
            \begin{equation*}
                \partial_t K_j(t,x) = \frac{1}{\mathrm{i}\varepsilon} \sum_{l=1}^r \left( -\frac{\varepsilon^2}{2} \delta_{jl} \Delta_x + \mathcal{C}_{jl}^{(1)} \right) K_l(t,x).
            \end{equation*}
            Perform the Gram-Schmidt procedure to obtain $K_j(t_{n+1}) = \sum_{i=1}^r X_i^{n+1} \widehat{S}_{ij}^{n+1}$.
            
            \STATE 2. \textbf{S-step}: Solve the following equation from $t_n$ to $t_{n+1}$ with $S_{ij}(t_n) = \widehat{S}_{ij}^{n+1}$:
            \begin{equation*}
                \dot{S}_{ij}(t) = -\frac{1}{\mathrm{i}\varepsilon} \sum_{l,k=1}^r \left[ -\frac{\varepsilon^2}{2} \mathcal{C}_{il}^{(21)} \delta_{kj} + \mathcal{D}_{iljk}^{(22)} \right] S_{lk}(t).
            \end{equation*}
            Set $\widetilde{S}_{ij}^{n} = S_{ij}(t_{n+1})$.
            
            \STATE 3. \textbf{L-step}: Solve the following equation from $t_n$ to $t_{n+1}$ with $L_i(t_n) = \sum_{j=1}^r \widetilde{S}_{ij}^{n} W_j^n$:
            \begin{equation*}
                \partial_t L_i(t,\xi) = \frac{1}{\mathrm{i}\varepsilon} \sum_{l=1}^r \left( -\frac{\varepsilon^2}{2} \mathcal{C}_{il}^{(21)} + \mathcal{V}_{il}(t, \xi) \right) L_l(t,\xi).
            \end{equation*}
            Perform the Gram-Schmidt procedure to obtain $L_i(t_{n+1}) = \sum_{j=1}^r S_{ij}^{n+1} W_j^{n+1}$.
        \ENDFOR
        \STATE \textbf{Output}: $\psi^N = \sum_{i,j=1}^r X_i^N S_{ij}^N W_j^N$.
    \end{algorithmic}
\end{algorithm}

\begin{algorithm}[htbp]
    \caption{The unconventional integrator for the semiclassical Schrödinger equation with uncertainties}
    \label{alg: UI_integrator_Schrodinger}
    \begin{algorithmic}[1]
        \STATE \textbf{Input}: Initial factors $X^0, S^0, W^0$, time step $\Delta t$, total steps $N$, spatial grid size $h$.
        \FOR {$n=0, 1, \dots, N-1$}
            \STATE 1. \textbf{Update bases in parallel}:
            \begin{itemize}
                \item \textbf{K-step}: Solve for $K_j(t,x)$ from $t_n$ to $t_{n+1}$ with initial condition $K_j(t_n) = \sum_{i=1}^r X_i^n S_{ij}^n$. 
                Perform Gram-Schmidt: $K_j(t_{n+1}) = \sum_{i=1}^r X_i^{n+1} \widehat{S}_{ij}^{n+1}$, and compute $a_{ki} = \langle X_k^n, X_i^{n+1} \rangle_x$.
                
                \item \textbf{L-step}: Solve for $L_i(t,\xi)$ from $t_n$ to $t_{n+1}$ with initial condition $L_i(t_n) = \sum_{j=1}^r S_{ij}^n W_j^n$. 
                Perform Gram-Schmidt: $L_i(t_{n+1}) = \sum_{j=1}^r \widetilde{S}_{ij}^{n+1} W_j^{n+1}$, and compute $b_{lj} = \langle W_l^n, W_j^{n+1} \rangle_\xi$.
            \end{itemize}
            
            \STATE 2. \textbf{S-step (Core Matrix Update)}:
            Solve for $S_{ij}(t)$ from $t_n$ to $t_{n+1}$ with the updated projection, using the initial condition $S_{ij}(t_n) = \sum_{k,l=1}^r a_{ki} S_{kl}^n b_{lj}$:
            \begin{equation*}
                \dot{S}_{ij}(t) = -\frac{1}{\mathrm{i}\varepsilon} \sum_{l,k=1}^r \left[ -\frac{\varepsilon^2}{2} \mathcal{C}_{il}^{(21)} \delta_{kj} + \mathcal{D}_{iljk}^{(22)} \right] S_{lk}(t).
            \end{equation*}
            Set $S_{ij}^{n+1} = S_{ij}(t_{n+1})$.
        \ENDFOR
        \STATE \textbf{Output}: $\psi^N = \sum_{i,j=1}^r X_i^N S_{ij}^N W_j^N$.
    \end{algorithmic}
\end{algorithm}

\subsection{Computational and storage complexity}

In this subsection, we analyze the computational and storage complexity of the DLR method per time step. Let $N_x$ denote the number of spatial grid points and $N_{\xi}$ be the number of quadrature points in the random space. The DLR rank is denoted by $r$, while $K = \binom{\text{deg}+n}{n}$ represents the number of basis functions in the stochastic Galerkin (SG) method.

\paragraph{Computational Complexity:}
Under the assumption that the potential is separable, i.e., $V(\bm{x}, \bm{\xi}) = V_1(\bm{x})V_2(\bm{\xi})$, the costs for the DLR substeps are as follows:
\begin{itemize}
    \item \textbf{K-step}: We solve a matrix-valued Schrödinger equation for the spatial factors $K_j(t, \bm{x})$. Using the time-splitting spectral method, the evolution requires $O(r N_x \log N_x)$ operations. The pre-computation of the coefficients $c_{jl}^{(11)}$ and $c_{jl}^{(12)}$ involves integrals over the random space $\bm{\xi}$, which scales as $O(r^2 N_{\xi})$.
    \item \textbf{S-step}: This step evolves the $r \times r$ core matrix $S(t)$ by solving a system of ordinary differential equations (ODEs). Due to the separability of the potential, the construction of the coefficient matrices $c_{il}^{(21)}$ and $c_{iljk}^{(22)}$ only requires independent spatial and random space integrations, reducing the complexity to $O(r^2 N_x + r^2 N_{\xi})$. The cost of the ODE solver itself is negligible, scaling as $O(r^3)$ or $O(r^4)$ per step.
    \item \textbf{L-step}: We solve $N_{\xi}$ independent ODEs for the random factors $L_i(t, \bm{\xi})$. The evolution scales as $O(N_{\xi} r^2)$. The pre-computation of the coefficients $c_{il}^{(32)}(\bm{\xi})$ across the random grid requires $O(r^2 N_x + N_{\xi} r^2)$.
\end{itemize}

\paragraph{Storage Complexity:}
The storage requirements highlight a significant advantage of the DLR approach. The DLR method only needs to store the factors $\{X_i\}_{i=1}^r$, $S$, and $\{W_j\}_{j=1}^r$, leading to a memory footprint of $O(r N_x + r N_{\xi} + r^2)$. In contrast:
\begin{itemize}
    \item The \textbf{stochastic collocation (SC)} method, if storing the full wave function $\psi(t, x, \xi)$, requires $O(N_x N_{\xi})$ storage.
    \item The \textbf{stochastic Galerkin (SG)} method requires $O(K N_x)$ to store the spatial coefficients for all stochastic modes.
\end{itemize}
When $r \ll \min(N_{\xi}, K)$, the DLR method drastically reduces the memory burden, making it feasible for high-dimensional problems where full-grid storage is impossible.

\paragraph{Summary:}
In summary, the DLR method provides significantly lower computational and storage complexities compared to the SC and SG methods. By dynamically adapting the basis to capture the essential dynamics, the DLR method provides a computationally lean and memory-efficient alternative for resolving high oscillations in uncertain semiclassical systems.

\section{Numerical experiments}
\label{Sec: Numerical}

In this section, we evaluate the performance of the proposed dynamical low-rank (DLR) schemes for the uncertain Schrödinger equation. The numerical test suite is designed to verify three key properties:
\begin{itemize}
    \item \textbf{Efficiency:} Comparison against the Stochastic Galerkin (SG) method to demonstrate dimension reduction capabilities.
    \item \textbf{Robustness:} Sensitivity analysis with respect to the magnitude of uncertainty, governed by the parameter $\gamma$ (Example 3).
    \item \textbf{Generalization:} Applicability to non-polynomial potentials, including periodic and confinement regimes (Examples 4 \& 5).
\end{itemize}

\subsection{Physical quantities and error metrics}

In this subsection, we introduce the physical quantities and error norms used to evaluate the performance of the proposed DLR approximation. 

For the uncertain Schrödinger equation, the two primary physical observables of interest are the particle density $\rho$ and the current density $j$, defined as:
\begin{equation}
    \rho(t, x, \xi) = |\psi(t, x, \xi)|^2, \quad j(t, x, \xi) = \varepsilon \operatorname{Im} \left( \bar{\psi}(t, x, \xi) \nabla_x \psi(t, x, \xi) \right).
\end{equation}
Due to the presence of uncertainties, we focus on their expected values over the random space:
\begin{equation}
    \mathbb{E}[\rho](t, x) = \int_{\mathbb{R}^n} \rho(t, x, \xi) \pi(\xi) \mathrm{d}\xi, \quad \mathbb{E}[j](t, x) = \int_{\mathbb{R}^n} j(t, x, \xi) \pi(\xi) \mathrm{d}\xi.
\end{equation}

To quantify the accuracy of the DLR method, we define the relative $L^2$ error of the wave function as:
\begin{equation}
    e_{\psi}(t) := \frac{\| \psi_{\text{DLR}} - \psi_{\text{ref}} \|}{\| \psi_{\text{ref}} \|},
\end{equation}
where $\| \cdot \|$ denotes the $L^2$ norm in the full $(x, \xi)$ space defined in Section \ref{Sec: DLR_Schrodinger}. Additionally, we define the relative errors for the expectations of the density and current density as:
\begin{equation}
    e_{\mathbb{E}[\rho]}(t) := \frac{\| \mathbb{E}[\rho_{\text{DLR}}] - \mathbb{E}[\rho_{\text{ref}}] \|_{L^2(x)}}{\| \mathbb{E}[\rho_{\text{ref}}] \|_{L^2(x)}}, \quad 
    e_{\mathbb{E}[j]}(t) := \frac{\| \mathbb{E}[j_{\text{DLR}}] - \mathbb{E}[j_{\text{ref}}] \|_{L^2(x)}}{\| \mathbb{E}[j_{\text{ref}}] \|_{L^2(x)}}.
\end{equation}
Here, the subscript ``ref'' refers to the reference solution obtained by the time-splitting spectral method (TSSP) on a full tensor grid, while the subscript ``DLR'' denotes the numerical solution obtained by the proposed DLR schemes.

\subsection{Numerical results}
\textbf{Example 1 (Harmonic Oscillator with Random Potential).} 
We consider a harmonic oscillator system with a random potential. The physical parameters are set to $\varepsilon=0.05$, a harmonic trap $V(x)=0.6x^2$, and a constant random potential term $V(\xi)=1$. The initial condition is coupled to the random space via the momentum and position parameters:
\begin{equation*}
    \tilde{p}=0.7(\xi_1-\xi_2), \quad \tilde{q}=1+0.8(\xi_1+2\xi_2).
\end{equation*}
For the numerical discretization, we set the time step $dt=0.01$ and the final time $T=0.5$. The spatial domain is truncated to $[x_l, x_r] = [-4, 6]$ with $N_x=1000$ grid points. The random space is discretized using $N_{\xi_1}=100$ and $N_{\xi_2}=100$ collocation points.

Table \ref{tab:error_comparison_example1} presents a comparative analysis between the Stochastic Galerkin (SG) method and the DLR approximations. The SG method exhibits spectral convergence as the polynomial degree increases. However, achieving a relatively small error requires a high polynomial degree, leading to a basis size of $\binom{40 + 2}{2} = 861$. In contrast, the proposed DLR method achieves superior accuracy with a fixed rank of $r=46$. This represents a significant reduction in degrees of freedom, highlighting the efficiency of the low-rank approximation. Furthermore, the proposed unconventional integrator outperforms the standard DLR-splitting scheme by an order of magnitude in the $L^2$ error of $\psi$, as shown in the highlighted entries of Table \ref{tab:error_comparison_example1}. Figures \ref{fig:density_example1} and \ref{fig:Jdensity_example1} visually confirm that the DLR solution captures the detailed structure of the density and current density evolution.

\begin{table}[!htbp]
\centering
\caption{Error comparison for Example 1 with different ranks $r$.}
\label{tab:error_comparison_example1}
\renewcommand{\arraystretch}{1.2}
\scalebox{0.7}{
\begin{tabular}{l|c|c|c|c|c}
\toprule
\diagbox{Method}{Error} & $\psi$ & $\rho_{mean}$ & $\rho_{std}$ & $j_{mean}$ & $j_{std}$ \\
\midrule
Galerkin ($r=20$)   & $6.022 \times 10^{-1}$ & $4.669 \times 10^{-1}$ & $3.483 \times 10^{-1}$ & $3.858 \times 10^{-1}$ & $4.893 \times 10^{-1}$ \\
Galerkin ($r=30$)   & $3.430 \times 10^{-1}$ & $1.634 \times 10^{-1}$ & $1.693 \times 10^{-1}$ & $1.406 \times 10^{-1}$ & $2.491 \times 10^{-1}$ \\
Galerkin ($r=40$)   & $1.027 \times 10^{-1}$ & $1.388 \times 10^{-2}$ & $2.986 \times 10^{-2}$ & $1.837 \times 10^{-2}$ & $4.462 \times 10^{-2}$ \\
\midrule
DLR-splitting       & $7.812 \times 10^{-2}$ & $4.867 \times 10^{-3}$ & $7.101 \times 10^{-3}$ & $7.584 \times 10^{-3}$ & $1.218 \times 10^{-2}$ \\
Unconventional      & $\mathbf{3.391 \times 10^{-2}}$ & $\mathbf{1.461 \times 10^{-3}}$ & $\mathbf{4.287 \times 10^{-3}}$ & $\mathbf{2.379 \times 10^{-3}}$ & $\mathbf{6.779 \times 10^{-3}}$ \\
\bottomrule
\end{tabular}
}
\end{table}


\begin{figure}[H]
    \centering
    \begin{subfigure}{0.32\textwidth}
        \centering
        \includegraphics[width=\textwidth]{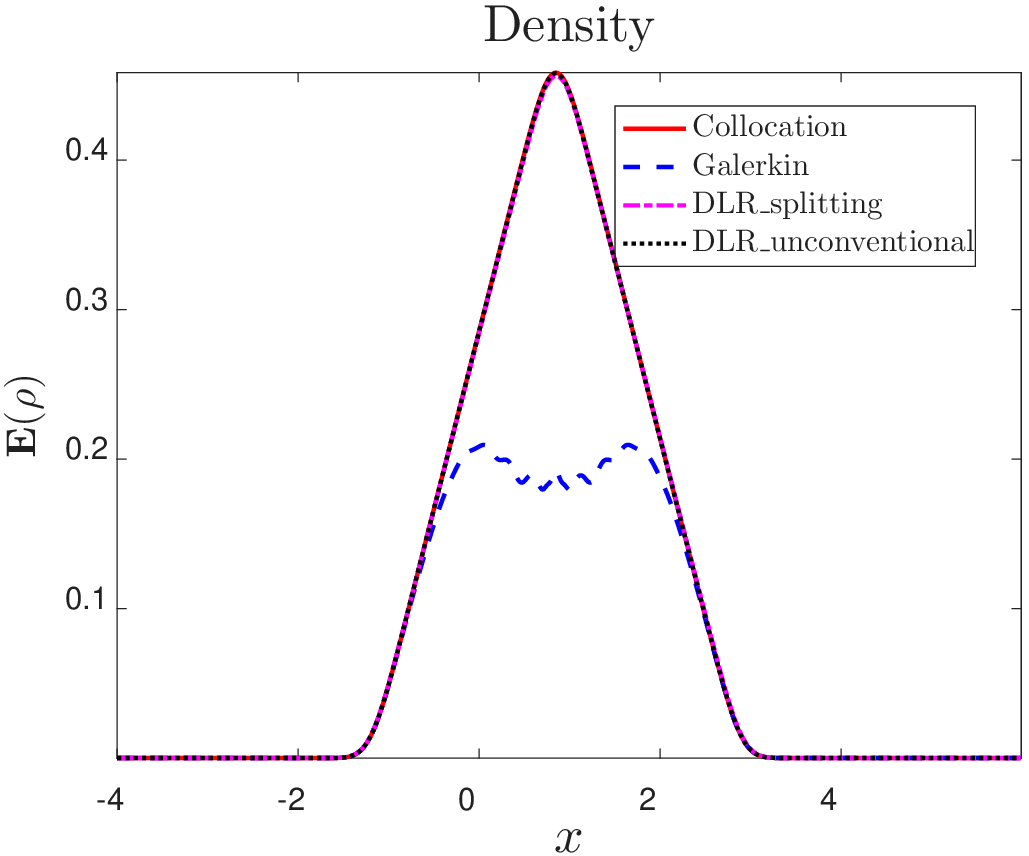}
        \caption{}
        \label{fig:subfig1_density_example1}
    \end{subfigure}
    \hfill
    \begin{subfigure}{0.32\textwidth}
        \centering
        \includegraphics[width=\textwidth]{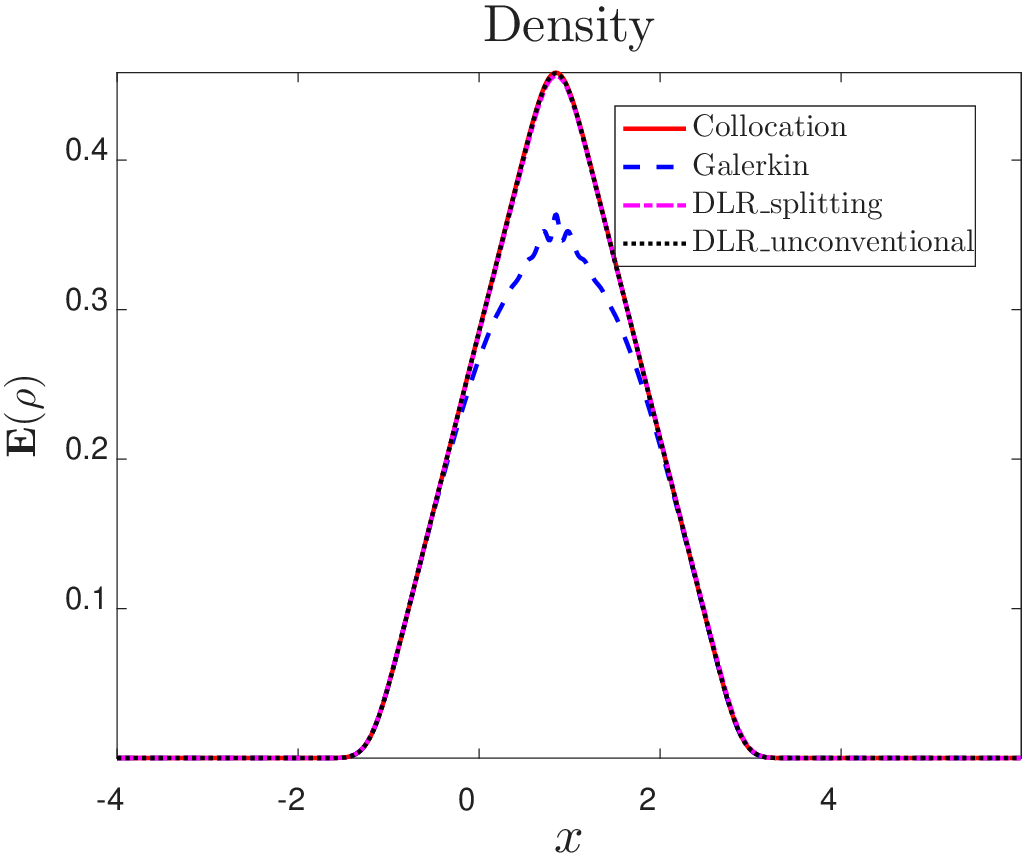}
        \caption{}
        \label{fig:subfig2_density_example1}
    \end{subfigure}
    \hfill
    \begin{subfigure}{0.32\textwidth}
        \centering
        \includegraphics[width=\textwidth]{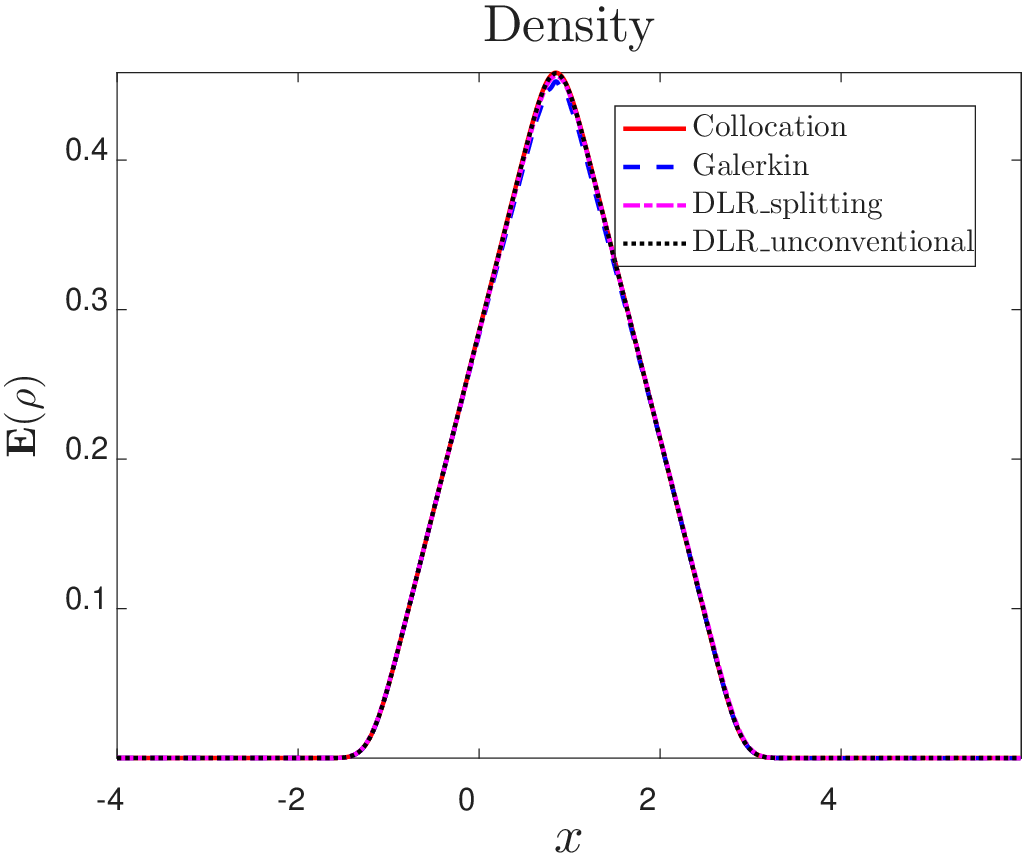}
        \caption{}
        \label{fig:subfig3_density_example1}
    \end{subfigure}
        \begin{subfigure}{0.32\textwidth}
        \centering
        \includegraphics[width=\textwidth]{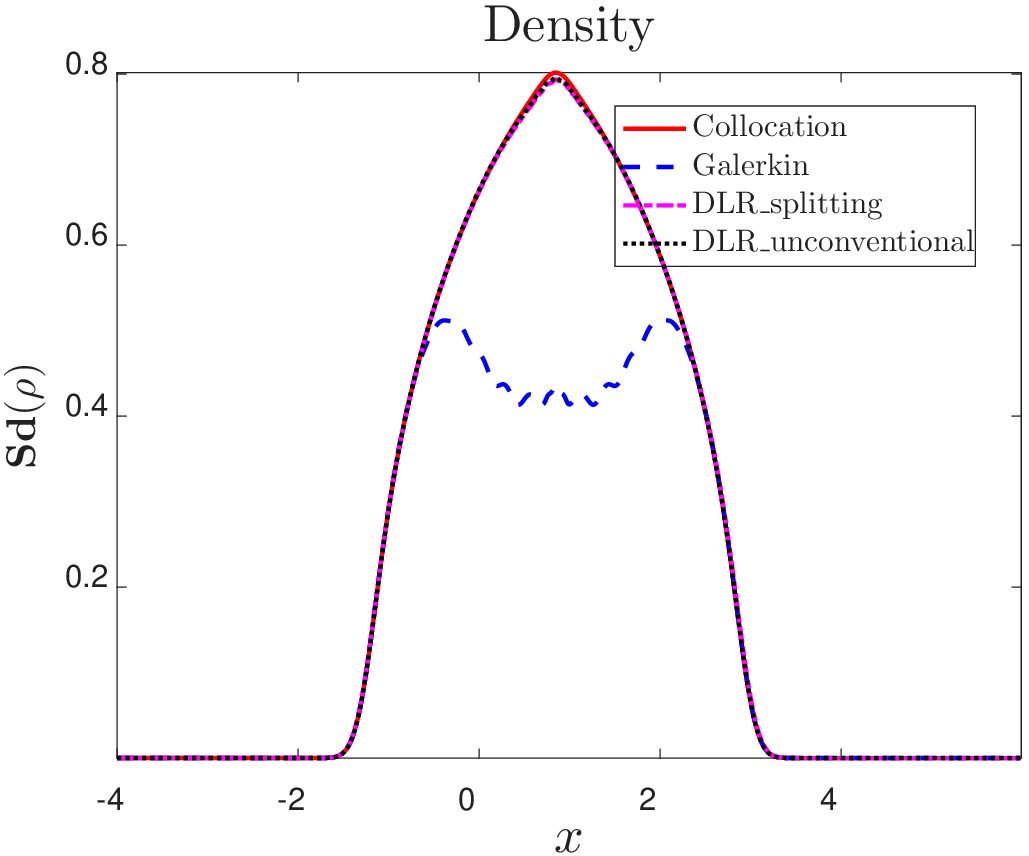}
        \caption{}
        \label{fig:subfig1_density_example1}
    \end{subfigure}
    \hfill
    \begin{subfigure}{0.32\textwidth}
        \centering
        \includegraphics[width=\textwidth]{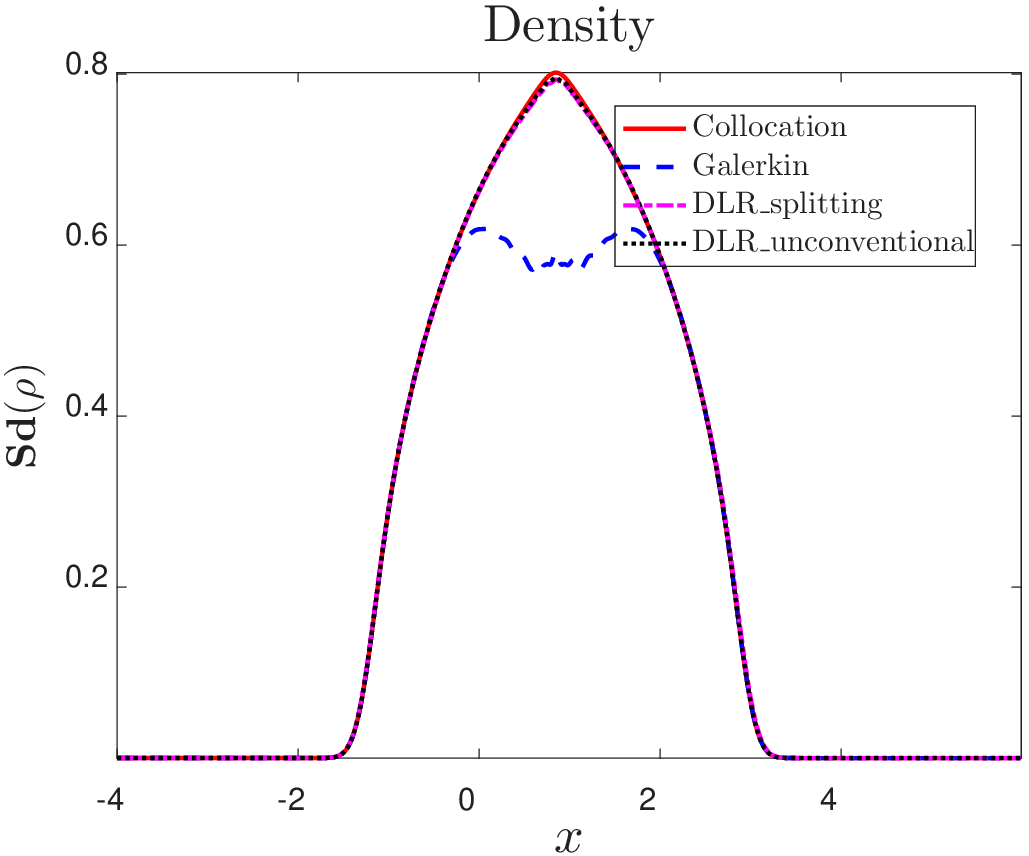}
        \caption{}
        \label{fig:subfig2_density_example1}
    \end{subfigure}
    \hfill
    \begin{subfigure}{0.32\textwidth}
        \centering
        \includegraphics[width=\textwidth]{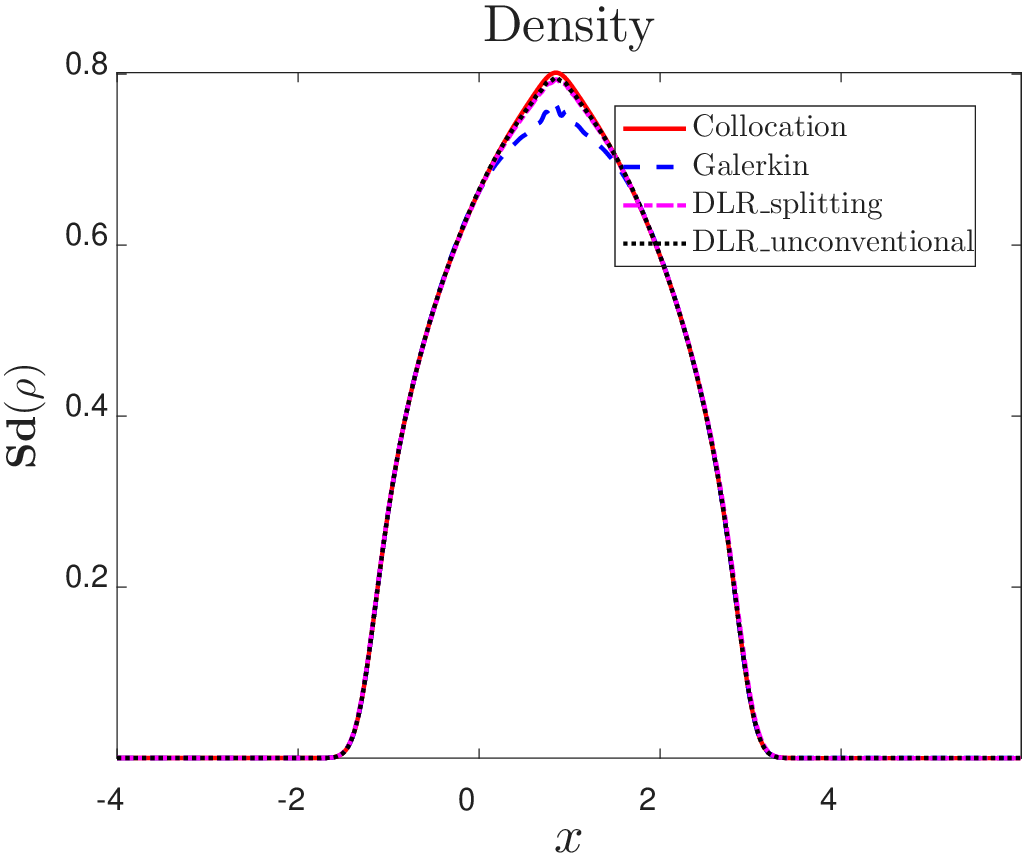}
        \caption{}
        \label{fig:subfig3_density_example1}
    \end{subfigure}
    \caption{Density in Example $1$. Left: degree of Galerkin $=20$; middle: degree of Galerkin $=30$; right: degree of Galerkin $=40$. The rank of the DLR method is $r=46$.}
    \label{fig:density_example1}
\end{figure}

\begin{figure}[H]
    \centering
    \begin{subfigure}{0.32\textwidth}
        \centering
        \includegraphics[width=\textwidth]{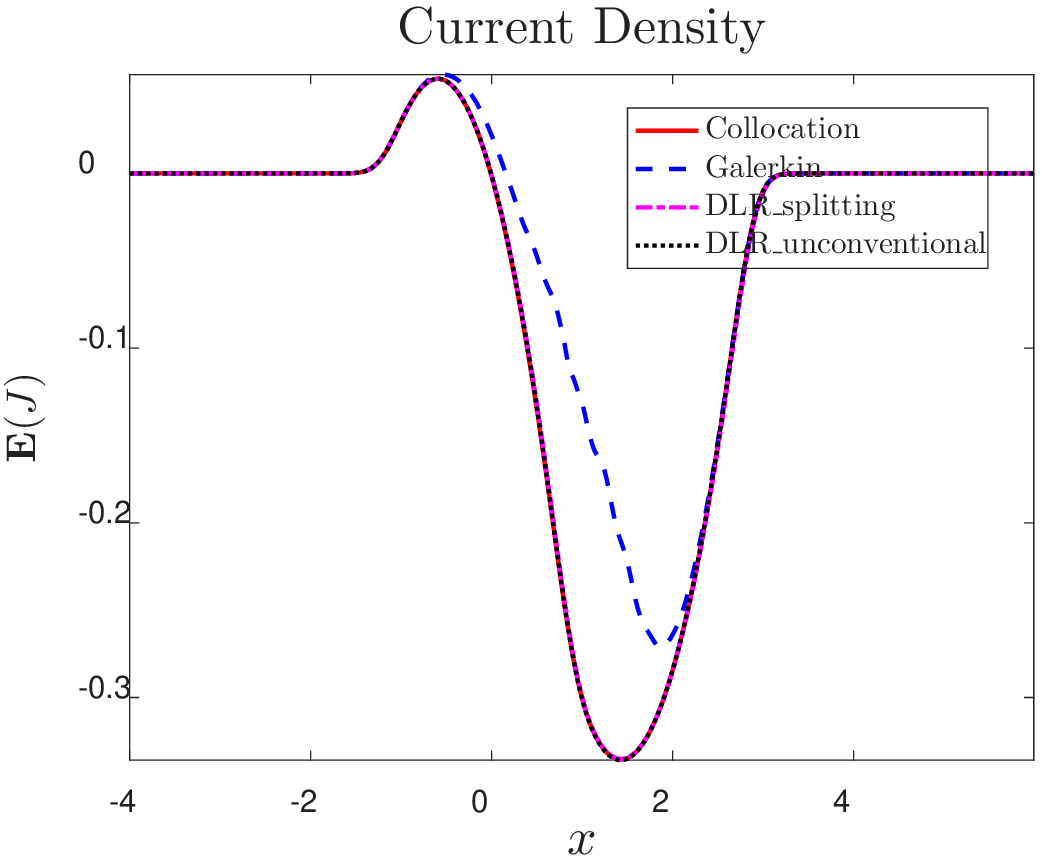}
        \caption{}
        \label{fig:subfig1_Jdensity_example1}
    \end{subfigure}
    \hfill
    \begin{subfigure}{0.32\textwidth}
        \centering
        \includegraphics[width=\textwidth]{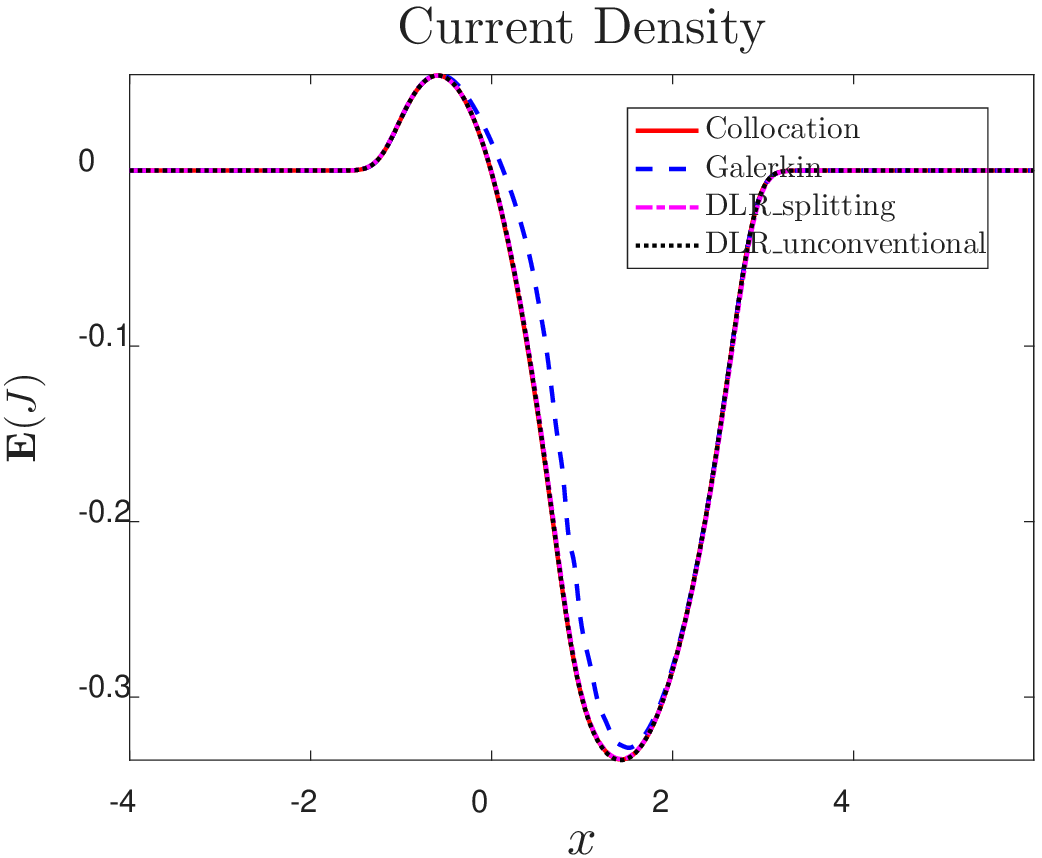}
        \caption{}
        \label{fig:subfig2_Jdensity_example1}
    \end{subfigure}
    \hfill
    \begin{subfigure}{0.32\textwidth}
        \centering
        \includegraphics[width=\textwidth]{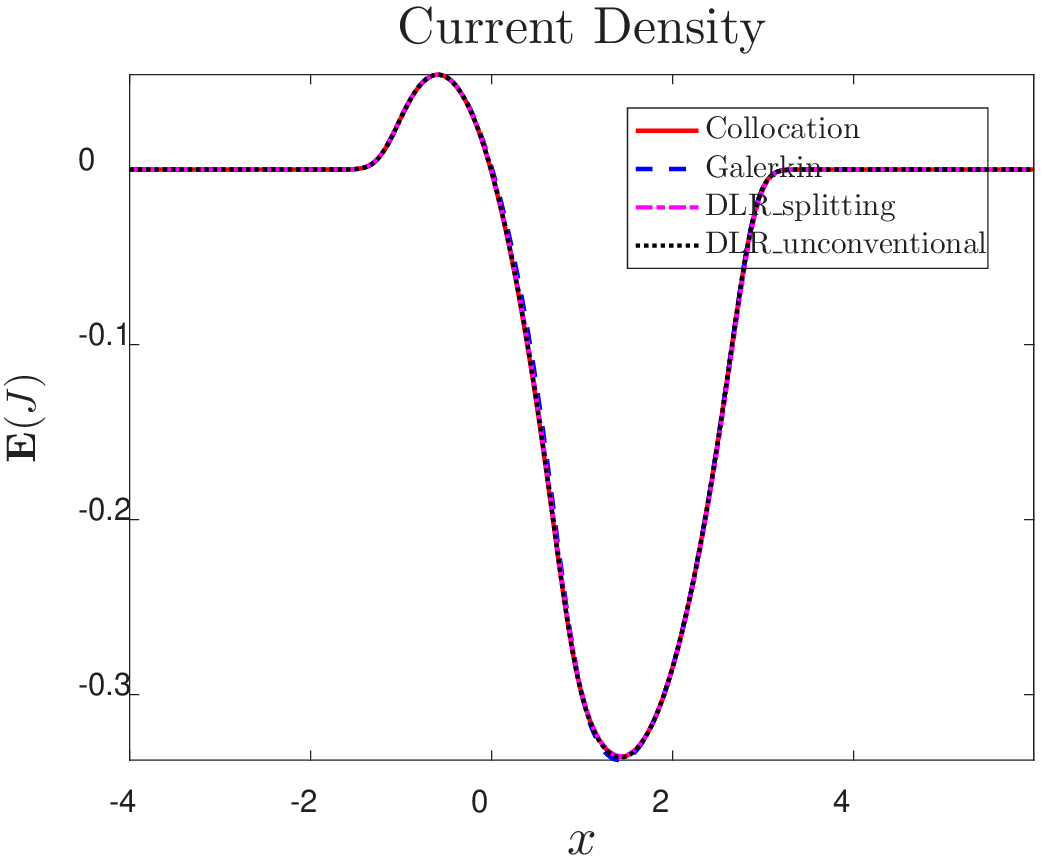}
        \caption{}
        \label{fig:subfig3_Jdensity_example1}
    \end{subfigure}
        \begin{subfigure}{0.32\textwidth}
        \centering
        \includegraphics[width=\textwidth]{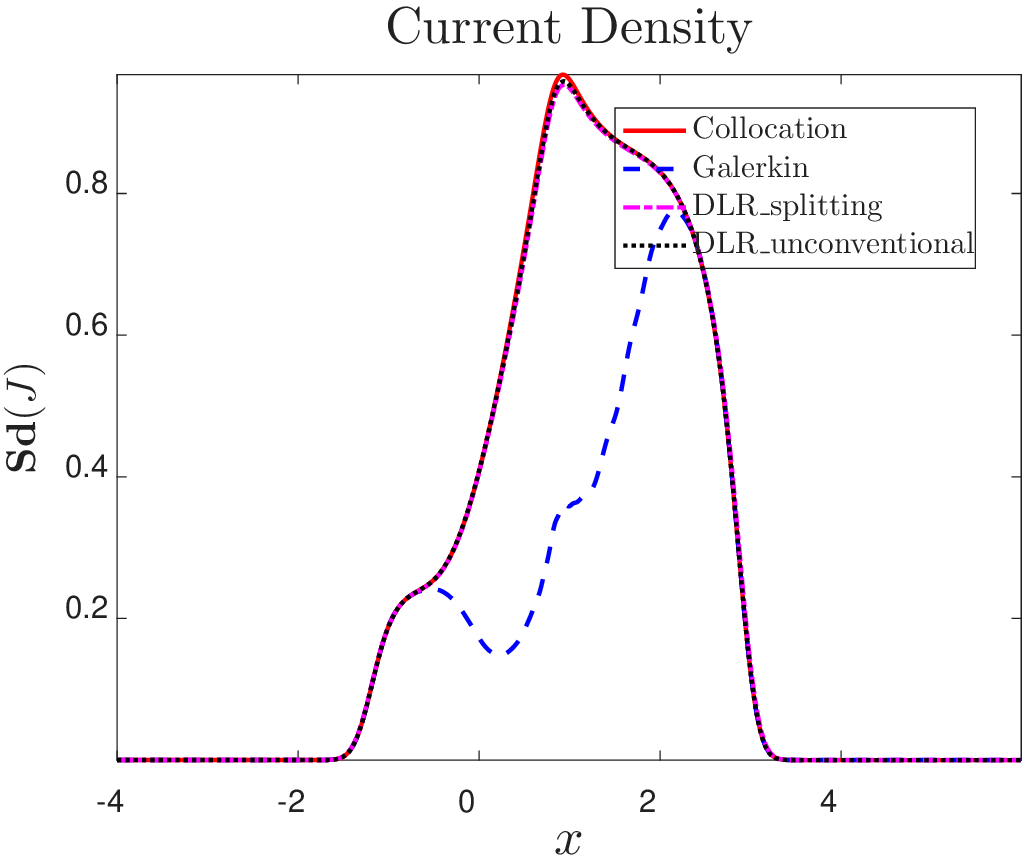}
        \caption{}
        \label{fig:subfig1_Jdensity_example1}
    \end{subfigure}
    \hfill
    \begin{subfigure}{0.32\textwidth}
        \centering
        \includegraphics[width=\textwidth]{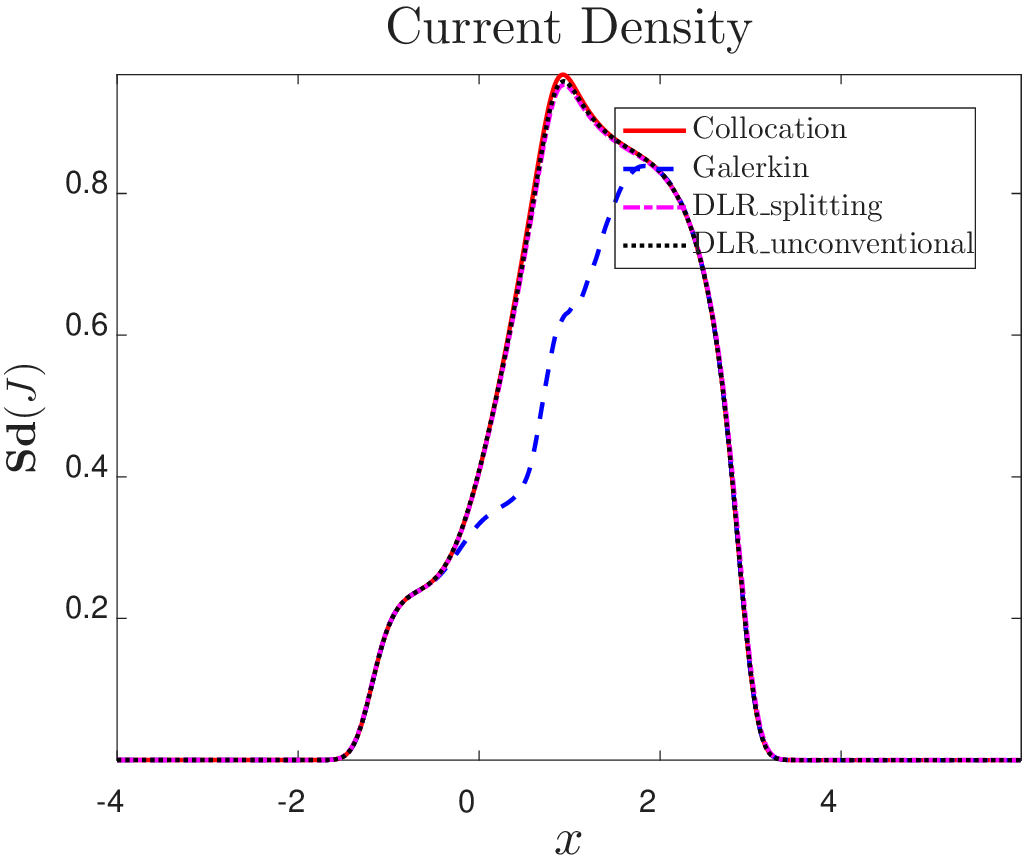}
        \caption{}
        \label{fig:subfig2_Jdensity_example1}
    \end{subfigure}
    \hfill
    \begin{subfigure}{0.32\textwidth}
        \centering
        \includegraphics[width=\textwidth]{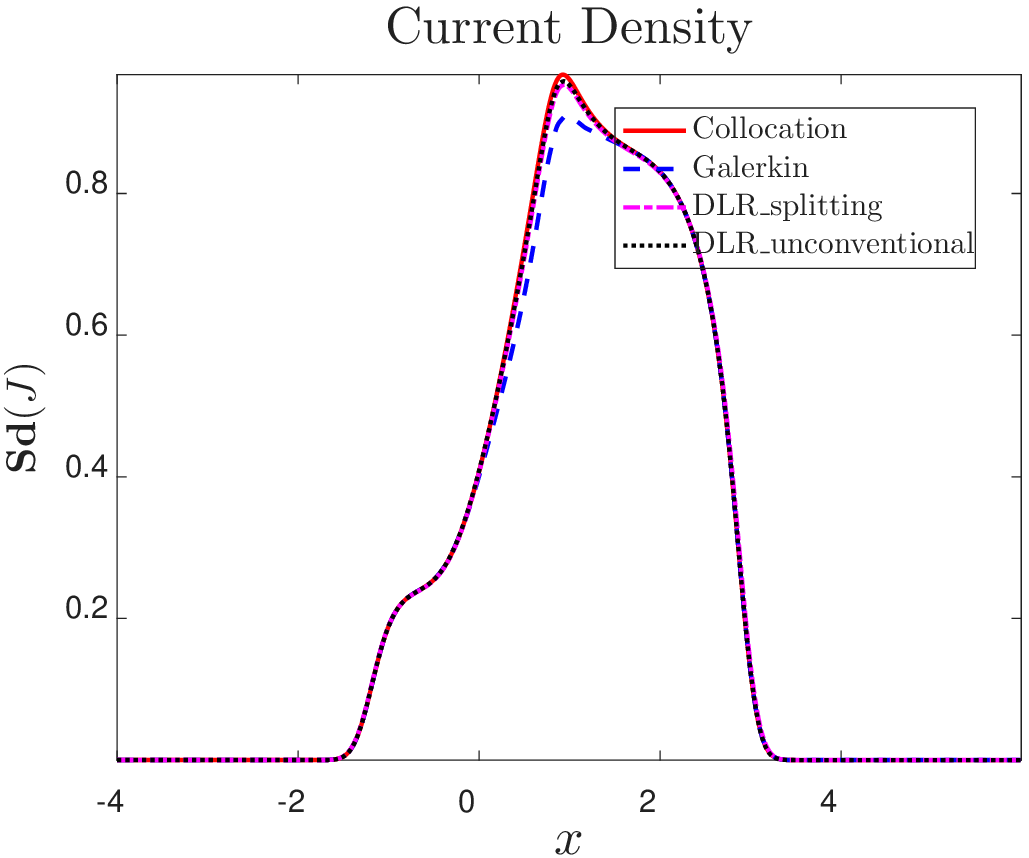}
        \caption{}
        \label{fig:subfig3_Jdensity_example1}
    \end{subfigure}
    \caption{Current density in Example $1$. Left: degree of Galerkin $=20$; middle: degree of Galerkin $=30$; right: degree of Galerkin $=40$. The rank of the DLR method is $r=46$.}
    \label{fig:Jdensity_example1}
\end{figure}

\textbf{Example 2 (Weakly Dependent Potential).} 
In this example, the potential includes a weak dependence on the random variables: $V(\xi)=1+0.1(\xi_1+2\xi_2)$. The remaining parameters are identical to Example 1: $\varepsilon=0.05, V(x)=0.6x^2, \tilde{p}=0.7(\xi_1-\xi_2), \tilde{q}=1+0.8(\xi_1+2\xi_2)$, with domain $[-4, 6]$ and discretization $N_x=1000, N_{\xi}=100$.

As shown in Table \ref{tab:error_comparison_example2}, the DLR method maintains its advantage over the stochastic Galerkin method. The proposed unconventional integrator consistently yields the lowest errors for both the wave function and the physical observables. The corresponding solutions are plotted in Figure \ref{fig:density_example2} (density) and Figure \ref{fig:Jdensity_example2} (current density). These figures confirm that the DLR approximation with rank $r=46$ accurately captures the mean and standard deviation of the solution dynamics, even with the stochastic perturbation in the potential.

\begin{table}[!htbp]
\centering
\caption{Error comparison for Example 2 with different ranks $r$.}
\label{tab:error_comparison_example2}
\renewcommand{\arraystretch}{1.2}
\scalebox{0.7}{
\begin{tabular}{l|c|c|c|c|c}
\toprule
\diagbox{Method}{Error} & $\psi$ & $\rho_{mean}$ & $\rho_{std}$ & $j_{mean}$ & $j_{std}$ \\
\midrule
Galerkin ($r=20$)   & $6.323 \times 10^{-1}$ & $4.662 \times 10^{-1}$ & $3.517 \times 10^{-1}$ & $3.659 \times 10^{-1}$ & $4.646 \times 10^{-1}$ \\
Galerkin ($r=30$)   & $3.698 \times 10^{-1}$ & $1.622 \times 10^{-1}$ & $1.701 \times 10^{-1}$ & $1.312 \times 10^{-1}$ & $2.276 \times 10^{-1}$ \\
Galerkin ($r=40$)   & $1.230 \times 10^{-1}$ & $1.372 \times 10^{-2}$ & $3.056 \times 10^{-2}$ & $1.706 \times 10^{-2}$ & $3.786 \times 10^{-2}$ \\
\midrule
DLR-splitting       & $1.000 \times 10^{-1}$ & $5.458 \times 10^{-3}$ & $8.046 \times 10^{-3}$ & $9.396 \times 10^{-3}$ & $1.277 \times 10^{-2}$ \\
Unconventional      & $\mathbf{4.518 \times 10^{-2}}$ & $\mathbf{2.743 \times 10^{-3}}$ & $\mathbf{5.313 \times 10^{-3}}$ & $\mathbf{4.369 \times 10^{-3}}$ & $\mathbf{7.539 \times 10^{-3}}$ \\
\bottomrule
\end{tabular}
}
\end{table}

\begin{figure}[H]
    \centering
        \begin{subfigure}{0.32\textwidth}
        \centering
        \includegraphics[width=\textwidth]{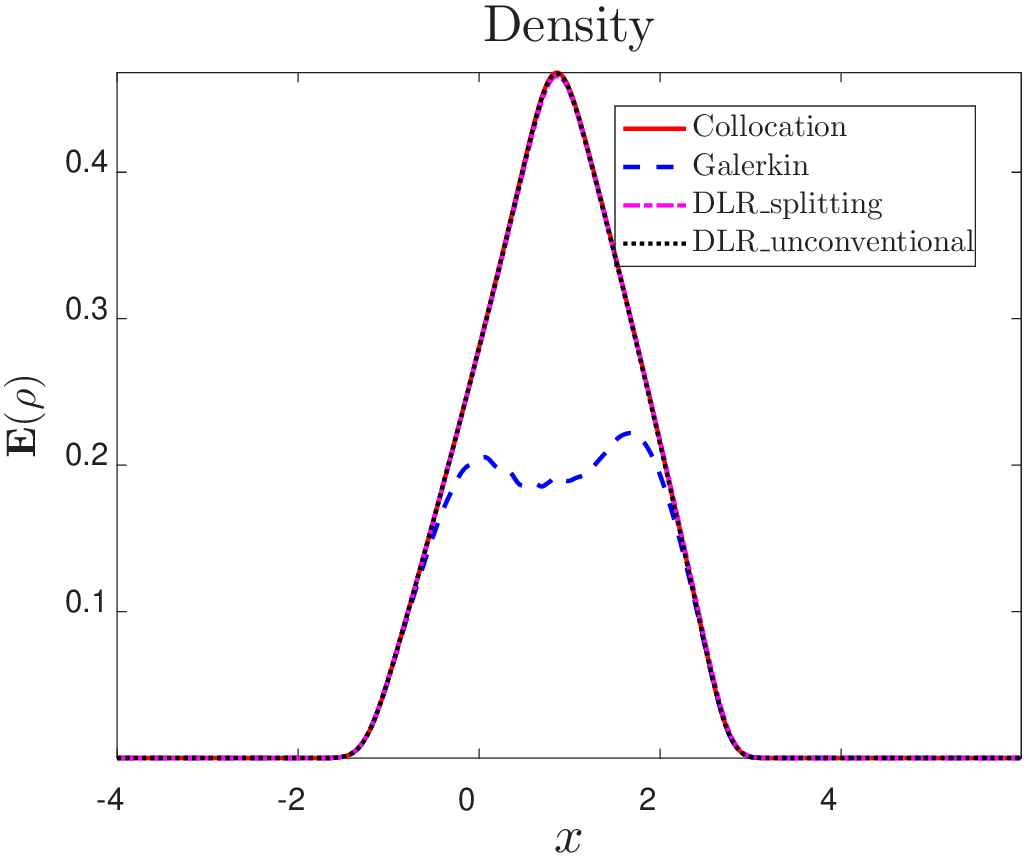}
        \caption{}
        \label{fig:subfig1_density_example1}
    \end{subfigure}
    \hfill
    \begin{subfigure}{0.32\textwidth}
        \centering
        \includegraphics[width=\textwidth]{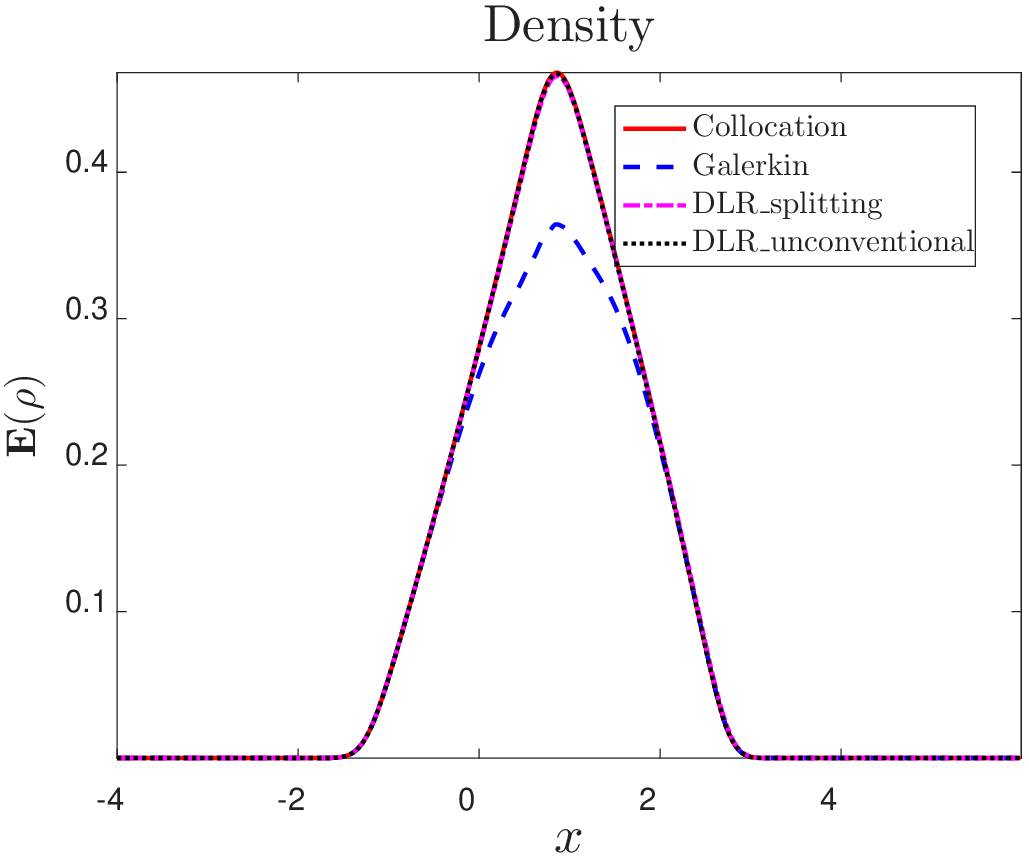}
        \caption{}
        \label{fig:subfig2_density_example1}
    \end{subfigure}
    \hfill
    \begin{subfigure}{0.32\textwidth}
        \centering
        \includegraphics[width=\textwidth]{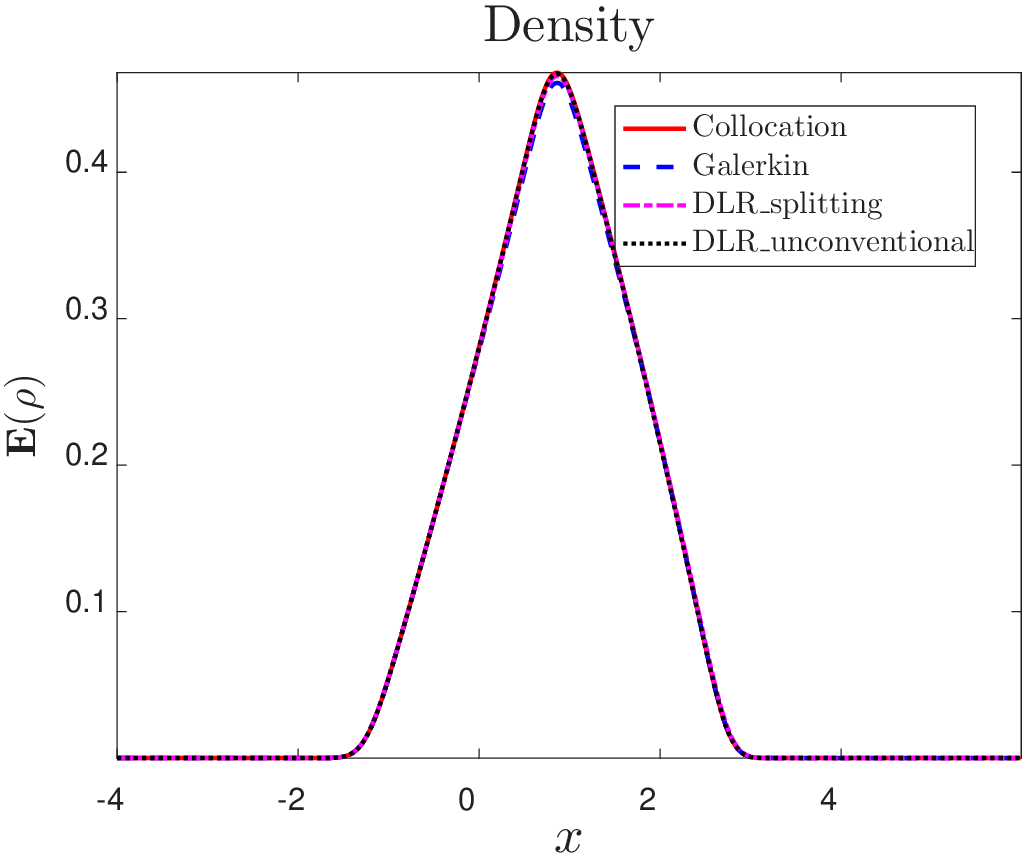}
        \caption{}
        \label{fig:subfig3_density_example1}
    \end{subfigure}
        \begin{subfigure}{0.32\textwidth}
        \centering
        \includegraphics[width=\textwidth]{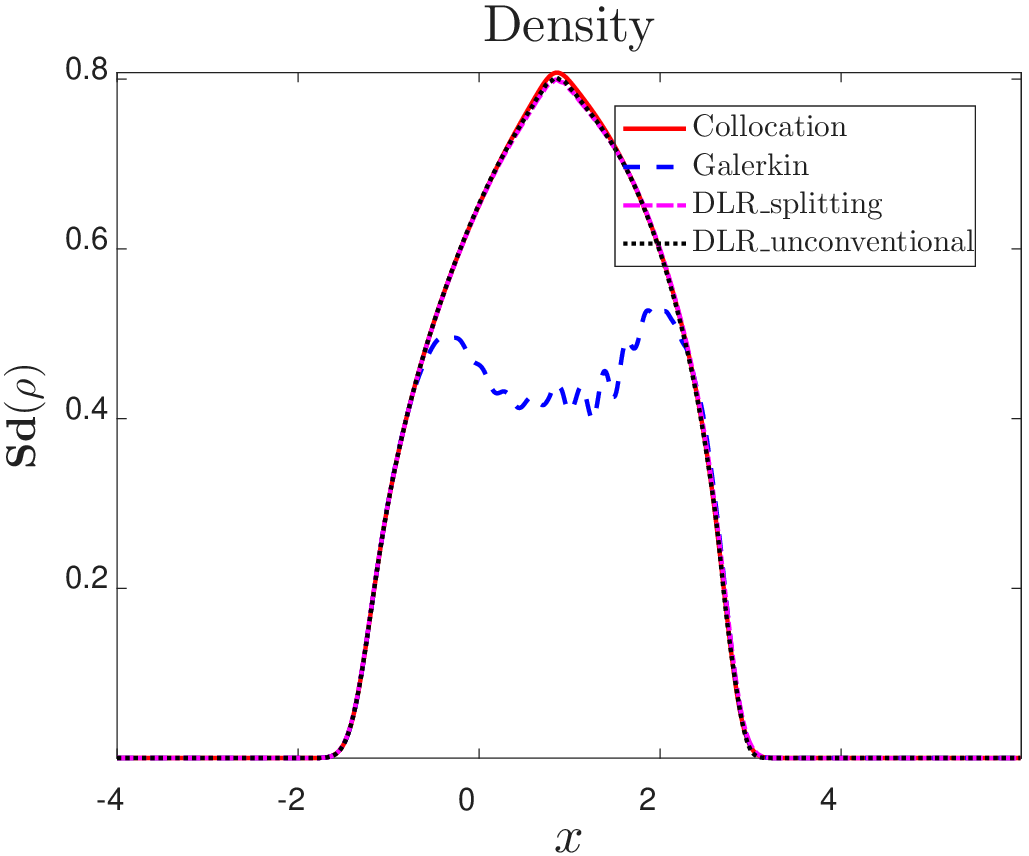}
        \caption{}
        \label{fig:subfig1_density_example1}
    \end{subfigure}
    \hfill
    \begin{subfigure}{0.32\textwidth}
        \centering
        \includegraphics[width=\textwidth]{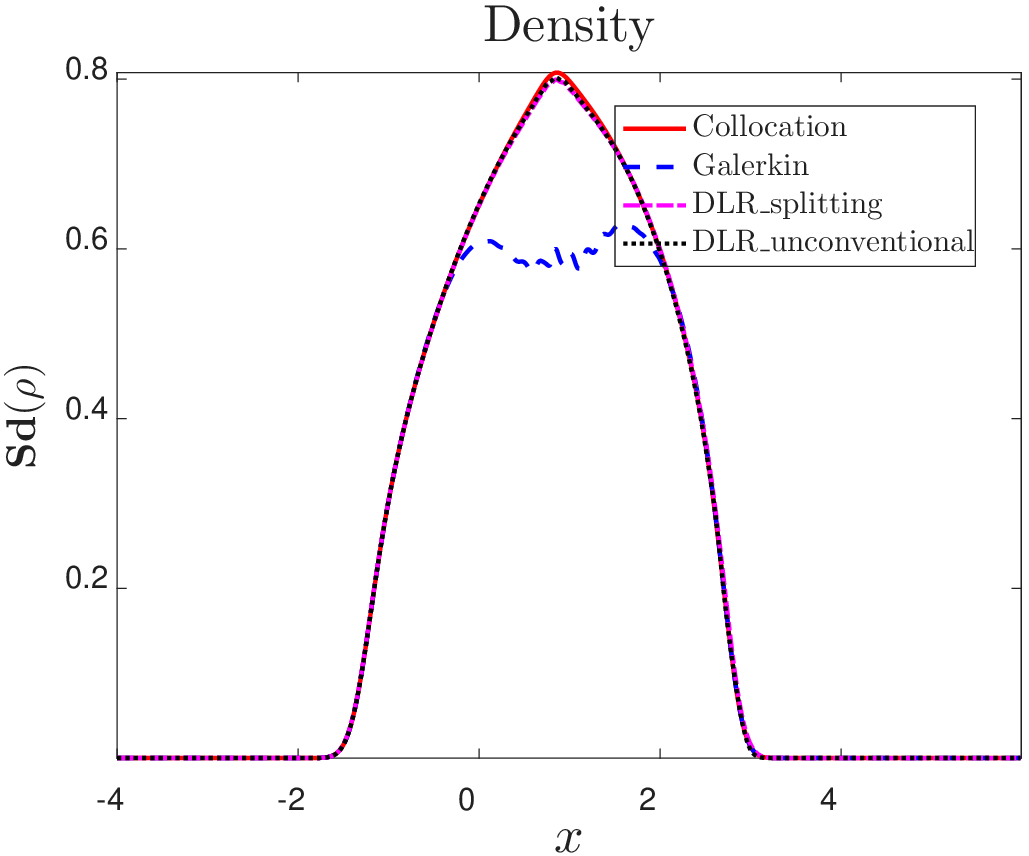}
        \caption{}
        \label{fig:subfig2_density_example1}
    \end{subfigure}
    \hfill
    \begin{subfigure}{0.32\textwidth}
        \centering
        \includegraphics[width=\textwidth]{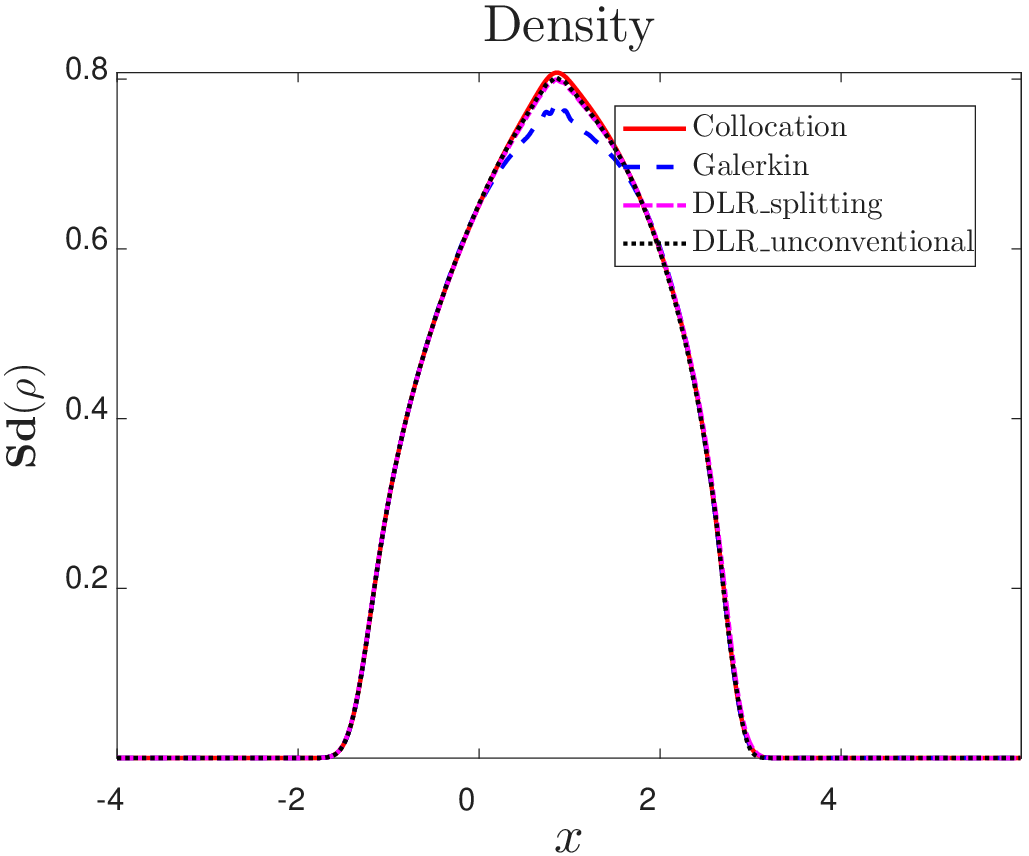}
        \caption{}
        \label{fig:subfig3_density_example1}
    \end{subfigure}
    \caption{Density in Example $2$. Left: degree of Galerkin $=20$; middle: degree of Galerkin $=30$; right: degree of Galerkin $=40$. The rank of the DLR method is $r=46$.}
    \label{fig:density_example2}
\end{figure}

\begin{figure}[H]
    \centering
    \begin{subfigure}{0.32\textwidth}
        \centering
        \includegraphics[width=\textwidth]{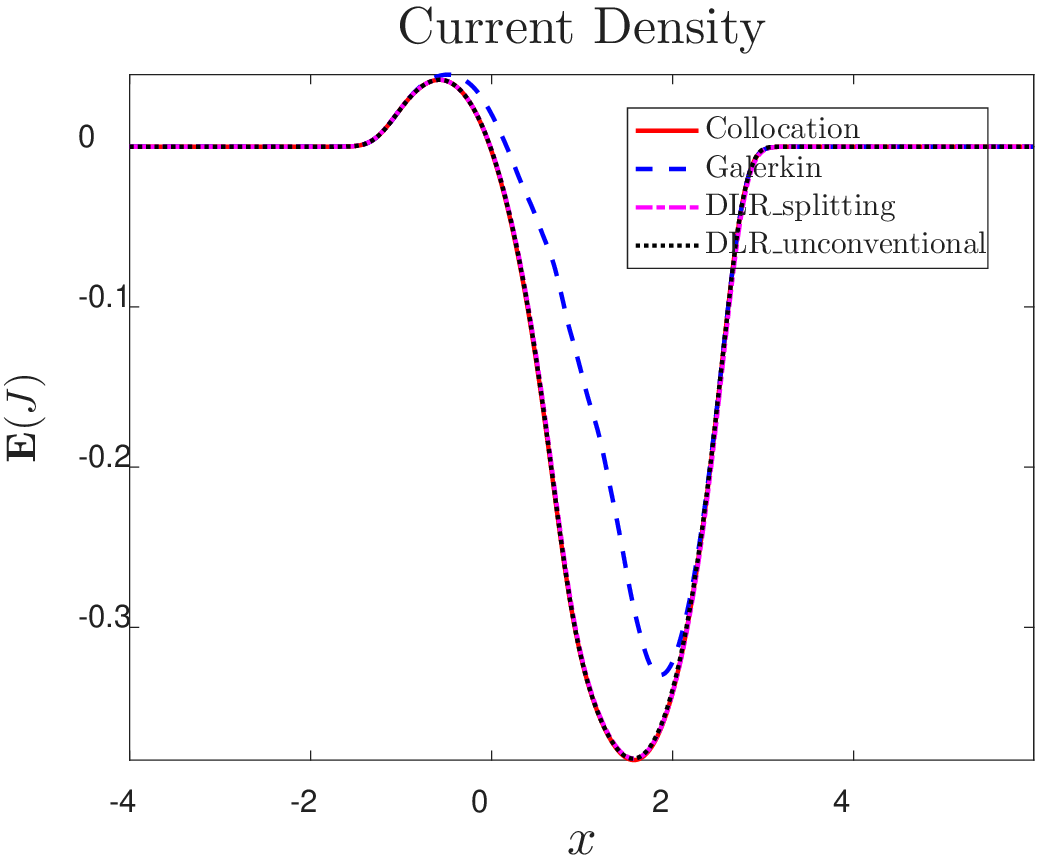}
        \caption{}
        \label{fig:subfig1_Jdensity_example1}
    \end{subfigure}
    \hfill
    \begin{subfigure}{0.32\textwidth}
        \centering
        \includegraphics[width=\textwidth]{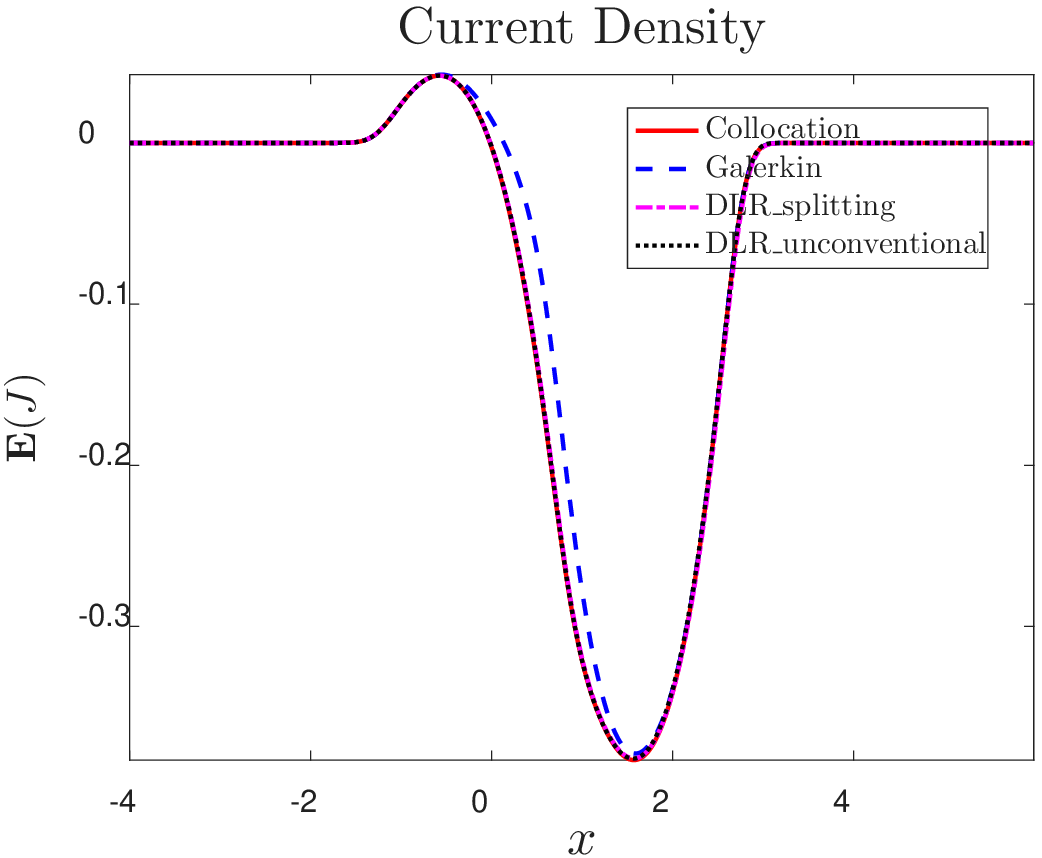}
        \caption{}
        \label{fig:subfig2_Jdensity_example1}
    \end{subfigure}
    \hfill
    \begin{subfigure}{0.32\textwidth}
        \centering
        \includegraphics[width=\textwidth]{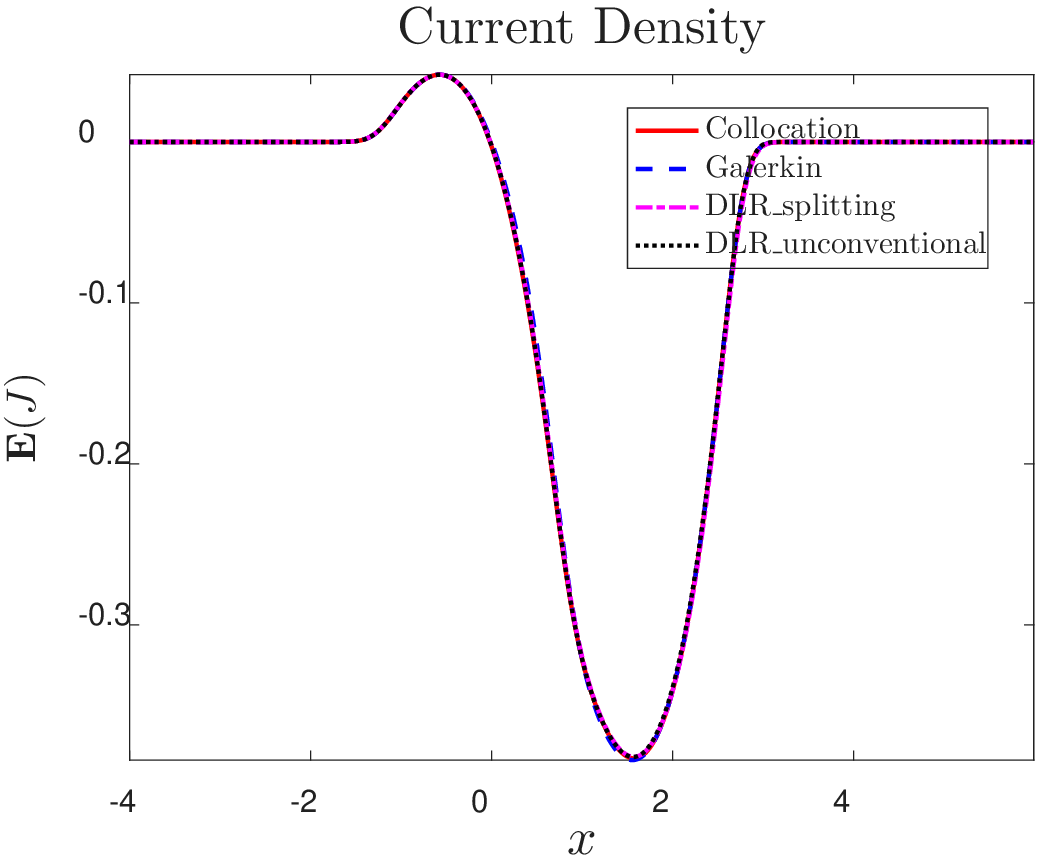}
        \caption{}
        \label{fig:subfig3_Jdensity_example1}
    \end{subfigure}
        \begin{subfigure}{0.32\textwidth}
        \centering
        \includegraphics[width=\textwidth]{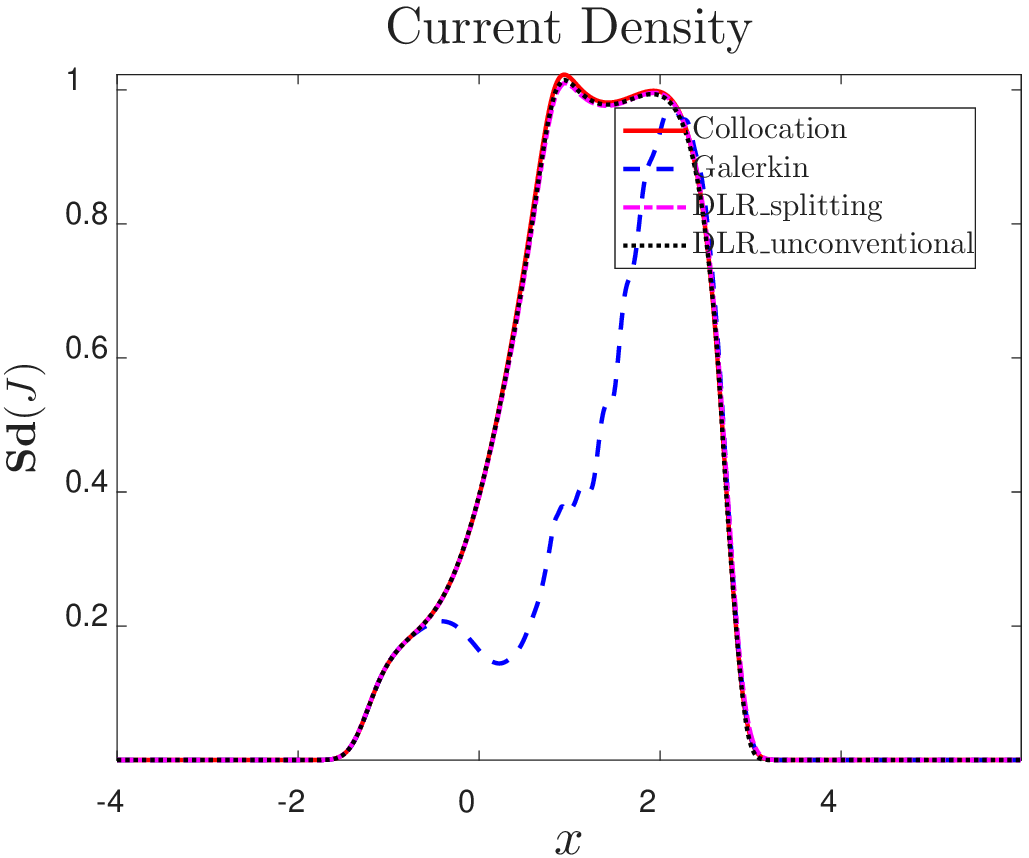}
        \caption{}
        \label{fig:subfig1_Jdensity_example1}
    \end{subfigure}
    \hfill
    \begin{subfigure}{0.32\textwidth}
        \centering
        \includegraphics[width=\textwidth]{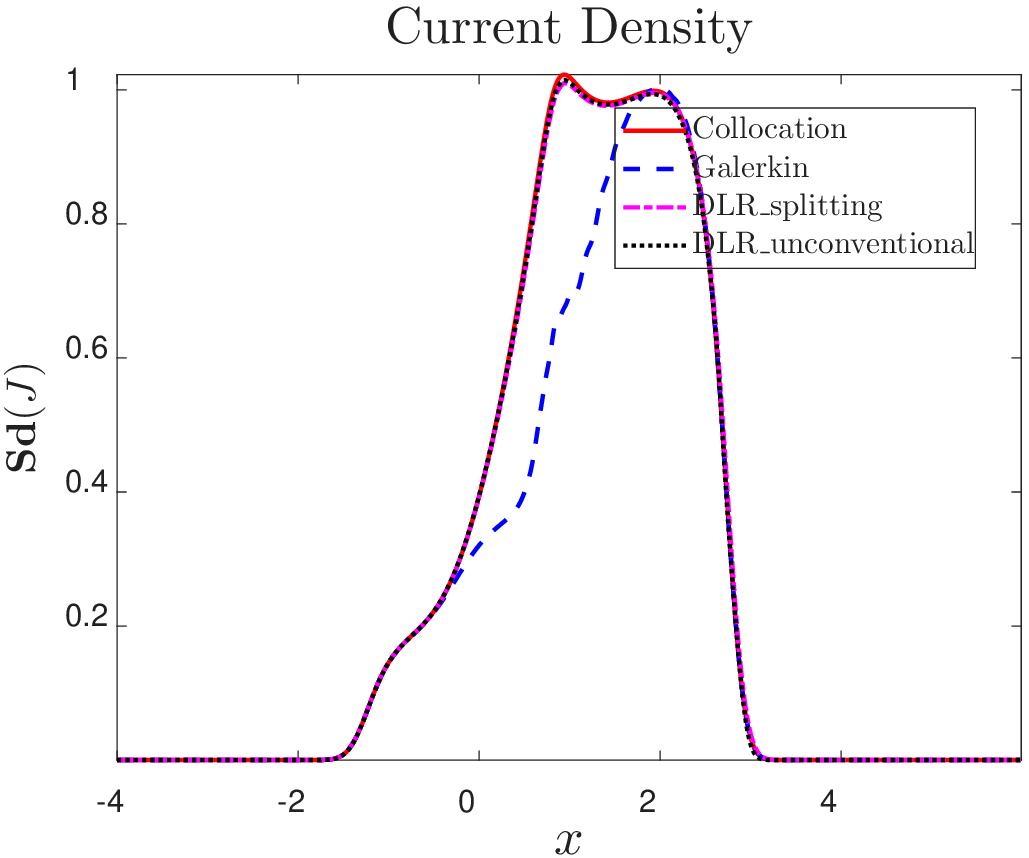}
        \caption{}
        \label{fig:subfig2_Jdensity_example1}
    \end{subfigure}
    \hfill
    \begin{subfigure}{0.32\textwidth}
        \centering
        \includegraphics[width=\textwidth]{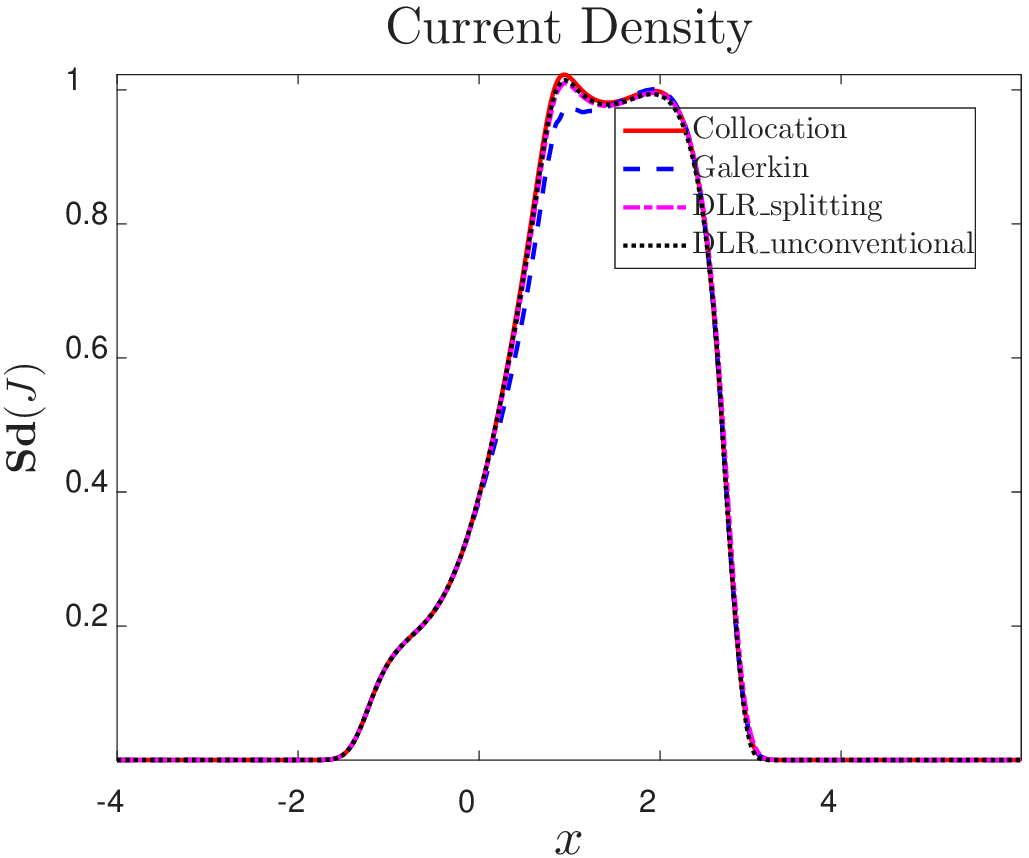}
        \caption{}
        \label{fig:subfig3_Jdensity_example1}
    \end{subfigure}
    \caption{Current density in Example $2$. Left: degree of Galerkin $=20$; middle: degree of Galerkin $=30$; right: degree of Galerkin $=40$. The rank of the DLR method is $r=46$.}
    \label{fig:Jdensity_example2}
\end{figure}

\textbf{Example 3 (Sensitivity Analysis).} We investigate the impact of the randomness strength $\gamma$ in the potential $V(\xi)=1+\gamma(\xi_1+2\xi_2)$, with $\gamma \in \{0.3, 0.5, 0.7\}$.

Table \ref{tab:error_comparison_example3} reveals that the approximation error scales with $\gamma$. Physically, a larger $\gamma$ implies stronger coupling between the spatial and random modes, which slows the decay of the singular values and increases the truncation error for a fixed-rank approximation, as visualized in Figures \ref{fig:density_example3} and \ref{fig:Jdensity_example3}. These results elucidate a fundamental mechanism: the initial data dictates the baseline rank, while the potential function drives its temporal evolution. In regimes where the initial randomness is high but the potential's randomness is low, the rank remains stable, allowing the DLR algorithm to achieve high accuracy with minimal storage costs. Conversely, if the initial randomness is low but the potential undergoes strong random perturbations, the effective rank grows rapidly, rendering the fixed-rank ansatz insufficient. Addressing this limitation via rank-adaptive strategies will be the subject of our follow-up work.

\begin{table}[!htbp]
\centering
\caption{Error comparison for different values of $\gamma_1$.}
\label{tab:error_comparison_example3}
\renewcommand{\arraystretch}{1.2}
\scalebox{0.6}{
\begin{tabular}{c|l|c|c|c|c|c}
\toprule
\diagbox{Parameter}{Method} & Method & $\psi$ & $\rho_{mean}$ & $j_{mean}$ & $\rho_{std}$ & $j_{std}$ \\
\midrule
\multirow{3}{*}{$\gamma_1 = 0.3$} & Galerkin & $7.476 \times 10^{-1}$ & $4.634 \times 10^{-1}$ & $3.301 \times 10^{-1}$ & $3.473 \times 10^{-1}$ & $4.532 \times 10^{-1}$ \\
 & DLR-splitting & $1.795 \times 10^{-1}$ & $\mathbf{7.231 \times 10^{-3}}$ & $\mathbf{1.449 \times 10^{-2}}$ & $\mathbf{1.684 \times 10^{-2}}$ & $\mathbf{2.119 \times 10^{-2}}$ \\
 & Unconventional & $\mathbf{1.196 \times 10^{-1}}$ & $1.628 \times 10^{-2}$ & $3.442 \times 10^{-2}$ & $2.469 \times 10^{-2}$ & $3.397 \times 10^{-2}$ \\
\midrule
\multirow{3}{*}{$\gamma_1 = 0.5$} & Galerkin & $8.441 \times 10^{-1}$ & $4.585 \times 10^{-1}$ & $2.991 \times 10^{-1}$ & $3.280 \times 10^{-1}$ & $4.238 \times 10^{-1}$ \\
 & DLR-splitting & $3.230 \times 10^{-1}$ & $\mathbf{2.583 \times 10^{-2}}$ & $\mathbf{3.511 \times 10^{-2}}$ & $\mathbf{5.228 \times 10^{-2}}$ & $\mathbf{5.935 \times 10^{-2}}$ \\
 & Unconventional & $\mathbf{2.798 \times 10^{-1}}$ & $4.126 \times 10^{-2}$ & $9.010 \times 10^{-2}$ & $8.054 \times 10^{-2}$ & $8.187 \times 10^{-2}$ \\
\midrule
\multirow{3}{*}{$\gamma_1 = 0.7$} & Galerkin & $9.056 \times 10^{-1}$ & $4.509 \times 10^{-1}$ & $2.729 \times 10^{-1}$ & $3.401 \times 10^{-1}$ & $5.446 \times 10^{-1}$ \\
 & DLR-splitting & $4.673 \times 10^{-1}$ & $9.736 \times 10^{-2}$ & $\mathbf{1.649 \times 10^{-1}}$ & $\mathbf{7.335 \times 10^{-2}}$ & $\mathbf{5.868 \times 10^{-2}}$ \\
 & Unconventional & $\mathbf{4.477 \times 10^{-1}}$ & $\mathbf{7.767 \times 10^{-2}}$ & $1.767 \times 10^{-1}$ & $1.154 \times 10^{-1}$ & $1.516 \times 10^{-1}$ \\
\bottomrule
\end{tabular}
}
\end{table}

\begin{figure}[H]
    \centering
    \begin{subfigure}{0.32\textwidth}
        \centering
        \includegraphics[width=\textwidth]{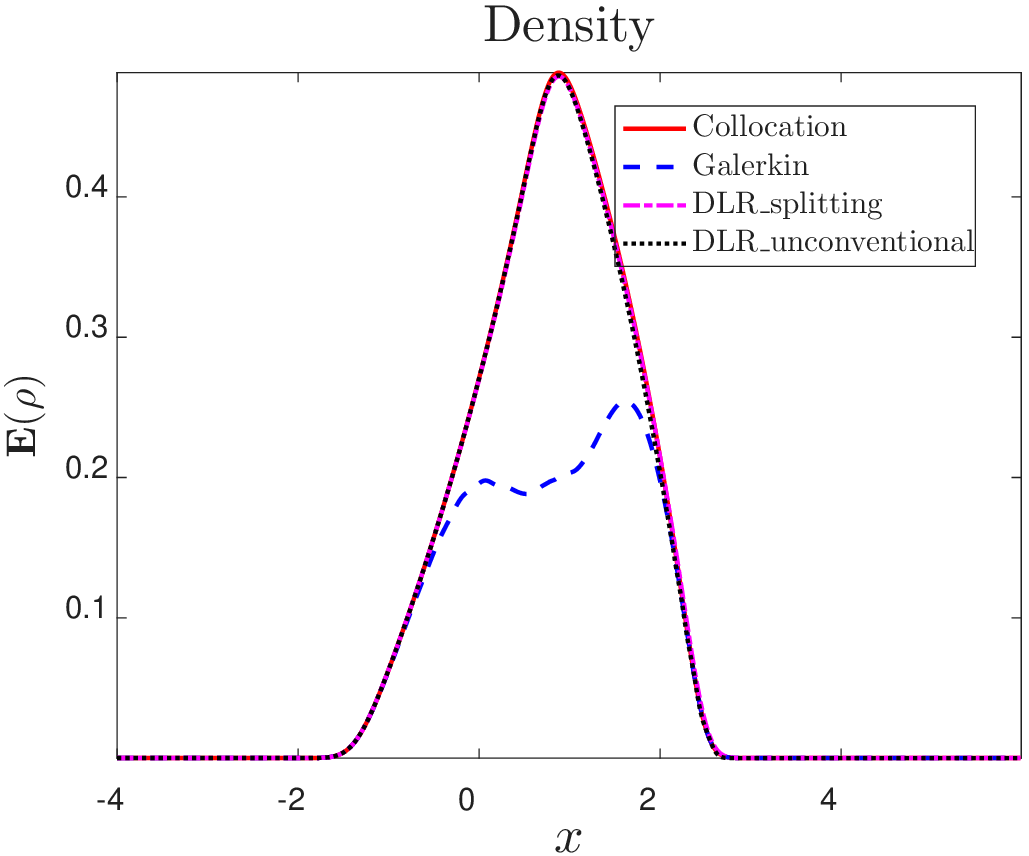}
        \caption{}
        \label{fig:subfig1_example3}
    \end{subfigure}
    \hfill
    \begin{subfigure}{0.32\textwidth}
        \centering
        \includegraphics[width=\textwidth]{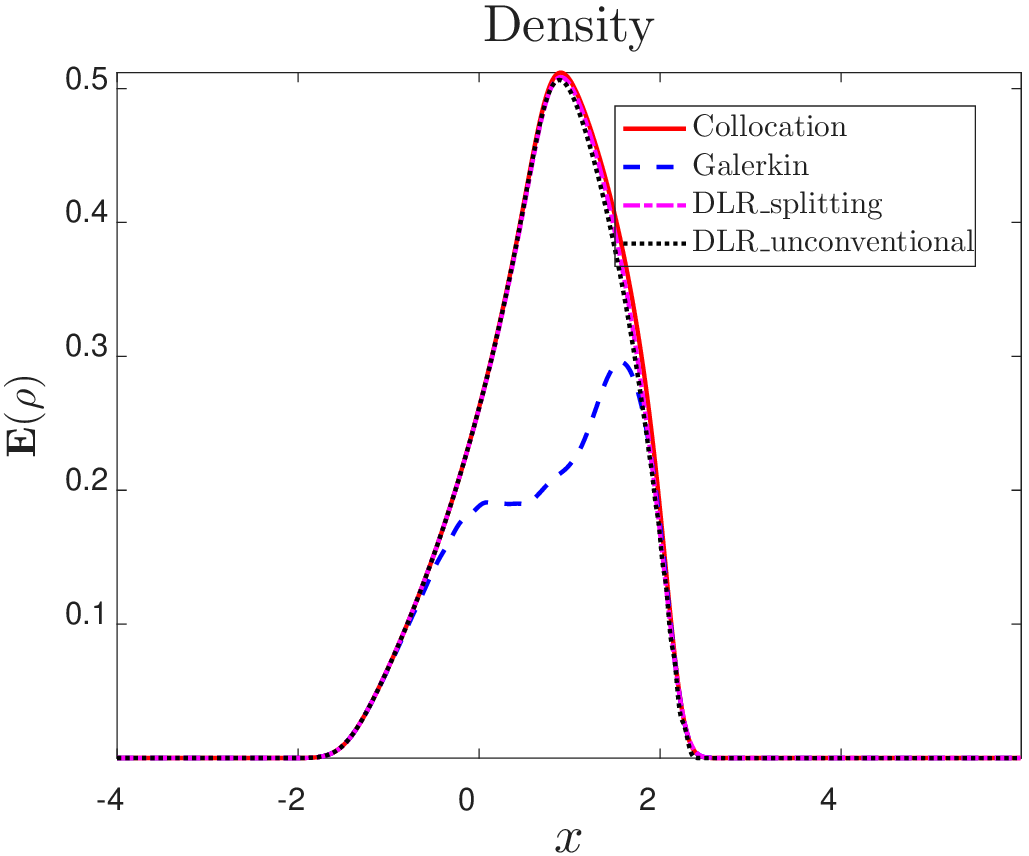}
        \caption{}
        \label{fig:subfig1_example3}
    \end{subfigure}
    \hfill
    \begin{subfigure}{0.32\textwidth}
        \centering
        \includegraphics[width=\textwidth]{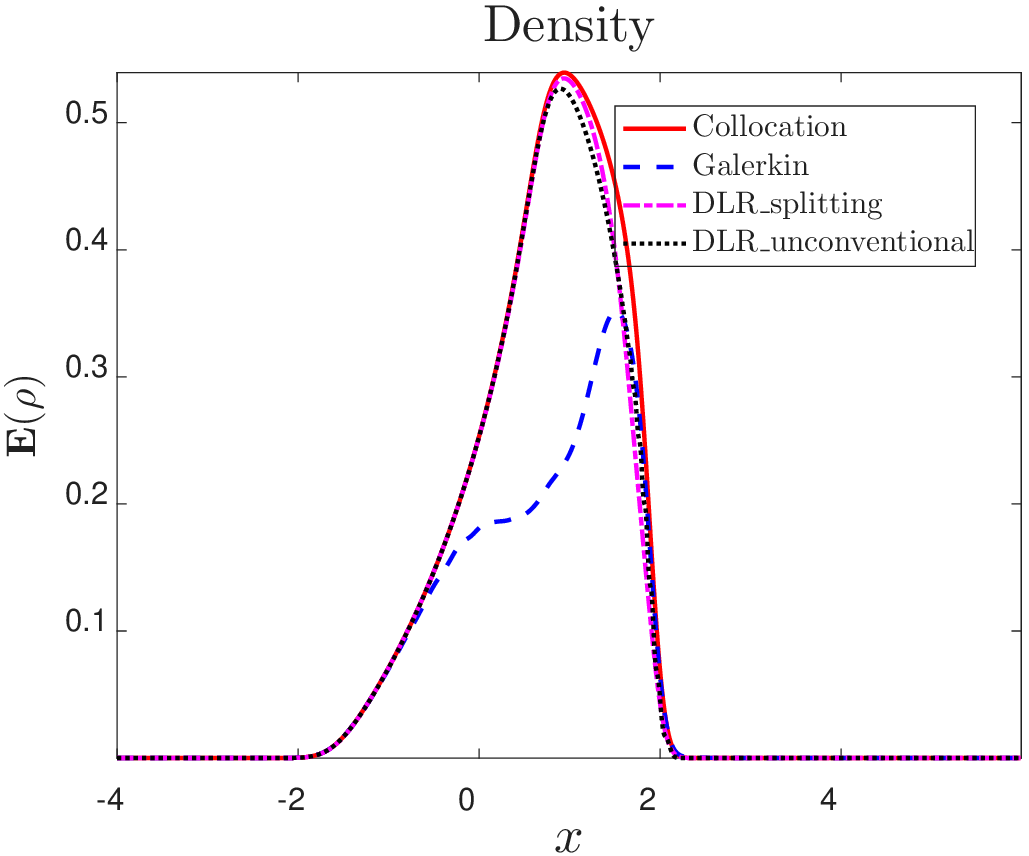}
        \caption{}
        \label{fig:subfig1_example3}
    \end{subfigure}
    \vspace{0.5cm} 
    \begin{subfigure}{0.32\textwidth}
        \centering
        \includegraphics[width=\textwidth]{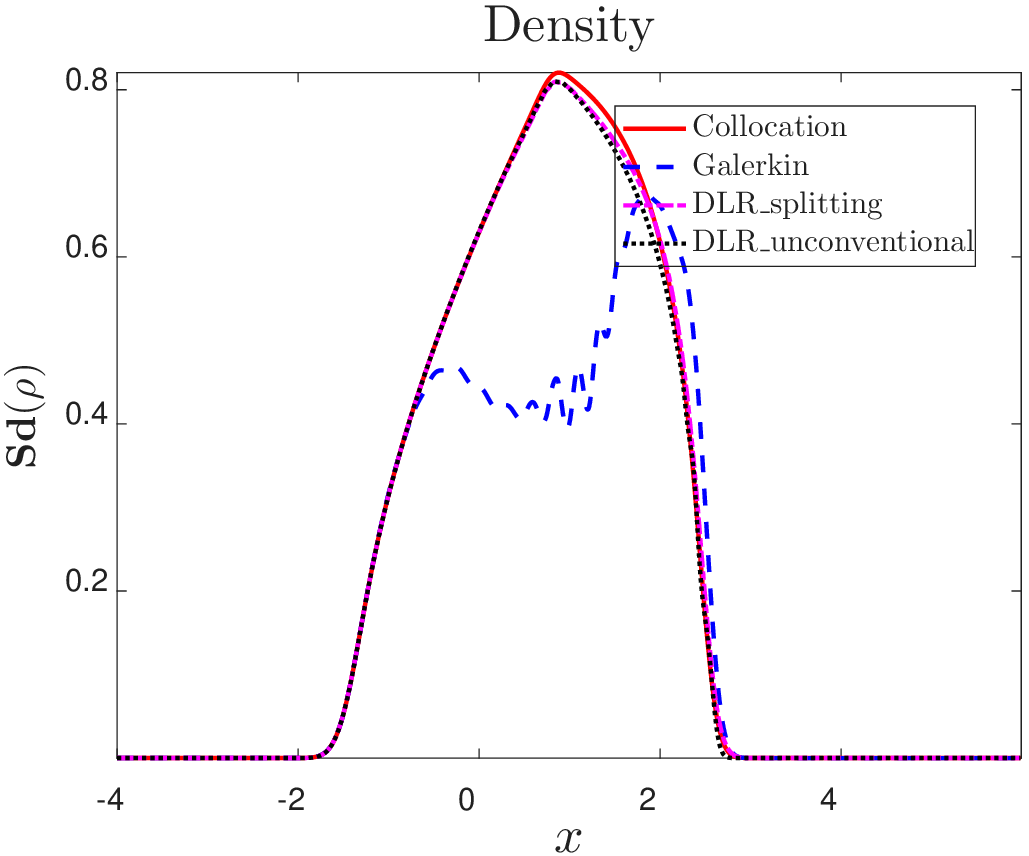}
        \caption{}
        \label{fig:subfig1_example3}
    \end{subfigure}
    \hfill
    \begin{subfigure}{0.32\textwidth}
        \centering
        \includegraphics[width=\textwidth]{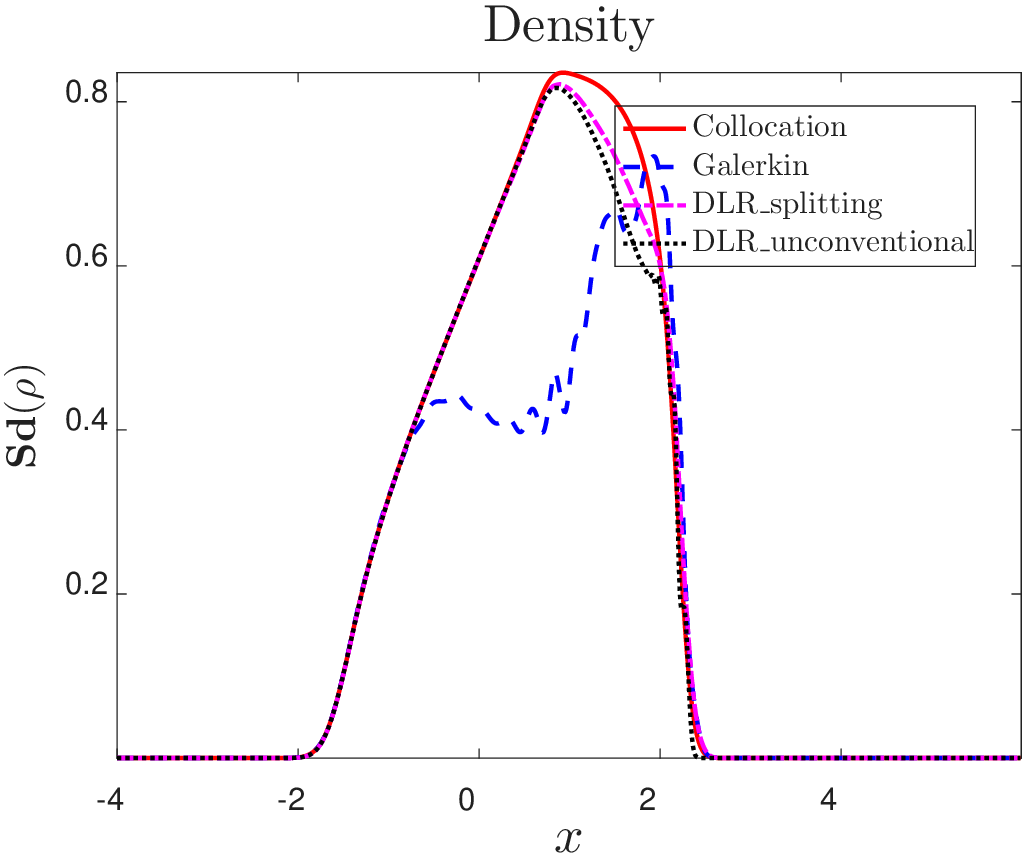}
        \caption{}
        \label{fig:subfig1_example3}
    \end{subfigure}
    \hfill
    \begin{subfigure}{0.32\textwidth}
        \centering
        \includegraphics[width=\textwidth]{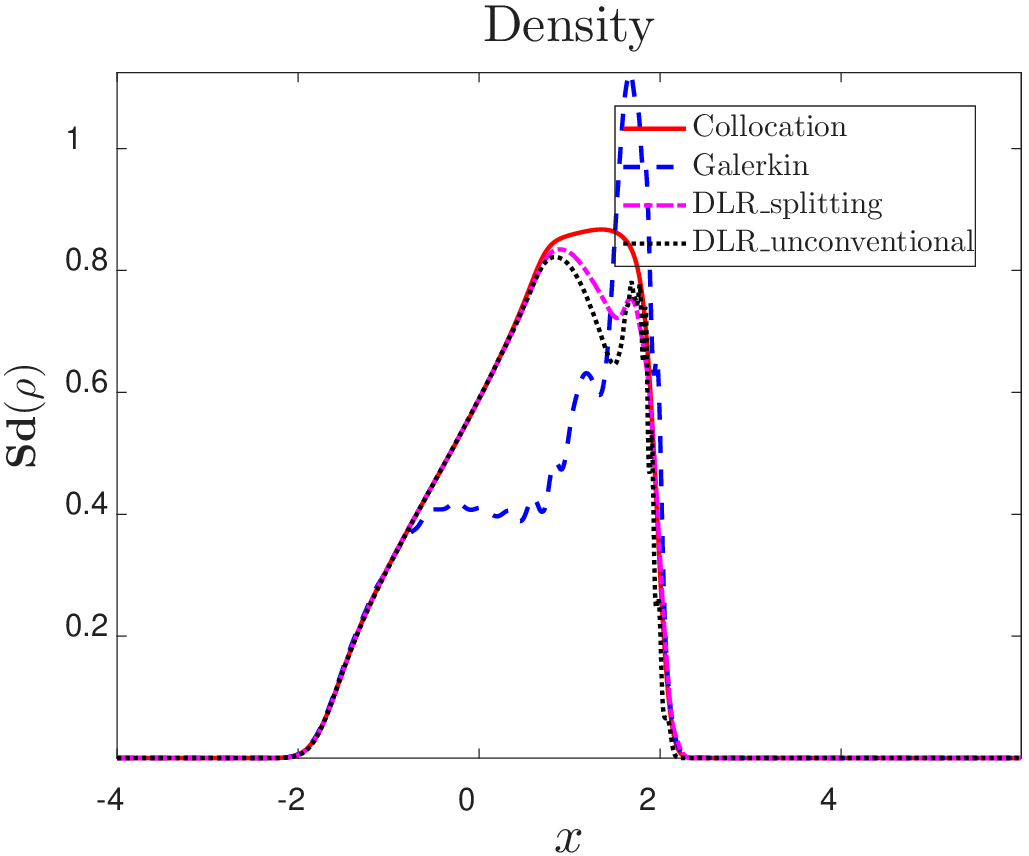}
        \caption{}
        \label{fig:subfig1_example3}
    \end{subfigure}

    \caption{The density solution in example 3. From top left to bottom right, $\gamma=0.3, 0.5, 0.7, 0.9, 1.1, 1.3$.}
    \label{fig:density_example3}
\end{figure}

\begin{figure}[H]
    \centering
    \begin{subfigure}{0.32\textwidth}
        \centering
        \includegraphics[width=\textwidth]{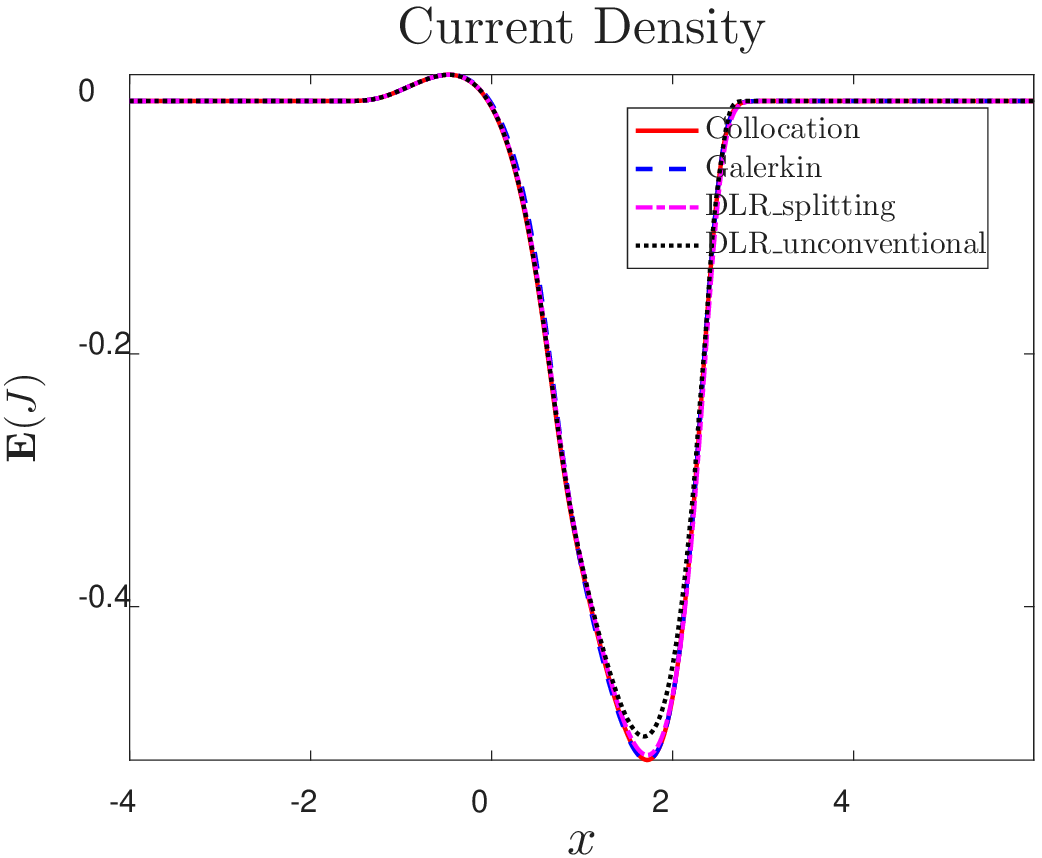}
        \caption{}
        \label{fig:subfig1_Jdensity_example3}
    \end{subfigure}
    \hfill
    \begin{subfigure}{0.32\textwidth}
        \centering
        \includegraphics[width=\textwidth]{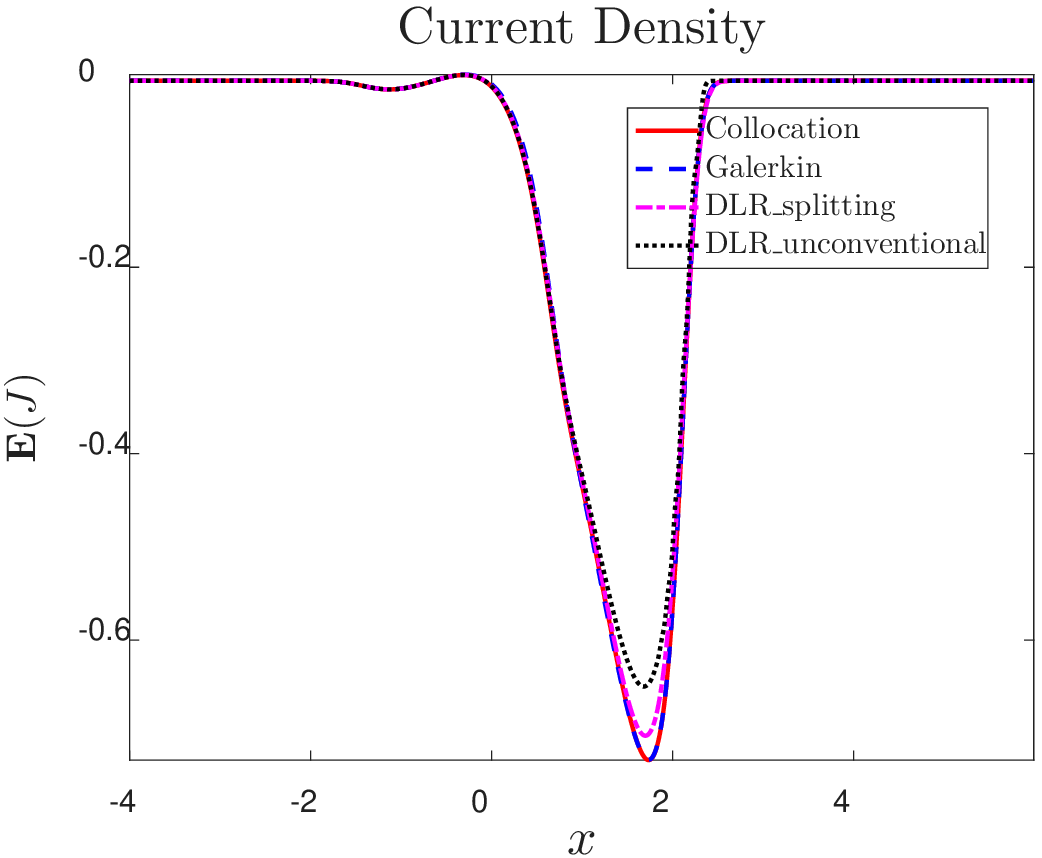}
        \caption{}
        \label{fig:subfig2_Jdensity_example3}
    \end{subfigure}
    \hfill
    \begin{subfigure}{0.32\textwidth}
        \centering
        \includegraphics[width=\textwidth]{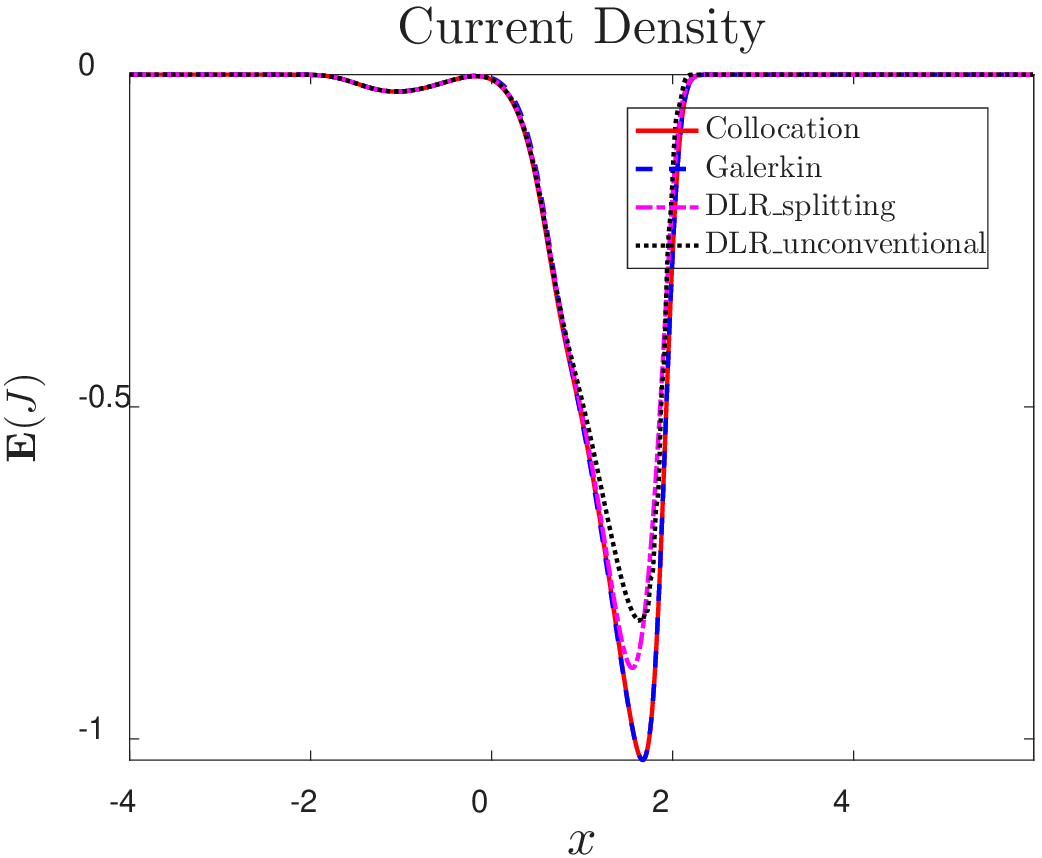}
        \caption{}
        \label{fig:subfig3_Jdensity_example3}
    \end{subfigure}

    \vspace{0.5cm} 
    \begin{subfigure}{0.32\textwidth}
        \centering
        \includegraphics[width=\textwidth]{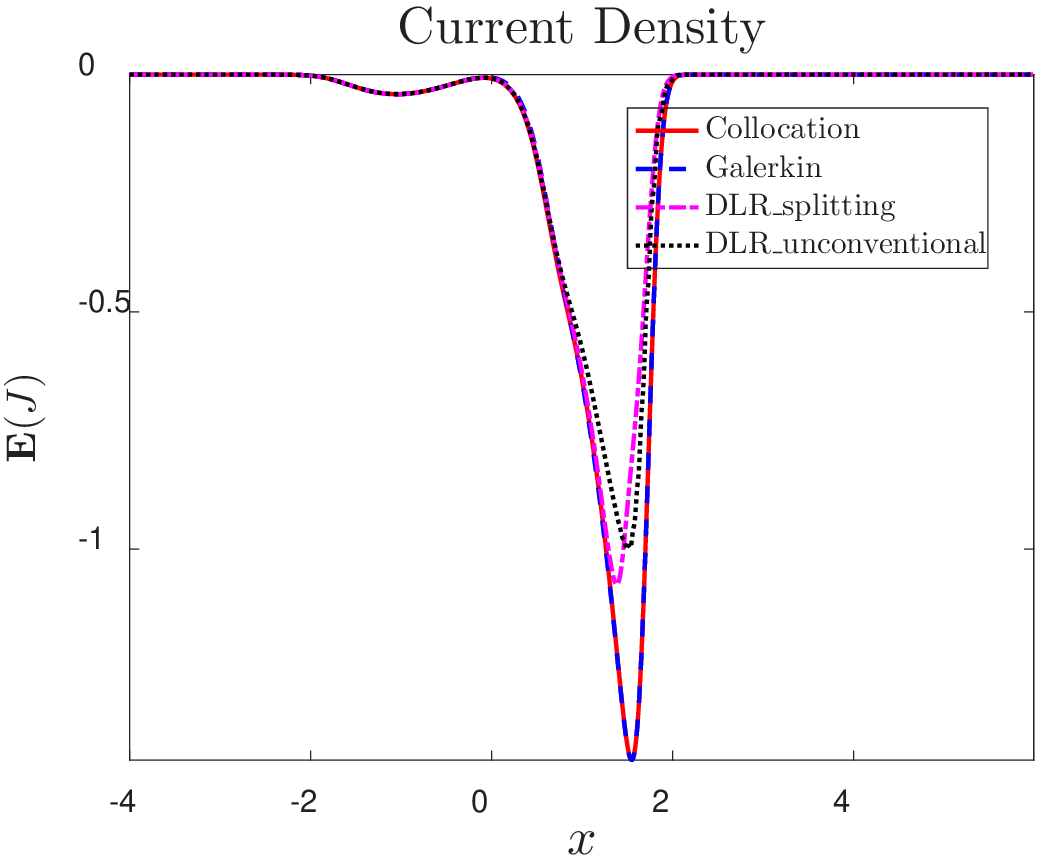}
        \caption{}
        \label{fig:subfig4_Jdensity_example3}
    \end{subfigure}
    \hfill
    \begin{subfigure}{0.32\textwidth}
        \centering
        \includegraphics[width=\textwidth]{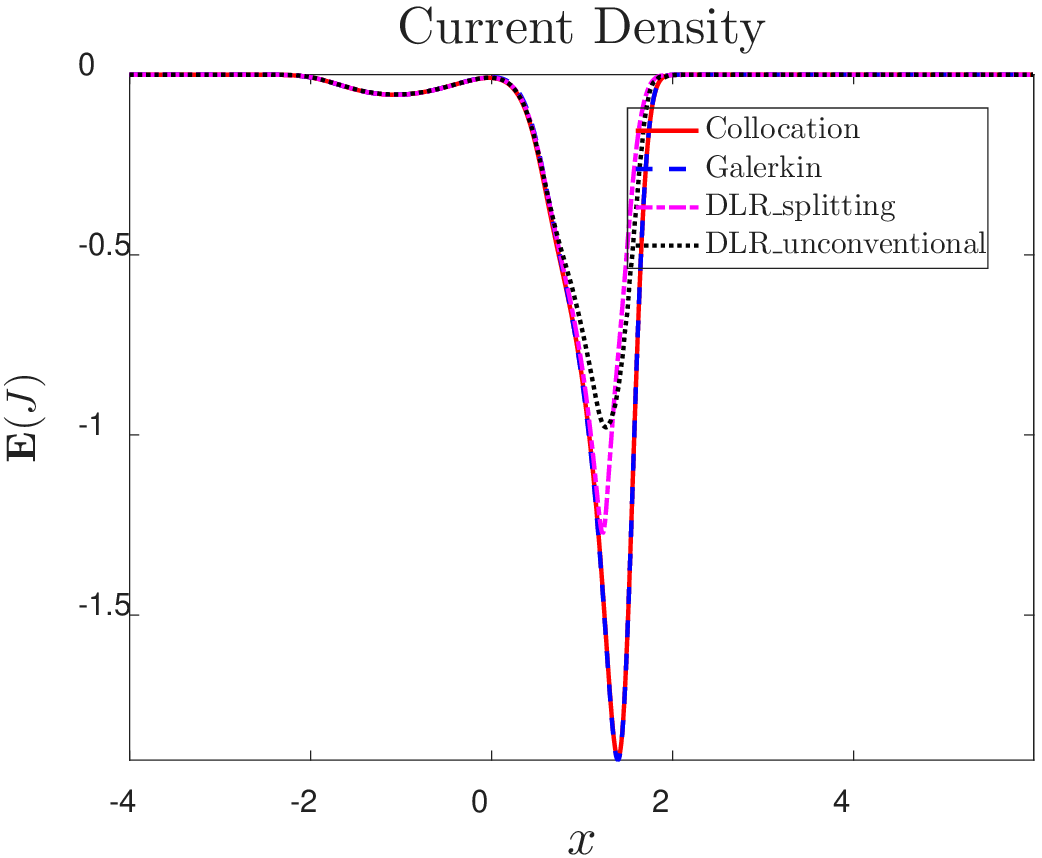}
        \caption{}
        \label{fig:subfig5_Jdensity_example3}
    \end{subfigure}
    \hfill
    \begin{subfigure}{0.32\textwidth}
        \centering
        \includegraphics[width=\textwidth]{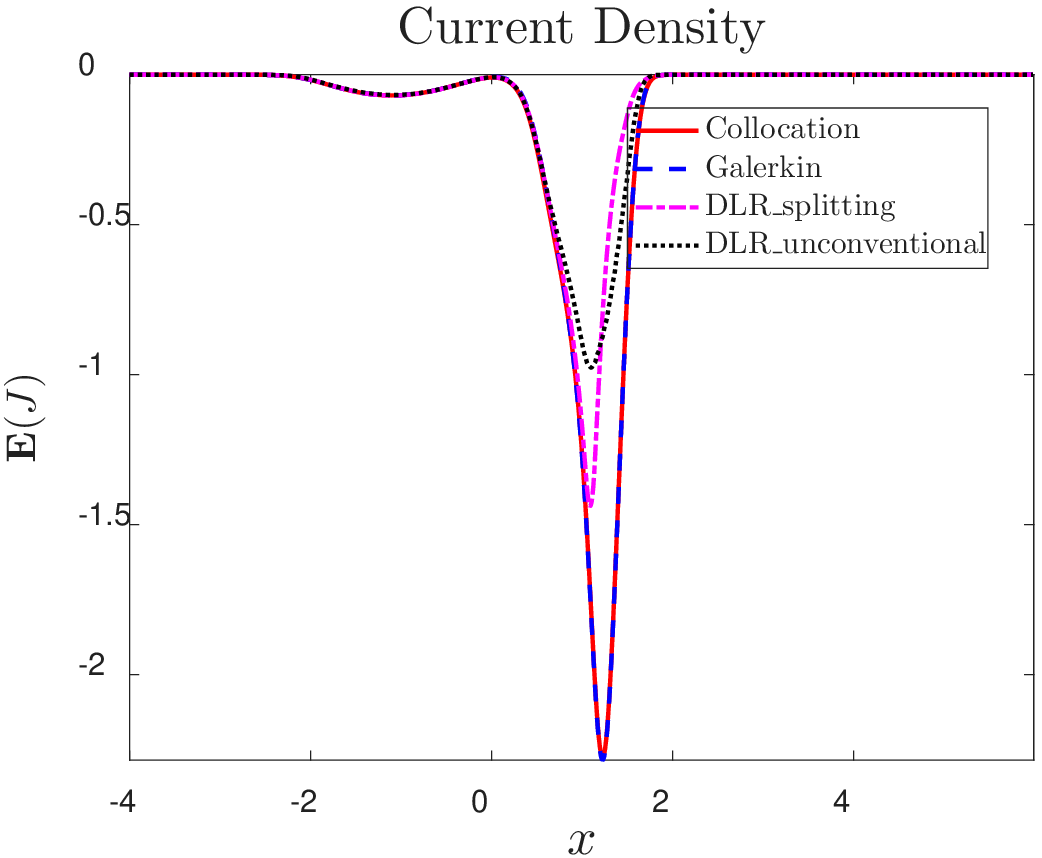}
        \caption{}
        \label{fig:subfig6_Jdensity_example3}
    \end{subfigure}

    \caption{The current density solution in example 3. From top left to bottom right, $\gamma=0.3, 0.5, 0.7, 0.9, 1.1, 1.3$.}
    \label{fig:Jdensity_example3}
\end{figure}

\begin{figure}[H]
    \centering
    \begin{subfigure}{0.32\textwidth}
        \centering
        \includegraphics[width=\textwidth]{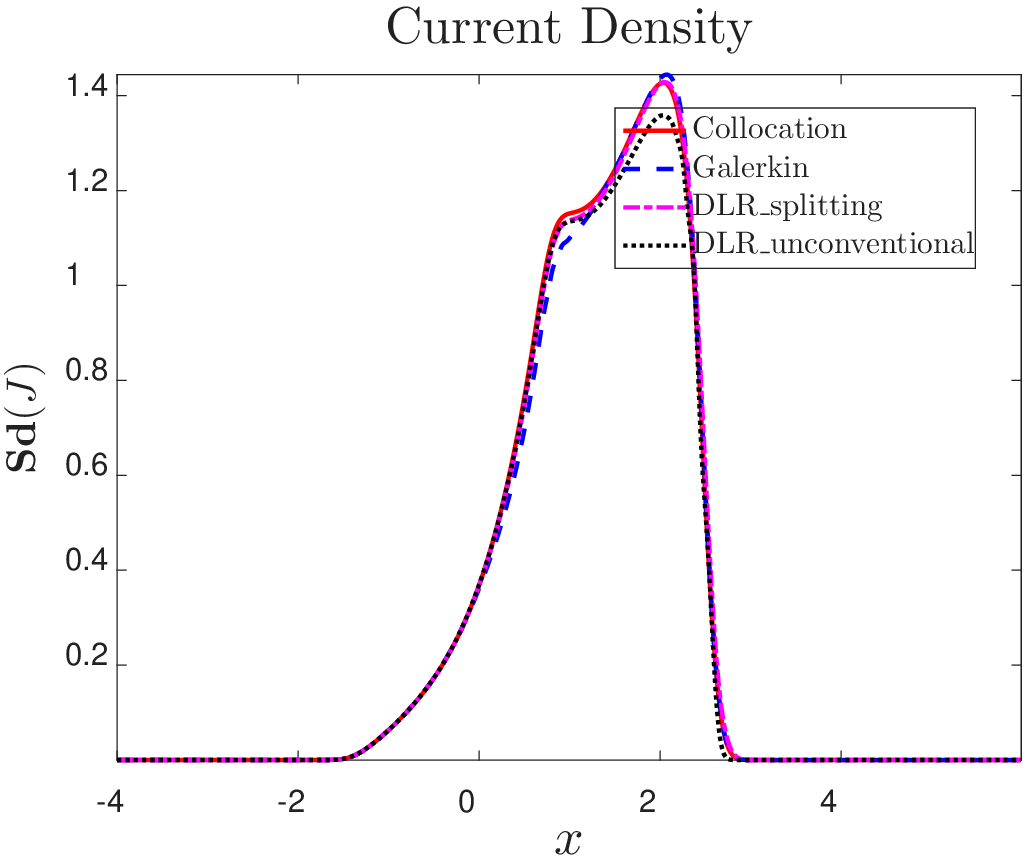}
        \caption{}
        \label{fig:subfig1_Jdensity_example3}
    \end{subfigure}
    \hfill
    \begin{subfigure}{0.32\textwidth}
        \centering
        \includegraphics[width=\textwidth]{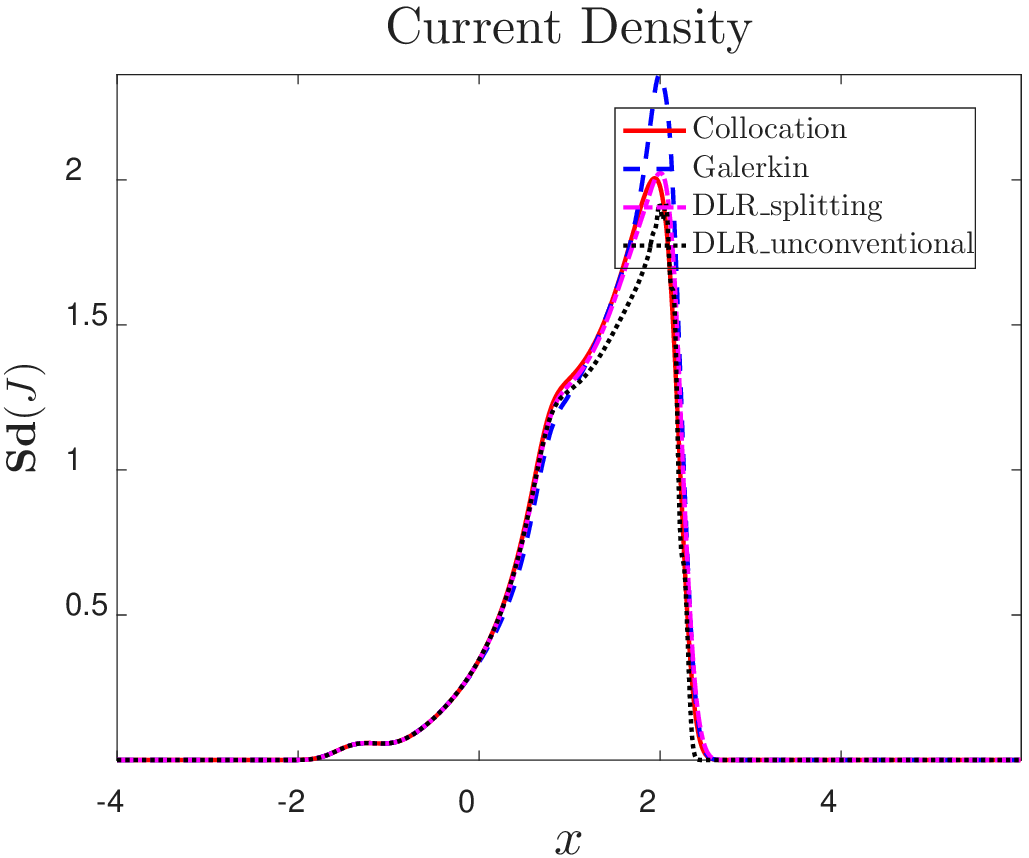}
        \caption{}
        \label{fig:subfig2_Jdensity_example3}
    \end{subfigure}
    \hfill
    \begin{subfigure}{0.32\textwidth}
        \centering
        \includegraphics[width=\textwidth]{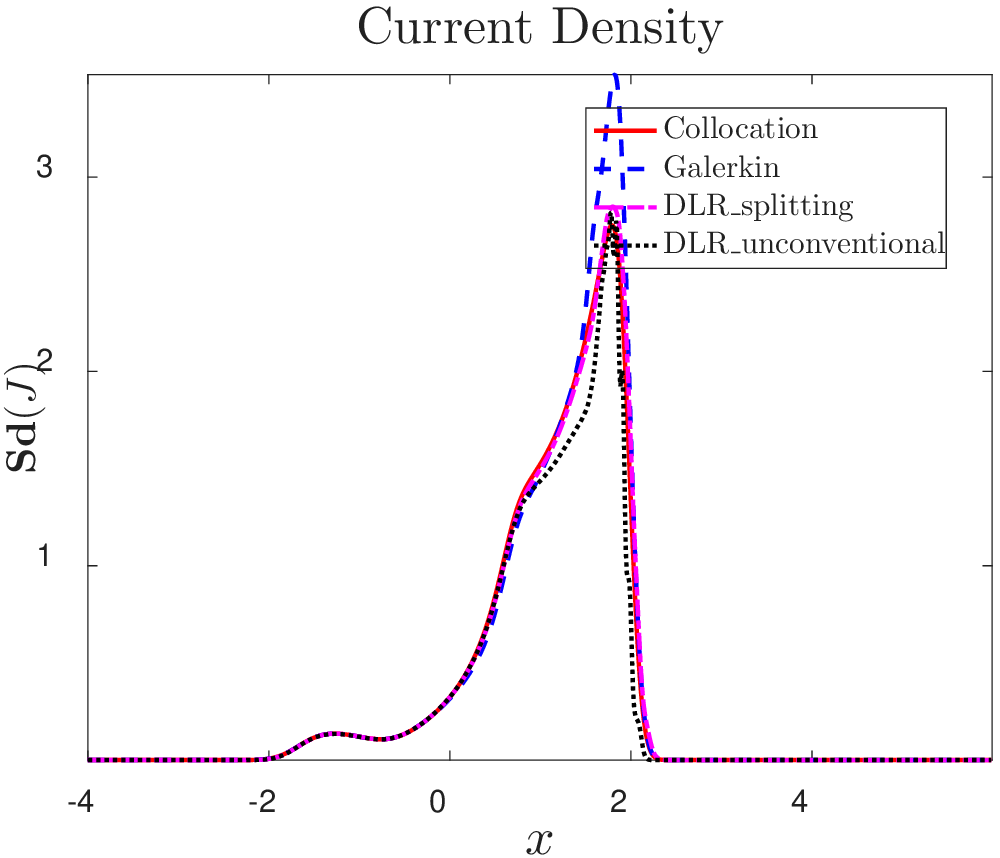}
        \caption{}
        \label{fig:subfig3_Jdensity_example3}
    \end{subfigure}

    \vspace{0.5cm} 
    \begin{subfigure}{0.32\textwidth}
        \centering
        \includegraphics[width=\textwidth]{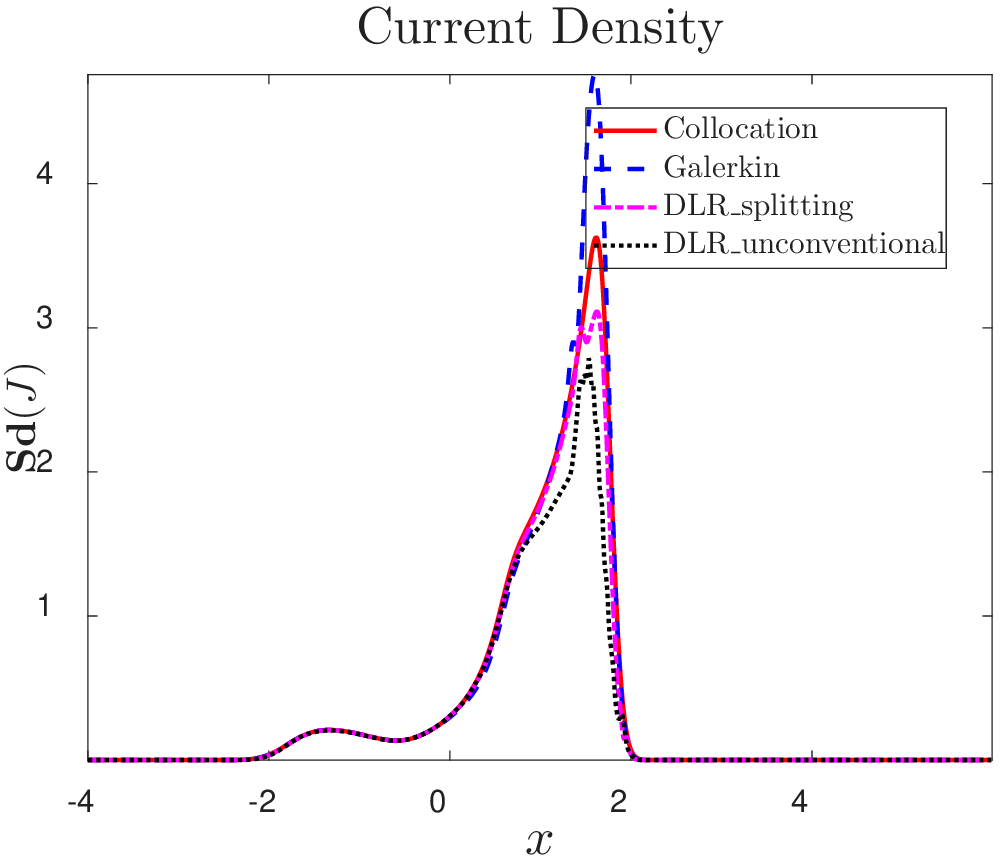}
        \caption{}
        \label{fig:subfig4_Jdensity_example3}
    \end{subfigure}
    \hfill
    \begin{subfigure}{0.32\textwidth}
        \centering
        \includegraphics[width=\textwidth]{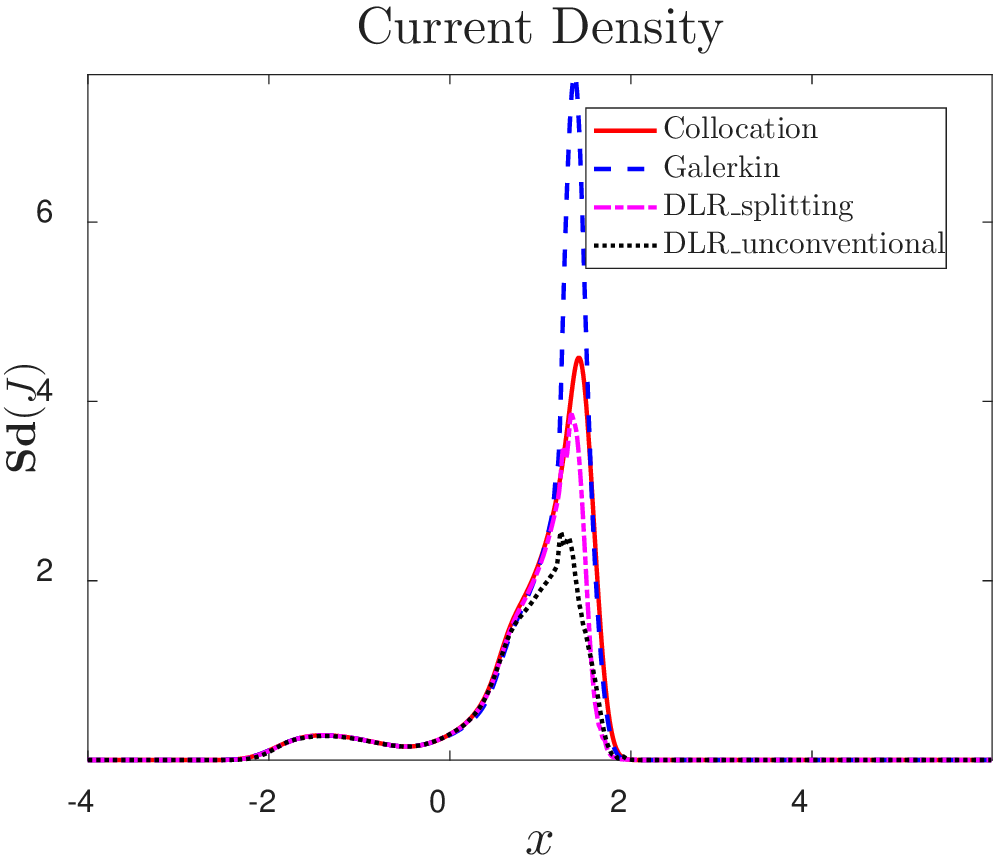}
        \caption{}
        \label{fig:subfig5_Jdensity_example3}
    \end{subfigure}
    \hfill
    \begin{subfigure}{0.32\textwidth}
        \centering
        \includegraphics[width=\textwidth]{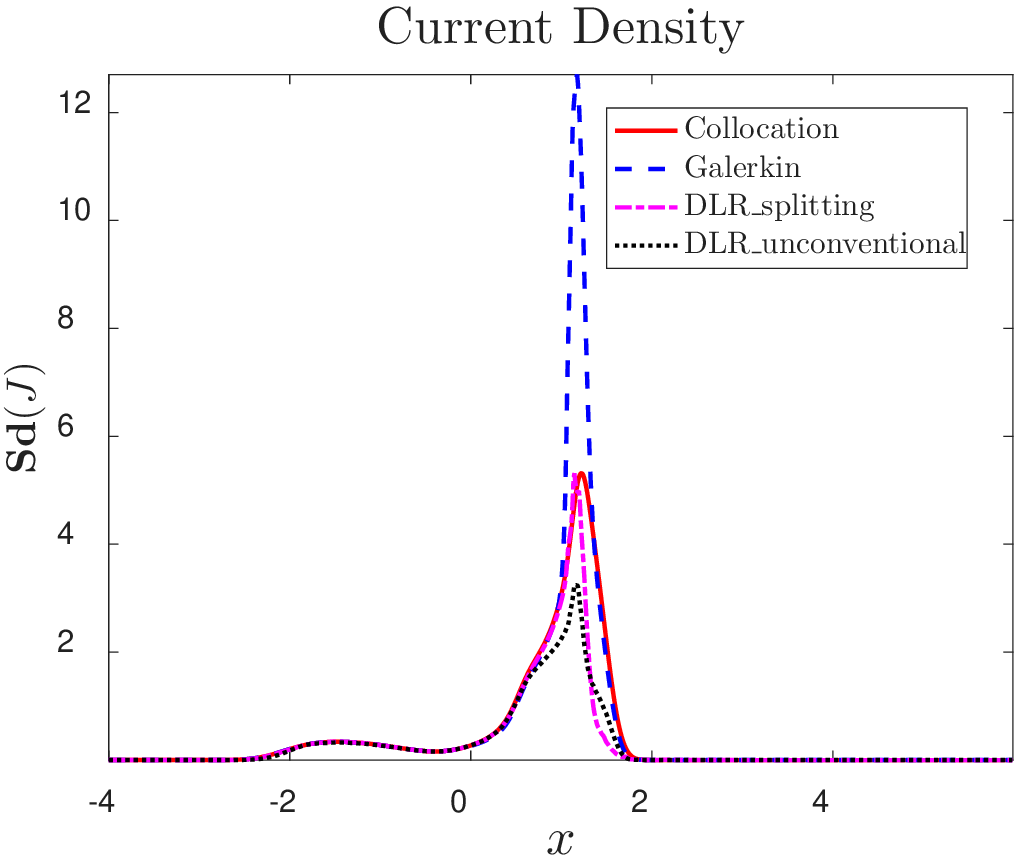}
        \caption{}
        \label{fig:subfig6_Jdensity_example3}
    \end{subfigure}

    \caption{The current density solution in example 3. From top left to bottom right, $\gamma=0.3, 0.5, 0.7, 0.9, 1.1, 1.3$.}
    \label{fig:Jdensity_example3_std}
\end{figure}

\textbf{Example 4 (Periodic Potential).} 
We assess the method using a periodic potential $V(x)=\sin(x)$ with random perturbation $V(\xi)=1+0.1(\xi_1+2\xi_2)$. The initial parameters are set to $\tilde{p}=0.7(\xi_1-\xi_2)$ and $\tilde{q}=1+0.8(\xi_1+2\xi_2)$. The domain and grid settings follow Example 1.

Table \ref{tab:error_comparison_example5} demonstrates that the method generalizes effectively to periodic systems, with the proposed scheme achieving an accuracy of $\mathcal{O}(10^{-3})$. The mean and standard deviation of the density and current density are displayed in Figure \ref{fig:density_example5}.




\begin{table}[!htbp]
\centering
\caption{Error comparison for Example 4.}
\label{tab:error_comparison_example5}
\renewcommand{\arraystretch}{1.2}
\scalebox{0.7}{
\begin{tabular}{l|c|c|c|c|c}
\toprule
\diagbox{Method}{Error} & $\psi$ & $\rho_{mean}$ & $\rho_{std}$ & $j_{mean}$ & $j_{std}$ \\
\midrule
Galerkin            & $1.206 \times 10^{-1}$ & $1.473 \times 10^{-2}$ & $3.688 \times 10^{-2}$ & $3.984 \times 10^{-2}$ & $6.910 \times 10^{-2}$ \\
DLR-splitting       & $4.720 \times 10^{-2}$ & $\mathbf{1.634 \times 10^{-3}}$ & $\mathbf{5.756 \times 10^{-3}}$ & $\mathbf{4.966 \times 10^{-3}}$ & $\mathbf{9.270 \times 10^{-3}}$ \\
Unconventional      & $\mathbf{4.329 \times 10^{-2}}$ & $1.916 \times 10^{-3}$ & $6.744 \times 10^{-3}$ & $5.652 \times 10^{-3}$ & $1.217 \times 10^{-2}$ \\
\bottomrule
\end{tabular}
}
\end{table}

\begin{figure}[H]
    \centering
    \begin{subfigure}{0.32\textwidth}
        \centering
        \includegraphics[width=\textwidth]{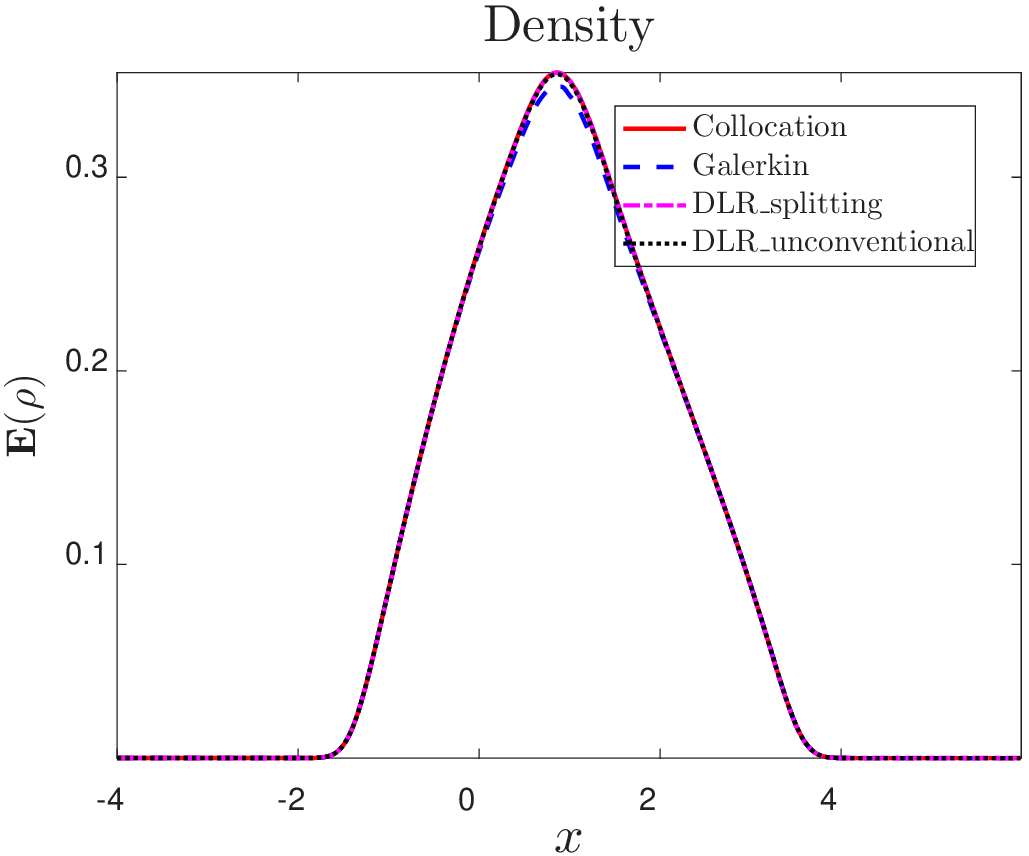}
        \caption{}
        \label{fig:subfig1_density_example1}
    \end{subfigure}
    \begin{subfigure}{0.32\textwidth}
        \centering
        \includegraphics[width=\textwidth]{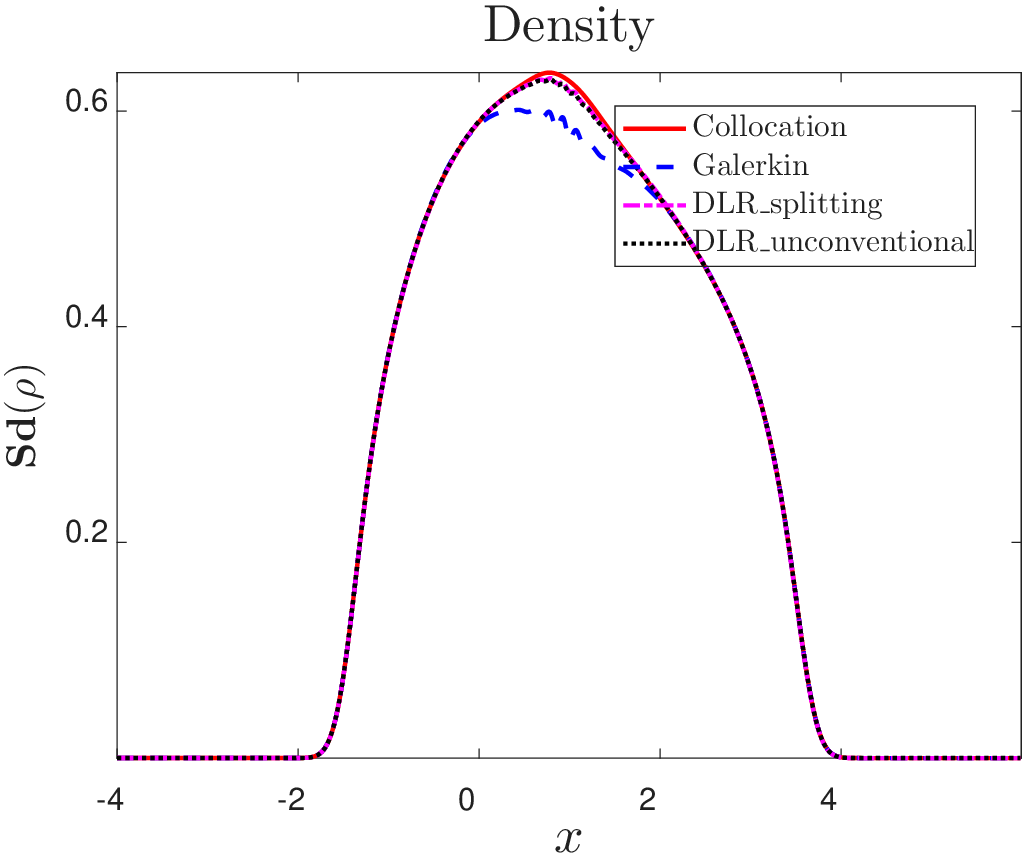}
        \caption{}
        \label{fig:subfig2_density_example1}
    \end{subfigure}
    \vspace{1em} 

    \begin{subfigure}{0.32\textwidth}
        \centering
        \includegraphics[width=\textwidth]{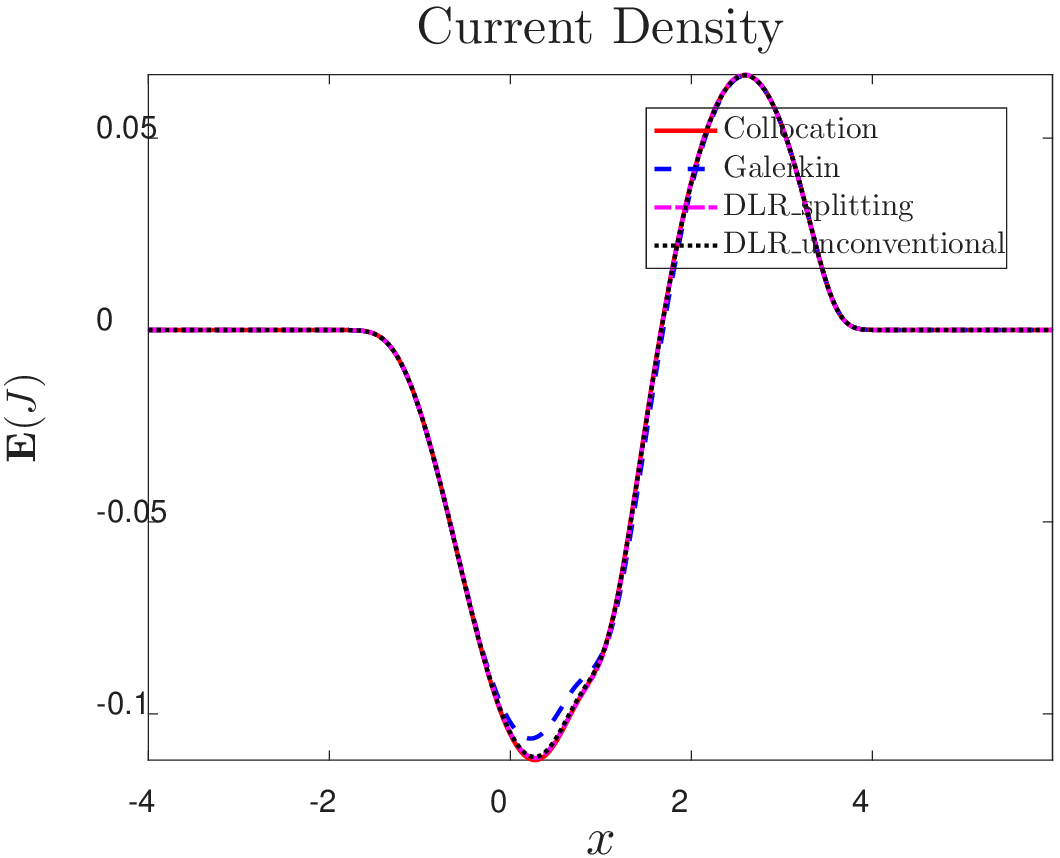}
        \caption{}
        \label{fig:subfig2_density_example1}
    \end{subfigure}
    \begin{subfigure}{0.32\textwidth}
        \centering
        \includegraphics[width=\textwidth]{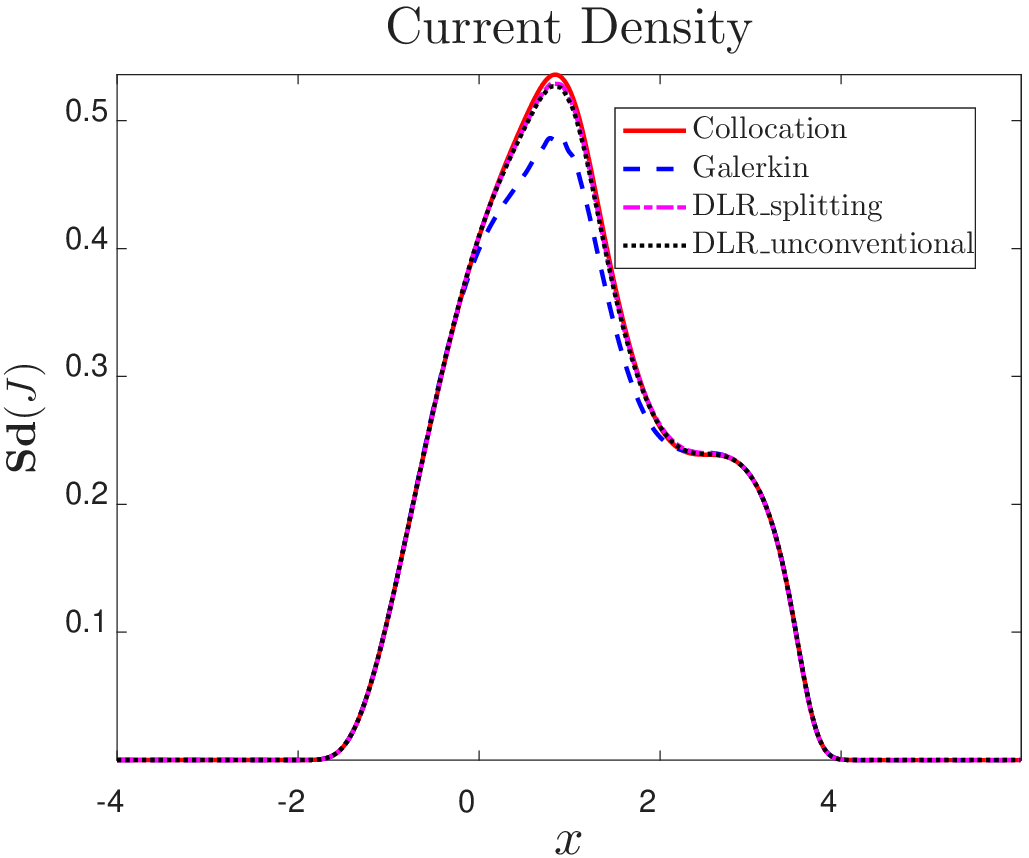}
        \caption{}
        \label{fig:subfig2_density_example1}
    \end{subfigure}
    \caption{Mean and std of $\rho$ and $j$ in Example $4$. The rank of the DLR method is $r=46$.}
    \label{fig:density_example5}
\end{figure}

\textbf{Example 5 (Confinement Regime).} 
Finally, we consider a confinement regime with a double-well potential $V(x)=(x^2-1)^2$ and $\varepsilon=0.01$. The random potential is $V(\xi)=1+0.1(\xi_1+2\xi_2)$. The initial condition parameter $\tilde{q}$ is modified to test a specific superposition state:
\begin{equation*}
    \tilde{p}=0.7(\xi_1-\xi_2), \quad \tilde{q}=1+0.01(\xi_1+2\xi_2).
\end{equation*}


Table \ref{tab:error_comparison_example6} highlights that while the standard Galerkin and DLR-splitting methods struggle with the complex tunneling dynamics (indicated by higher errors in $j$), the proposed unconventional integrator maintains robust accuracy ($e_{\psi} \approx 4.3 \times 10^{-3}$). Figure \ref{fig:density_example6} visualizes the results, confirming the method's ability to resolve the split-density structure inherent to double-well problems.



\begin{table}[!htbp]
\centering
\caption{Error analysis for Example 5.}
\label{tab:error_comparison_example6}
\renewcommand{\arraystretch}{1.2}
\scalebox{0.7}{
\begin{tabular}{l|c|c|c|c|c}
\toprule
\diagbox{Method}{Error} & $\psi$ & $\rho_{mean}$ & $j_{mean}$ & $\rho_{std}$ & $j_{std}$ \\
\midrule
Galerkin            & $1.067 \times 10^{-1}$ & $1.215 \times 10^{-2}$ & $1.444 \times 10^{-2}$ & $4.660 \times 10^{-2}$ & $7.299 \times 10^{-2}$ \\
DLR-splitting       & $1.115 \times 10^{-1}$ & $\mathbf{7.183 \times 10^{-4}}$ & $6.646 \times 10^{-2}$ & $6.654 \times 10^{-3}$ & $4.323 \times 10^{-2}$ \\
Unconventional      & $\mathbf{4.370 \times 10^{-3}}$ & $4.548 \times 10^{-3}$ & $\mathbf{7.044 \times 10^{-3}}$ & $\mathbf{6.119 \times 10^{-3}}$ & $\mathbf{5.884 \times 10^{-3}}$ \\
\bottomrule
\end{tabular}
}
\end{table}

\begin{figure}[H]
    \centering
    \begin{subfigure}{0.32\textwidth}
        \centering
        \includegraphics[width=\textwidth]{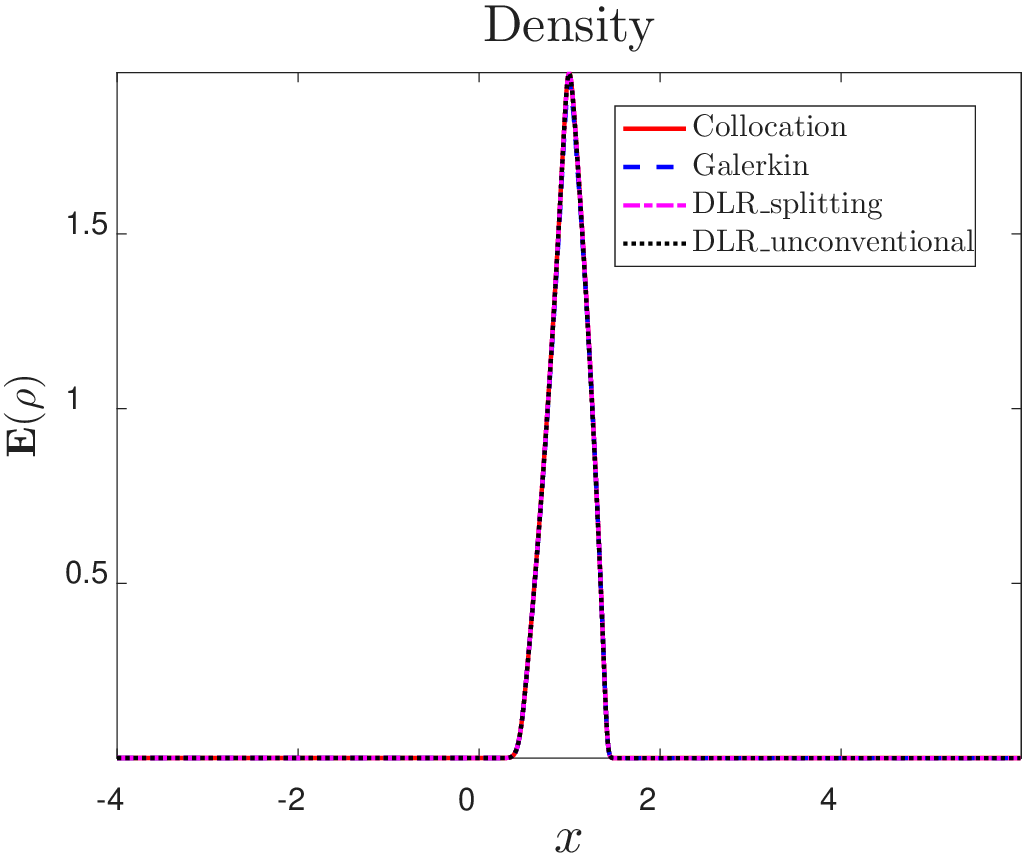}
        \caption{}
        \label{fig:subfig1_density_example1}
    \end{subfigure}
    \begin{subfigure}{0.32\textwidth}
        \centering
        \includegraphics[width=\textwidth]{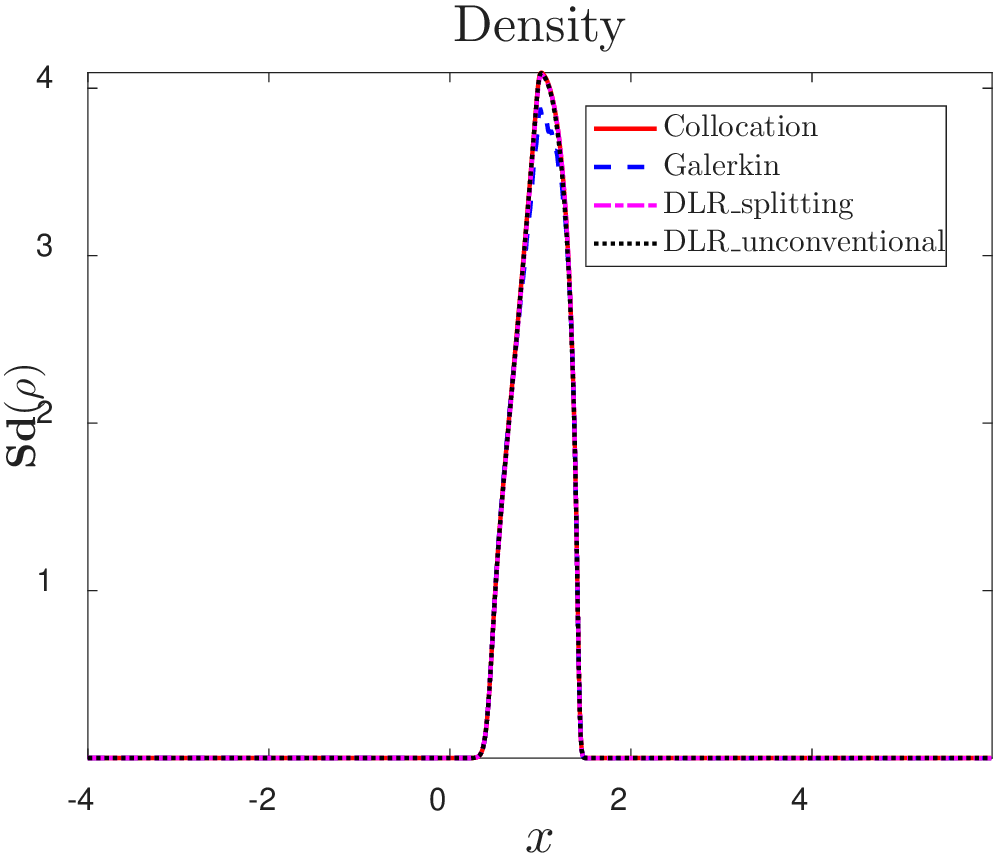}
        \caption{}
        \label{fig:subfig2_density_example1}
    \end{subfigure}
    \vspace{1em} 

    \begin{subfigure}{0.32\textwidth}
        \centering
        \includegraphics[width=\textwidth]{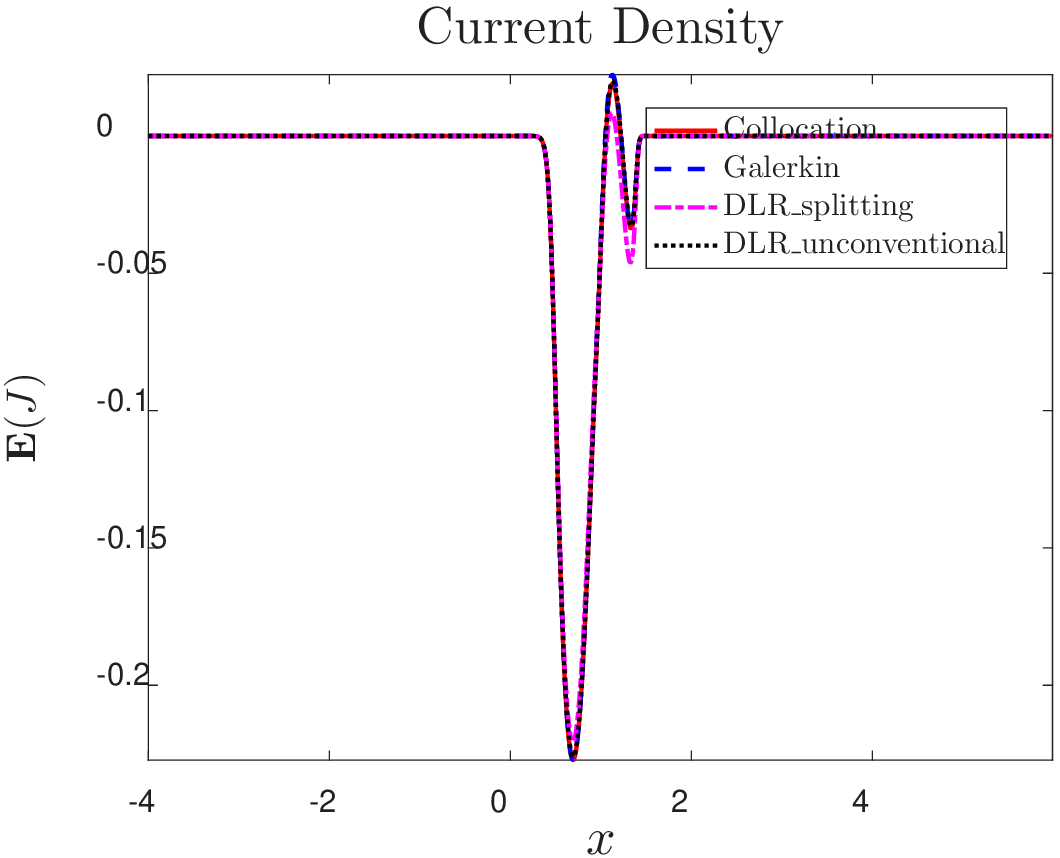}
        \caption{}
        \label{fig:subfig2_density_example1}
    \end{subfigure}
    \begin{subfigure}{0.32\textwidth}
        \centering
        \includegraphics[width=\textwidth]{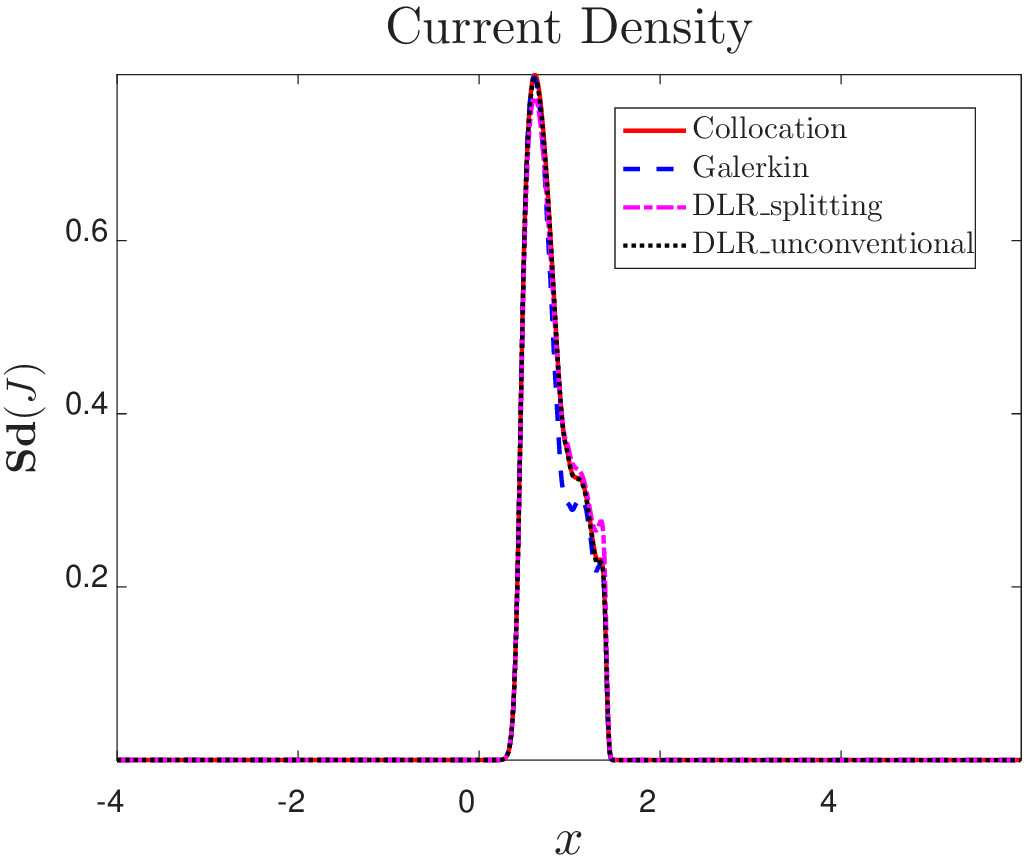}
        \caption{}
        \label{fig:subfig2_density_example1}
    \end{subfigure}
    \caption{Mean and std of $\rho$ and $j$ in Example $5$. The rank of the DLR method is $r=46$.}
    \label{fig:density_example6}
\end{figure}


\section{Conclusion}
\label{Sec: conclusion}

In this paper, we developed a dynamical low-rank (DLR) approximation framework for the semiclassical Schrödinger equation with uncertainties, successfully overcoming the computational bottlenecks posed by dual oscillations in spatial and random domains. By evolving the solution on a low-rank manifold via projector-splitting and unconventional integrators, we demonstrated that the wave function's oscillatory nature does not preclude an efficient low-rank representation. Most notably, the DLR method significantly outperforms traditional stochastic Galerkin methods by capturing complex quantum dynamics with a remarkably smaller number of basis functions. Our sensitivity analysis further reveals that while initial data dictates the baseline rank, the randomness of the potential primarily drives the subsequent rank evolution. These findings establish the DLR approach as a computationally lean and robust alternative for high-dimensional uncertainty quantification in the semiclassical regime, with future work focusing on rank-adaptive strategies for stronger potential perturbations.
\section*{Acknowledgments}
L. Xu thanks the Westlake University HPC Center for computational support. L.~Liu acknowledges the support by National Key R\&D Program of China (2021YFA1001200), Ministry of Science and Technology in China, General Research Fund (14303022, 14301423, 14307125) funded by Research Grants Council of Hong Kong.

\appendix
\section{Dynamical low-rank approximation for matrices}
\label{Sec: DLR_matrix}

In this section, we review the mathematical framework of the dynamical low-rank (DLR) approximation for matrices and describe the two primary integration schemes: the projector-splitting integrator and the unconventional integrator.

\subsection{Formulas}
Consider a matrix-valued ordinary differential equation for $A(t)\in\mathbb{R}^{m\times n}$:
\begin{equation}
    \frac{\mathrm{d}}{\mathrm{d}t}A(t)=F(A(t)),\quad A(0)=A_0,
\end{equation}
where $F:\mathbb{R}^{m\times n}\rightarrow\mathbb{R}^{m\times n}$ is a smooth function. When the dimensions $m$ and $n$ are large, direct numerical simulation using standard integrators, such as the fourth-order Runge-Kutta method, becomes computationally prohibitive. Such approaches require $O(mn)$ storage and $O(m^2n)$ operations per time step, making them inefficient for large-scale problems.

The dynamical low-rank approximation is a technique designed to reduce both storage requirements and computational complexity by leveraging the underlying low-rank structure of the solution \cite{kochDynamicalLowRankApproximation2007}. This approach is based on the observation that if the matrix $A(t)$ admits a low-rank approximation, i.e., $A(t) \approx Y(t) = U(t)S(t)V^T(t)$, where $U(t) \in \mathcal{V}_{m,r}$ and $V(t) \in \mathcal{V}_{n,r}$ are matrices with orthonormal columns and $S(t) \in \mathbb{R}^{r \times r}$ is a non-singular core matrix, then the storage cost is significantly reduced from $O(mn)$ to $O((m+n+r)r)$. 

Furthermore, the computational complexity per time step is also substantially lowered. While direct simulation of the full system scales as $O(m^2n)$, the DLR approximation typically scales linearly with respect to the larger dimension, i.e., $O((m+n)r^2)$ or similar, depending on the specific form of $F(Y)$ and the choice of integrator. The goal of the DLR approximation is to efficiently identify this low-rank structure and update its components at each time step. Since many scientific and engineering problems—such as those in image processing, quantum chemistry, and quantum physics—naturally exhibit low-rank dynamics, this method has gained broad applicability. We refer the readers to 
\cite{kochDynamicalLowRankApproximation2007,kochDynamicalTensorApproximation2010,lubichDynamicalApproximationHierarchical2013}
for detailed discussions on DLR for matrices and tensors with various integrator algorithms, and to \cite{einkemmerLowRankProjectorSplittingIntegrator2018,einkemmerLowRankAlgorithmWeakly2019,ostermannConvergenceLowRankLieTrotter2019,einkemmerLowrankProjectorsplittingIntegrator2020} for its application to partial differential equations.

Let $\mathcal{M}_r$ denote the manifold of matrices of fixed rank $r$, defined as 
\begin{align*}
    \mathcal{M}_r := \{Y = USV^T : U \in \mathcal{V}_{m,r}, V \in \mathcal{V}_{n,r}, S \in \mathbb{R}^{r \times r} \text{ is nonsingular}\}, 
\end{align*}
where $\mathcal{V}_{n,r}$ is the Stiefel manifold, with the superscript $T$ denoting the matrix transpose, consisting of matrices with orthonormal columns: 
\begin{align}
    \mathcal{V}_{n,r} := \{U \in \mathbb{R}^{n \times r} : U^T U = I_r\}. 
\end{align}
The tangent space of $\mathcal{M}_r$ at a point $Y(t)$ is denoted by $\mathcal{T}_{Y(t)}\mathcal{M}_r$. The dynamical low-rank approximation seeks an approximate solution $Y(t) \in \mathcal{M}_r$ by imposing the orthogonal projection of the residual onto the tangent space, leading to the Galerkin condition: 
\begin{align}\label{eq: Galerkin condition}
    \langle \dot{Y}-F(Y), Z \rangle = 0, \quad \forall Z \in \mathcal{T}_{Y(t)}\mathcal{M}_r, 
\end{align}
where $\dot{Y}=\frac{\mathrm{d}}{\mathrm{d}t}Y(t)$ is the tangent vector at $Y(t)$, $Z$ is a test matrix in the tangent space, and $\langle \cdot, \cdot \rangle$ denotes the Frobenius inner product, i.e., $\langle A, B \rangle = \operatorname{tr}(A^T B)$.

The following theorem provides the explicit evolution equations for the low-rank factors: 
\begin{theorem}
    The Galerkin condition \cref{eq: Galerkin condition} is equivalent to the system
    \begin{align}\label{eq: dynamical low-rank approximation}
        \dot{Y} = \dot{U} S V^T + U \dot{S} V^T + U S \dot{V}^T, 
    \end{align}
    where the factors $U, S$, and $V$ evolve according to 
    \begin{align}\label{eq: dynamical low-rank approximation 2}
        \dot{U} = P_{U}^{\perp} F(Y(t)) V S^{-1}, \quad \dot{S} = U^T F(Y(t)) V, \quad \dot{V} = P_{V}^{\perp} F(Y(t))^T U S^{-T}. 
    \end{align}
    Here, $P_{U}^{\perp} = I - UU^T$ and $P_{V}^{\perp} = I - VV^T$ are the orthogonal projectors onto the orthogonal complements of the ranges of $U$ and $V$, respectively.
\end{theorem}

Consequently, the approximation $Y(t)=U(t)S(t)V^T(t)$ can be obtained by evolving the low-rank factors $U, S$, and $V$ according to the system in \cref{eq: dynamical low-rank approximation 2}. However, while this system can be integrated using standard numerical schemes such as the fourth-order Runge-Kutta method, it becomes numerically unstable when the core matrix $S$ is near-singular. This difficulty, often referred to as numerical instability, typically arises in cases of over-approximation where the chosen rank $r$ exceeds the effective rank of the solution. To overcome this issue, Lubich et al. proposed two robust integration algorithms: the projector-splitting integrator \cite{lubichProjectorsplittingIntegratorDynamical2014} and the unconventional integrator \cite{cerutiUnconventionalRobustIntegrator2022}.

\subsection{Numerical Algorithms}
The core idea of the projector-splitting integrator is to reformulate the DLR equation \eqref{eq: dynamical low-rank approximation} as a projection onto the tangent space $\mathcal{T}_{Y}\mathcal{M}_r$:
\begin{align}\label{eq: projection form}
    \dot{Y}(t) = P(Y(t)) F(Y(t)),
\end{align}
where $P(Y)$ is the orthogonal projection operator onto the tangent space at $Y=USV^T$. As shown in \cite{lubichProjectorsplittingIntegratorDynamical2014}, this operator can be explicitly expressed as
\begin{align}\label{eq: projection operator}
    P(Y)Z = ZP_V - P_U ZP_V + P_U Z,
\end{align}
where $P_U = UU^T$ and $P_V = VV^T$ denote the orthogonal projectors onto the range of $U$ and $V$, respectively.
The projector-splitting scheme then decomposes this operator into three parts according to the three terms in \eqref{eq: projection operator}, leading to the following sequence of sub-problems for $t \in [t_n, t_{n+1}]$:
\begin{align}
    \dot{Y}_{I}   &= F(Y_{I}) P_{V_I}, \quad Y_I(t_n)=Y_n, \\
    \dot{Y}_{II}  &= -P_{U_{II}} F(Y_{II}) P_{V_{II}}, \quad Y_{II}(t_n)=Y_I(t_{n+1}), \\
    \dot{Y}_{III} &= P_{U_{III}} F(Y_{III}), \quad Y_{III}(t_n)=Y_{II}(t_{n+1}).
\end{align}
The following lemma characterizes the evolution of the low-rank factors within each substep:
\begin{lemma}
    The solutions to the three sub-problems above are given by:
    \begin{enumerate}
        \item \textbf{K-step}: $Y_I(t) = U_I(t) S_I(t) V_n^T$, where $V_I(t) = V_n$ remains constant and the product $K(t) = U_I(t) S_I(t)$ evolves according to $\dot{K}(t) = F(K(t) V_n^T) V_n$.
        \item \textbf{S-step}: $Y_{II}(t) = U_{n+1} S_{II}(t) V_n^T$, where $U_{II}(t) = U_{n+1}$ and $V_{II}(t) = V_n$ are constant, and the core matrix evolves as $\dot{S}_{II}(t) = -U_{n+1}^T F(U_{n+1} S_{II}(t) V_n^T) V_n$.
        \item \textbf{L-step}: $Y_{III}(t) = U_{n+1} S_{III}(t) V_{III}(t)^T$, where $U_{III}(t) = U_{n+1}$ remains constant and the product $L(t) = V_{III}(t) S_{III}(t)^T$ evolves as $\dot{L}(t) = F(U_{n+1} L(t)^T)^T U_{n+1}$.
    \end{enumerate}
\end{lemma}

The aforementioned lemma provides the foundation for the projector-splitting integrator, as summarized in Algorithm \ref{alg: projector splitting integrator}.

A key advantage of the projector-splitting integrator is its robustness against near-singular core matrices $S$ by avoiding explicit inversion. However, the scheme requires a backward time integration step (with $-\Delta t$), which can trigger numerical instabilities if the underlying problem possesses dissipative or irreversible physical properties. To circumvent this, Ceruti and Lubich \cite{cerutiUnconventionalRobustIntegrator2022} proposed the unconventional integrator. This approach updates the basis matrices $U$ and $V$ in parallel through independent $K$- and $L$-steps, subsequently determining the core matrix via an orthogonal projection. By eliminating the backward integration, the unconventional integrator ensures stability for a broader class of problems while enabling efficient parallelization. Detailed derivations are available in \cite{cerutiUnconventionalRobustIntegrator2022}, and the procedure is summarized in Algorithm \ref{alg: unconventional integrator}.

\begin{algorithm}[htbp]
    \caption{The projector-splitting integrator}
    \label{alg: projector splitting integrator}
    \begin{algorithmic}
        \STATE \textbf{Input}: Initial factors $U_0, S_0, V_0$, time step $\Delta t$, total steps $N$.
        \FOR {$n=0, 1, \dots, N-1$}
            \STATE 1. \textbf{K-step}: Integrate the matrix ODE for $K(t)$ from $t_n$ to $t_{n+1}$:
            \begin{equation*}
                \dot{K}(t) = F(K(t)V_n^T) V_n, \quad K(t_n) = U_n S_n.
            \end{equation*}
            Perform QR decomposition: $K(t_{n+1}) = U_{n+1} \hat{S}_{n+1}$.
            
            \STATE 2. \textbf{S-step}: Integrate the matrix ODE for $S(t)$ from $t_n$ to $t_{n+1}$ (backward step):
            \begin{equation*}
                \dot{S}(t) = -U_{n+1}^T F(U_{n+1} S(t) V_n^T) V_n, \quad S(t_n) = \hat{S}_{n+1}.
            \end{equation*}
            Set $\tilde{S}_n = S(t_{n+1})$.
            
            \STATE 3. \textbf{L-step}: Integrate the matrix ODE for $L(t)$ from $t_n$ to $t_{n+1}$:
            \begin{equation*}
                \dot{L}(t) = F(U_{n+1} L(t)^T)^T U_{n+1}, \quad L(t_n) = V_n \tilde{S}_n^T.
            \end{equation*}
            Perform QR decomposition: $L(t_{n+1}) = V_{n+1} S_{n+1}^T$.
        \ENDFOR
        \STATE \textbf{Output}: $Y_N = U_N S_N V_N^T$.
    \end{algorithmic}
\end{algorithm}



\begin{algorithm}[!h]
    \caption{The unconventional integrator}
    \label{alg: unconventional integrator}
    \begin{algorithmic}
        \STATE \textbf{Input}: Initial factors $U_0, S_0, V_0$, time step $\Delta t$, total steps $N$.
        \FOR {$n=0, 1, \dots, N-1$}
            \STATE 1. \textbf{Update bases in parallel}:
            \begin{itemize}
                \item \textbf{K-step}: Integrate the matrix ODE for $K(t)$ from $t_n$ to $t_{n+1}$:
                \begin{equation*}
                    \dot{K}(t) = F(K(t)V_n^T) V_n, \quad K(t_n) = U_n S_n.
                \end{equation*}
                Perform QR decomposition: $K(t_{n+1}) = U_{n+1} \hat{S}_{n+1}$, and compute $M = U_{n+1}^T U_n$.
                
                \item \textbf{L-step}: Integrate the matrix ODE for $L(t)$ from $t_n$ to $t_{n+1}$:
                \begin{equation*}
                    \dot{L}(t) = F(U_n L(t)^T)^T U_n, \quad L(t_n) = V_n S_n^T.
                \end{equation*}
                Perform QR decomposition: $L(t_{n+1}) = V_{n+1} \tilde{S}_{n+1}$, and compute $N = V_{n+1}^T V_n$.
            \end{itemize}
            
            \STATE 2. \textbf{S-step (Update core matrix)}:
            Integrate the matrix ODE for $S(t)$ from $t_n$ to $t_{n+1}$:
            \begin{equation*}
                \dot{S}(t) = -U_{n+1}^T F(U_{n+1} S(t) V_{n+1}^T) V_{n+1}, \quad S(t_n) = M S_n N^T.
            \end{equation*}
            Set $S_{n+1} = S(t_{n+1})$.
        \ENDFOR
        \STATE \textbf{Output}: $Y_N = U_N S_N V_N^T$.
    \end{algorithmic}
\end{algorithm}

\bibliographystyle{plain}
\bibliography{reference.bib}
\end{document}